\newif{\ifarxiv}
\arxivtrue

\documentclass[preprint,12pt]{elsarticle}
\pdfoutput=1 

\usepackage{latexsym,amsmath,amssymb,stmaryrd,graphicx}
\usepackage{color}
\usepackage[all]{xy}
\usepackage{url}
\usepackage[inline]{enumitem}
\usepackage[T1]{fontenc}

\newcommand\eqdef{\mathrel{\buildrel \text{def}\over=}}

\newcommand\pow{\mathbb{P}}
\newcommand\Pfin{\pow_{\mathrm{fin}}}

\newcommand{\real}{\mathbb{R}}
\newcommand{\creal}{\overline{\mathbb{R}}_+}
\newcommand{\Rp}{\mathbb{R}_+}

\newcommand\diff{\smallsetminus}
\DeclareMathOperator{\upc}{\uparrow\!}
\DeclareMathOperator{\dc}{\downarrow\!}
\newcommand\nat{\mathbb{N}}

\newcommand\limp{\mathrel{\Rightarrow}}

\newcommand\Lform{\mathcal L}

\newcommand\uuarrow{\rlap{$\uparrow$}\raise.5ex\hbox{$\uparrow$}}
\newcommand\ddarrow{\rlap{$\downarrow$}\raise.5ex\hbox{$\downarrow$}}

\newcommand\identity[1]{\mathrm{id}_{#1}}

\newcommand\Pred{\mathbb{P}}
\newcommand\Angel{{\mathtt{A}}}
\newcommand\Demon{{\mathtt{D}}}
\newcommand\Nature{{\mathtt{P}}}
\newcommand\AN{{\Angel\Nature}}
\newcommand\DN{{\Demon\Nature}}
\newcommand\ADN{{\Angel\Demon\Nature}}
\newcommand\one{\mathbf 1}

\newcommand\invto[1]{{\buildrel {#1} \over\Longrightarrow}}

\newcommand\Smyth{{\mathcal Q}}
\newcommand\Hoare{{\mathcal H}}

\newcommand\Vt{\mathsf{V}}
\newcommand\HV{\Hoare_\Vt}

\newcommand\SV{\Smyth_{\Vt}}

\newcommand\quasi{\mathrm{q}} 
\newcommand\QL{\Plotkin^\quasi}
\newcommand\QLV{\QL_\Vt}
\newcommand\QLc{\Plotkin^{\quasi, cvx}}
\newcommand\QLVc{\QLc_\Vt}

\newcommand\Plotkin{\mathcal P\ell} 
\newcommand\TEMleq{\sqsubseteq^{\text{TEM}}}

\newcommand\Plotkinn{\Plotkin_{\mathcal V}}
\newcommand\PV\Plotkinn 
\newcommand\Val{{\mathbf V}}

\newcommand\Open{{\mathcal O}}

\newcommand\Topcat{\mathbf{Top}}
\newcommand\Setcat{\mathbf{Set}}

\newcommand{\interior}[1]{\text{int} ({#1})} 

\newcommand\patch{{\text{\textsf{patch}}}}

\newcommand\dsup{\sup\nolimits^{\scriptstyle\uparrow}}
\newcommand\finf{\inf\nolimits^{\scriptstyle\downarrow}}
\newcommand\dcup{\bigcup\nolimits^{\scriptstyle\uparrow}}
\newcommand\fcap{\bigcap\nolimits^{\scriptstyle\downarrow}}

\newtheorem{theorem}{Theorem}[section]
\newtheorem{proposition}[theorem]{Proposition}

\newtheorem{corollary}[theorem]{Corollary}
\newtheorem{lemma}[theorem]{Lemma}
\newproof{proof}{Proof}

\newtheorem{fact}[theorem]{Fact}

\newtheorem{definition}[theorem]{Definition}

\newtheorem{remark}[theorem]{Remark}

\newcommand\ForAuthors[1]
 {\par\smallskip                     
  \begin{center}
   \fbox
   {\parbox{0.9\linewidth}
    {\raggedright--- #1}
   }
  \end{center}
  \par\smallskip                     %
 }

\journal{Topology and its Applications}

\begin{document}

\begin{frontmatter}



\title{Weak Distributive Laws between Monads of Continuous
  Valuations and of Non-Deterministic Choice}


\author{Jean Goubault-Larrecq}

\address{Universit\'e Paris-Saclay, CNRS, ENS Paris-Saclay,
  Laboratoire M\'ethodes
  Formelles, 91190, Gif-sur-Yvette, France.\\
  \texttt{jgl@lmf.cnrs.fr}
}

\begin{abstract}
  We show that there is weak distributive law of the Smyth hyperspace
  monad $\SV$ (resp., the Hoare hyperspace monad $\HV$, resp.\ the
  monad $\QLV$ of quasi-lenses, resp.\ the monad $\Plotkinn$ of
  lenses) over the continuous valuation monad $\Val$, as well as over
  the subprobability valuation monad $\Val_{\leq 1}$ and the
  probability valuation monad $\Val_1$, on the whole category
  $\Topcat$ of topological spaces (resp., on certain full
  subcategories such as the category of locally compact spaces or of
  stably compact spaces).  We show that the resulting weak composite
  monad is the author's monad of superlinear previsions (resp.,
  sublinear previsions, resp.\ forks), possibly subnormalized or
  normalized depending on whether we consider $\Val_{\leq 1}$ or
  $\Val_1$ instead of $\Val$.  As a special case, we obtain a weak
  distributive law of the monad $\QLV \cong \Plotkinn$ over the monad
  of (sub)probability Radon measures $\mathbf R_\bullet$ on the
  category of stably compact spaces, which specializes further to a
  weak distributive laws of the Vietoris monad over
  $\mathbf R_\bullet$.  The associated weak composite monad is the
  monad of (sub)normalized forks.
\end{abstract}


\end{frontmatter}

\noindent
\begin{minipage}{0.25\linewidth}
  \includegraphics[scale=0.2]{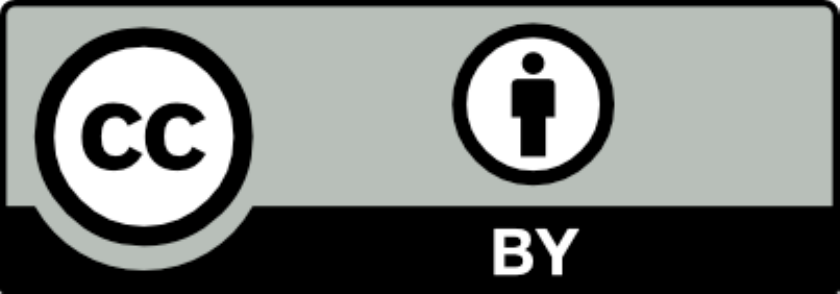}
\end{minipage}
\begin{minipage}{0.74\linewidth}
  \scriptsize
  For the purpose of Open Access, a CC-BY public copyright licence has
  been applied by the authors to the present document and will be
  applied to all subsequent versions up to the Author Accepted
  Manuscript arising from this submission.
\end{minipage}

\begin{flushright}
  \scriptsize
  Une montagne en mal d'enfant\\
  Jetait une clameur si haute\\
  Que chacun, au bruit accourant, \\
  Crut qu'elle accoucherait, sans faute, \\
  D'une cité plus grosse que Paris; \\
  Elle accoucha d'une souris. \\
  Quand je songe à cette fable, \\
  Dont le récit est menteur \\
  Et le sens est véritable, \\
  Je me figure un auteur \\
  Qui dit: Je chanterai la guerre \\
  Que firent les Titans au Maître du tonnerre. \guillemotright\\
  C'est promettre beaucoup: mais qu'en sort-il souvent? \\
  Du vent. \\
  \vskip0.5em
  \emph{La montagne qui accouche},
  Jean de la Fontaine, Fables, livre~V, fable~10, 1668.
\end{flushright}

\section{Introduction}
\label{sec:introduction}

Jon Beck introduced the notion of distributive laws
$\lambda \colon TS \to ST$ of a monad $S$ over a monad $T$, and showed
that any such distributive law allowed for the construction of a
composite monad $ST$ \cite{Beck:distr}.  The theory is much richer
than what the latter may suggest, and there is a one-to-one
correspondence between distributive laws as above, extensions of $T$
to the Kleisli category of $S$, and liftings of $S$ to the category of
$T$-algebras.

It was once dreamed that one could perhaps use this construction in
order to define new monads, combining various side-effects in the
semantics of programming languages.  But distributive laws are rare,
and, as a case in point, there is no distributive law of the monad of
non-deterministic choice over the monad of discrete probability
measure on $\Setcat$ \cite{VW:distr}.  The same argument shows that
there is no distributive law of any of the monads of non-deterministic
choice (the Smyth powerdomain monad $\Smyth$, the Hoare powerdomain
monad $\Hoare$, the Plotkin powerdomain monad $\Plotkin$) over the
monad $\Val$ of continuous valuations (or $\Val_1$ [probability
valuations], or $\Val_{\leq 1}$ [subprobability valuations],
implementing probabilistic choice).

Gabriella B\"ohm was the first to study weakenings of the notion of
distributive law between monads \cite{Bohm:weak:monads}, in a general
$2$-categorical context.  Given two monads $S$ and $T$, a distributive
law $\lambda \colon TS \to ST$ is a natural transformation satisfying
four equations expressing compatibility with the units and
multiplications of each of the two monads.  Weak distributive laws are
obtained by dropping one or the other of those equations.  Richard
Garner then observed that one of these kinds of weak distributive
laws, where compatibility with the unit of the $T$ monad is ignored,
could be applied to algebraic-effect type monads
\cite{Garner:weak:distr}.  This is the kind that we are interested in
here.

The theory of weak distributive law parallels Beck's theory of
distributive laws, and notably, weak distributive laws of $S$ over $T$
are in one-to-one correspondence with weak extensions of $T$ to the
Kleisli category of $S$, and weak liftings of $S$ to the category of
$T$-algebras.

It was suggested by Alexandre Goy \cite[Chapter~7]{Goy:PhD} that,
while there is no distributive law of the Vietoris monad (which sends
every space $X$ to its space of non-empty compact subsets with the
Vietoris topology) over the monad $\mathbf R_1$ of Radon probability
measures on the category $\mathbf{KHaus}$ of compact Hausdorff spaces,
there should be a weak distributive law relating the two.  We will
confirm that it exists, as a special case of a weak distributive law
between the Plotkin hyperspace monad $\Plotkinn$ over
$\Val_1 \cong \mathbf R_1$, in the final section of this paper
(Section~\ref{sec:radon-measures}).  (A toy version of this result was
established by Goy and Petri\c san \cite{goypetr-dp}, on the category
$\Setcat$, with the powerset monad instead of the Vietoris monad and
the finite distribution monad instead of $\mathbf R_1$.)

We will show that the resulting weak composite monad is the author's
monad of normalized forks, introduced thirteen years earlier
\cite{Gou-csl07} on the category of dcpos.  We take the opportunity to
make the relatively obvious generalization to topological spaces of
that monad here.

Before we reach that final section on the Vietoris and Radon monads,
we will explore essentially all possible combinations of monads of
non-deterministic choice on $\Topcat$ on the one hand, and of monads
of probabilistic choice on the other hand.  The latter include the
monads $\Val$ of continuous valuations, $\Val_{\leq 1}$ of
subprobability valuations, $\Val_1$ of probability valuations.
Continuous valuations are a close cousin to measures, and we will soon
explain how they relate to each other.  The former include the Smyth
hyperspace monad $\SV$ (demonic non-determinism) the Hoare hyperspace
monad $\HV$ (angelic non-determinism), and two variants of the
Plotkinn hyperspace monad (erratic non-determinism), the monad
$\Plotkinn$ of lenses and the monad $\QLV$ of quasi-lenses.  We show
that all of those have weak distributivity laws over $\Val_\bullet$,
where $\bullet$ is nothing, ``$\leq 1$'', or ``$1$''.  In the case of
$\SV$, this works on the whole category $\Topcat$ of topological
spaces.  For $\HV$, we have to restrict to certain full subcategories,
such as the subcategory of locally compact spaces.  For $\QLV$ and
$\Plotkinn$, we need to restrict even further, typically to stably
compact spaces---a category that still contains $\mathbf{KHaus}$ as a
proper subcategory.

In each case, we make the resulting weak composite monad explicit.
And, in each case, that is a well-known monad: the monad of
superlinear (resp., subnormalized or normalized, depending on
$\bullet$) previsions in the demonic case, the monad of sublinear
(resp., subnormalized or normalized) previsions in the angelic case,
and the monad of (resp., subnormalized or normalized) forks in the
erratic case.  Those monads first appeared in \cite{Gou-csl07} on the
category of dcpos, and we will rely on the (pretty obvious)
generalization to $\Topcat$ here.  ``Tout \c ca pour \c ca!'', as we
say in French; in a more elegant form---and written in La Fontaine's
admirable style---see the quote before this introduction.  At least
this study should confirm that Garner definitely hit upon a, perhaps
the, right way of combining monads.

\emph{Outline.}  Section~\ref{sec:preliminaries} sets some technical
preliminaries.  We will introduce other notions and results along the
way.  The rest of the paper is structured in three parts, for each
version of non-determinism (demonic, angelic, erratic) we consider.
Each part has exactly the same structure, except for the last one,
which has an extra, final Section~\ref{sec:radon-measures} on Radon
measures and the Vietoris monad.  Explicitly, that structure consists
of one section on the monad we will eventually discover as the weak
composite monad (superlinear previsions in
Section~\ref{sec:monad-superl-prev}, sublinear previsions in
Section~\ref{sec:monad-subl-prev}, forks in
Section~\ref{sec:monad-forks}), one on the weak distributive law
(Section~\ref{sec:weak-distr-law},
Section~\ref{sec:weak-distr-law:flat},
Section~\ref{sec:weak-distr-law:nat}), and one on the resulting weak
composite monad (Section~\ref{sec:associated-monad},
Section~\ref{sec:associated-monad:flat},
Section~\ref{sec:assoc-weak-comp}).


\section{Preliminaries}
\label{sec:preliminaries}

For background on topology, we refer the reader to
\cite{JGL-topology}.  We write $\interior A$ for the interior of $A$,
$cl (A)$ (or $cl_X (A)$) for the closure of $A$ (in a space $X$), and
$\Open X$ for the lattice of open subsets of $X$.  The specialization
preordering $\leq$ of a topological space $X$ is defined on points
$x, y \in X$ by $x \leq y$ if and only if every open neighborhood of
$x$ contains $y$, if and only if $x$ lies in the closure of $\{y\}$.
We also write $x \in A \mapsto f (x)$ for the function that maps every
element $x$ of a set $A$ to the value $f (x)$; logicians and computer
scientists would use the notation $\lambda x \in A . f (x)$, but this
would interfere with the standard use of the letter $\lambda$ for
(weak) distributive laws.

A \emph{directed} family of elements in a poset is a non-empty family
of elements $D$ such that any two elements of $D$ have an upper bound
in $D$.  A \emph{filtered} family is defined similarly, with lower
bounds replacing upper bounds.  We use $\dsup$ for suprema of directed
families, $\dcup$ for unions of directed families of sets, $\finf$ for
filtered infima, $\fcap$ for intersections of filtered families of
sets.

An \emph{irreducible} closed subset $C$ of $X$ is a non-empty closed
subset such that, for any two closed subsets $C_1$ and $C_2$ of $X$
such that $C \subseteq C_1 \cup C_2$, $C$ is included in $C_1$ or in
$C_2$ already; equivalently, if $C$ intersects two open sets, it must
intersect their intersection.  A space $X$ is \emph{sober} if and only
if it is $T_0$ and every irreducible closed subset is of the form
$\dc x$ for some point $x \in X$.  Every Hausdorff space, for example,
is sober.  The notation $\dc x$ stands for the \emph{downward closure}
of $x$ in $X$, namely the set of points $y$ below $x$.  Symmetrically,
$\upc x$ stands for the \emph{upward closure} of $x$, namely the set
of points $y$ above $x$.  This notation extends to $\upc A$, for any
subset $A$, denoting $\bigcup_{x \in A} \upc x$.

A \emph{compact} subset $A$ of a space $X$ is one such that one can
extract a finite subcover from any of its open covers.  No separation
property is assumed.  A subset $A$ of $X$ \emph{saturated} if and only
if it is equal to the intersection of its open neighborhoods, or
equivalently if and only if it is upwards-closed in the specialization
preordering of $X$.

A space $X$ is \emph{locally compact} if and only if every point has a
base of compact neighborhoods, or equivalently of compact saturated
neighborhoods, since for any compact subset $K$ of $X$, the upward
closure $\upc K$ of $K$ with respect to the specialization preordering
of $X$ is compact saturated.  Please beware that, in non-Hausdorff
spaces, a compact space may fail to be locally compact.

A space is \emph{coherent} if and only if the intersection of any two
compact saturated subsets is compact (and necessarily saturated).
That, too, is a property that may fail in non-Hausdorff spaces.

A \emph{stably locally compact} space is a coherent, locally compact,
sober space, and a \emph{stably compact} is a compact, stably locally
compact space.  We note that every compact Hausdorff space is stably
compact, as well as any coherent continuous dcpo with a least element.
(We will not define dcpos; see \cite{JGL-topology}, \cite{AJ:domains},
or \cite{GHKLMS:contlatt}.)

Let $\creal$ be the set of extended non-negative real numbers
$\Rp \cup \{\infty\}$, with its usual ordering.  When needed, we will
consider it with its Scott topology, whose open sets are the intervals
$]t, \infty]$, $t \in \Rp$, plus $\emptyset$ and $\creal$ itself.  In
general, the \emph{Scott topology} on a poset has as open sets the
upward closed subsets $U$ (i.e., $U = \upc U$) such that every
directed family $D$ whose supremum exists and is in $U$ intersects
$U$.

A \emph{continuous valuation} on a space $X$ is a map
$\nu \colon \Open X \to \creal$ that is \emph{strict}
($\nu (\emptyset)=0$), 
\emph{modular} (for all $U, V \in \Open X$,
$\nu (U) + \nu (V) = \nu (U \cup V) + \nu (U \cap V)$) and
Scott-continuous.  
We say that $\nu$ is a \emph{probability} valuation if and only if
$\nu (X)=1$, and a \emph{subprobability} valuation if and only if
$\nu (X) \leq 1$.

Let $\Val X$ denote the space of continuous valuations on a space $X$,
with the following \emph{weak topology}.  That is defined by subbasic
open sets $[U > r] \eqdef \{\nu \in \Val X \mid \nu (U) > r\}$, where
$U \in \Open X$ and $r \in \Rp$.  We define its subspace $\Val_1 X$ of
probability valuations and $\Val_{\leq 1} X$ (subprobability)
similarly.  In general, we will write $\Val_\bullet X$, where
$\bullet$ is nothing, ``$\leq 1$'', or ``$1$''.

The specialization ordering of each is the \emph{stochastic ordering}
$\leq$ given by $\nu \leq \nu'$ if and only if $\nu (U) \leq \nu' (U)$
for every $U \in \Open X$; indeed, $\nu \leq \nu'$ if and only if for
every $U \in \Open X$, for every $r \in \Rp$, $\nu \in [U > r]$
implies $\nu' \in [U > r]$.

Continuous valuations are an alternative to measures that have become
popular in domain theory \cite{jones89,Jones:proba}.  The first
results that connected continuous valuations and measures are due to
Saheb-Djahromi \cite{saheb-djahromi:meas} and Lawson
\cite{Lawson:valuation}.  The current state of the art on similar
results is the following.  In one direction, every measure on the
Borel $\sigma$-algebra of $X$ induces a continuous valuation on $X$ by
restriction to the open sets, if $X$ is hereditarily Lindel\"of
(namely, if every directed family of open sets contains a cofinal
monotone sequence).  This is an easy observation, and one half of
Adamski's theorem \cite[Theorem~3.1]{Adamski:measures}, which states
that a space is hereditary Lindel\"of if and only if every measure on
its Borel $\sigma$-algebra restricts to a continuous valuation on its
open sets.  In the other direction, every continuous valuation on a
space $X$ extends to a measure on the Borel sets provided that $X$ is
an LCS-complete space \cite[Theorem~1]{dBGLJL:LCScomplete}.  An
\emph{LCS-complete} space is a space homeomorphic to a $G_\delta$
subset of a locally compact sober space---$G_\delta$ meaning a
countable intersection of open sets.

There is a monad $(\Val_\bullet, \eta, \mu)$ on $\Topcat$.  For every
continuous map $f \colon X \to Y$, $\Val_\bullet f$ sends every
$\nu \in \Val_\bullet X$ to the \emph{pushforward} continuous
valuation $f [\nu]$, defined by $f [\nu] (V) \eqdef \nu (f^{-1} (V))$
for every $V \in \Open Y$.  The unit $\eta$ is defined by
$\eta_X (x) \eqdef \delta_x$ for every $x \in X$, for every space $X$;
$\delta_x$ is the \emph{Dirac valuation} at $x$, and maps every open
subset $U$ of $X$ to $1$ if $x \in U$, to $0$ otherwise.  For every
continuous map $f \colon X \to \Val_\bullet Y$, its \emph{extension}
$f^\dagger \colon \Val_\bullet X \to \Val_\bullet Y$ is defined by
$f^\dagger (\nu) (V) \eqdef \int_{x \in X} f (x) (V) \,d\nu$ for every
$\nu \in \Val_\bullet X$ and for every $V \in \Open Y$, and then
$\mu_X \eqdef \identity {\Val_\bullet X}^\dagger$.  Integration is
best defined by a Choquet formula \cite[Section~4]{Tix:bewertung}.
Explicitly, let $\Lform X$ be the set of all lower semicontinuous
functions $h$ from $X$ to $\creal$, that is, the set of all continuous
maps from $X$ to $\creal$ equipped with its Scott topology.  For every
$h \in \Lform X$, $\int_{x \in X} h (x) \,d\nu$ is defined as the
indefinite Riemann integral
$\int_0^\infty \nu (h^{-1} (]t, \infty])) \,dt$.  The following
\emph{change of variable} formula follows immediately: for every
continuous map $f \colon X \to Y$, for every $h \in \Lform Y$, for
every $\nu \in \Val_\bullet X$,
$\int_{y \in Y} h (y) \,df[\nu] = \int_{x \in X} h (f (x)) \,d\nu$.
We also note that, for every $h \in \Lform X$, for every $\xi \in
\Val_\bullet {\Val_\bullet X}$,
\begin{align*}
  \int_{x \in X} h (x) \,d\mu_X (\xi)
  & = \int_{\nu \in \Val_\bullet X}
    \left(\int_{x \in X} h (x) \,d\nu\right)\,d\xi
\end{align*}
See \cite[Hilfssatz~6.1]{kirch93} for a proof.

The Choquet formula makes sense when $\nu$ is not just a continuous
valuation, but any monotone map from $\Open X$ to $\creal$, and we
will use that observation later on.

It will be useful to note that $\eta_X^{-1} ([U > r]) = U$ if $r < 1$
and is empty otherwise, that
${(\Val_\bullet f)}^{-1} ([V > r]) = [f^{-1} (V) > r]$.  The formula
for ${(\mu_X)}^{-1} ([U > r])$ is more involved.  For every
$h \in \Lform X$, for every $r \in \Rp$, let
$[h > r] \eqdef \{\nu \in \Val_\bullet X \mid \int_{x \in X} h (x)
\,d\nu > r\}$.  This is also an open subset of $\Val_\bullet X$, and
such sets form an alternate base of the weak topology, when $h$ varies
over $\Lform X$ and $r$ over $\Rp$, see
\cite[Theorem~3.3]{Jung:scs:prob} where this was proved for spaces of
probability and subprobability valuations; the proof is similar for
arbitrary continuous valuations.  We may even restrict $r$ to be
non-zero, since $[h > 0] = \bigcup_{r > 0} [h > r]$, and in fact to be
equal to $1$ since $[h > r] = [h/r > 1]$ when $r > 0$.  The point is
that
${(\mu_X)}^{-1} ([U > r]) = [(\nu \in \Val_\bullet X \mapsto \nu (U))
> r]$, and more generally:
\begin{align}
  \label{eq:mu-1}
  {(\mu_X)}^{-1} ([h > r])
  & = \left[\left(\nu \in \Val_\bullet X \mapsto \int_{x \in X} h (x) \,d\nu\right) > r\right].
\end{align}
We will also need to note that
${(\Val f)}^{-1} ([h > r]) = [h \circ f > r]$, a consequence of the
change of variable formula, and that
$\eta_X^{-1} ([h > r]) = h^{-1} (]r, \infty])$.

\part{A weak distributive law between $\Val_\bullet$ and $\SV$}

For every topological space $X$, let $\Smyth X$ be the set of
non-empty compact saturated subsets of $X$.  The \emph{upper Vietoris}
topology on that set has basic open subsets $\Box U$ consisting of
those non-empty compact saturated subsets of $X$ that are included in
$U$, where $U$ ranges over the open subsets of $X$.  We write $\SV X$
for the resulting topological space.  Its specialization ordering is
reverse inclusion $\supseteq$.

The $\SV$ constructions was studied by a number of people, starting
with Mike Smyth \cite{Smyth:power:pred}, and later by Andrea Schalk
\cite[Section~7]{schalk:diss} who studied not only this, but also the
variant with the Scott topology, and a localic counterpart.  See also
\cite[Sections~6.2.2, 6.2.3]{AJ:domains} or
\cite[Section~IV-8]{GHKLMS:contlatt}, where the accent is rather on
the Scott topology of $\supseteq$.  In this setting, $\SV$ is a model
of \emph{demonic} non-determinism.

There is a monad $(\SV, \eta^\Smyth, \mu^\Smyth)$ on $\Topcat$.  For
every continuous map $f \colon X \to Y$, $\SV f$ maps every
$Q \in \SV X$ to $\upc f [Q]$, where $f [Q]$ denotes
$\{f (x) \mid x \in Q\}$.  The unit $\eta^\Smyth$ is defined by
$\eta^\Smyth_X (x) \eqdef \upc x$ for every $x \in X$.  For every
continuous map $f \colon X \to \SV Y$, there is an extension
$f^\sharp \colon \SV X \to \SV Y$, defined by
$f^\sharp (Q) \eqdef \bigcup_{x \in Q} f (x)$.  The multiplication
$\mu^\Smyth$ is defined by
$\mu^\Smyth_X \eqdef \identity {\SV X}^\sharp$, namely
$\mu^\Smyth_X (\mathcal Q) \eqdef \bigcup \mathcal Q$
\cite[Proposition~7.21]{schalk:diss}.

We note the equalities ${\eta^\Smyth}^{-1} (\Box U) = U$,
${(\SV f)}^{-1} (\Box V) = \Box {f^{-1} (V)}$, and
${\mu^\Smyth}^{-1} (\Box U) = \Box {\Box U}$.

\section{The monad of superlinear previsions}
\label{sec:monad-superl-prev}

We will show that there is a weak distributive law of $\SV$ of
$\Val_\bullet$, and we will build the resulting weak composite monad.
That monad has been well-known for some time, and is the monad of
superlinear previsions of \cite{Gou-csl07}, except that this paper
studied it on the category of dcpos and Scott-continuous maps instead
of $\Topcat$; the study of superlinear previsions on $\Topcat$ was
initiated in \cite{JGL-mscs16}, but it was not stated that they formed
a monad there, although, as we will see below, this is pretty
elementary.

At any rate, we will need to learn quite a number of facts about
superlinear previsions in our endeavor to build the desired weak
distributive law and its resulting weak composite monad.

We recall that $\Lform X$ is the space of lower semicontinuous maps
from $X$ to $\creal$.  We equip it with the Scott topology of the
pointwise ordering.  A \emph{prevision} on $X$ is a Scott-continuous
functional $F \colon \Lform X \to \creal$ such that $F (ah)=aF(h)$ for
all $a \in \Rp$ and $h \in \Lform X$.  It is \emph{linear} (resp.,
\emph{superlinear}) if and only if $F (h+h')=F (h)+F (h')$ (resp.,
$\geq$) for all $h, h' \in \Lform X$.  It is \emph{subnormalized}
(resp., \emph{normalized}) iff $F (\one+h) \leq \one+F (h)$ (resp.,
$=$) for every $h \in \Lform X$, where $\one$ is the constant function
with value $1$.  We write $\Pred_\DN X$ for the space of superlinear
previsions on $X$, $\Pred_\DN^{\leq 1} X$ for its subspace of
subnormalized superlinear previsions, and $\Pred_\DN^1 X$ for its
subspace of normalized superlinear previsions; synthetically, we use
the notation $\Pred_\DN^\bullet X$.  Similarly, we denote by
$\Pred_\Nature^\bullet X$ the corresponding spaces of linear
previsions.  The topology on each of those spaces is generated by
subbasic open sets
$[h > r] \eqdef \{F \in \Pred_\DN^\bullet X \text{ (resp.,
  $\Pred_\Nature^\bullet X$)} \mid F (h) > r\}$, $h \in \Lform X$,
$r \in \Rp$.  Its specialization ordering is the pointwise ordering:
$F \leq F'$ if and only if $F (h) \leq F (h')$ for every
$h \in \Lform X$.

There is a $\Pred_\DN^\bullet$ endofunctor on $\Topcat$, and its
action on morphisms $f \colon X \to Y$ is given by
$\Pred_\DN^\bullet f (F) = (h \in \Lform Y \mapsto F (h \circ f))$.
The $\Pred_\Nature^\bullet$ endofunctor is defined similarly.  Those
are both functor parts of monads.  Those monads, as well as the
forthcoming monad of sublinear previsions, were made explicit in
\cite[Proposition~2]{Gou-csl07}, on the category of dcpos (resp.,
pointed dcpos) and Scott-continuous maps, except that spaces of
previsions were given the Scott topology of the pointwise ordering.
The formulae for unit and multiplication is the same below.  The proof
that those form monads is elementary, and the only change compared to
\cite{Gou-csl07} is how we show continuity.
\begin{proposition}
  \label{prop:DN:monad}
  Let $\bullet$ be nothing, ``$\leq 1$'', or ``$1$''.  There is a
  monad
  $(\Pred_\DN^\bullet, \allowbreak \eta^\DN, \allowbreak \mu^\DN)$
  (resp.,
  $(\Pred_\Nature^\bullet, \allowbreak \eta^\Nature, \allowbreak
  \mu^\Nature)$) on $\Topcat$, whose unit and multiplication are
  defined at every topological space $X$ by:
  \begin{itemize}
  \item for every $x \in X$, $\eta^\DN_X (x) \eqdef (h \in \Lform X
    \mapsto h (x))$,
  \item for every
    $\mathcal F \in \Pred_\DN^\bullet {\Pred_\DN^\bullet X}$ (resp.,
    in $\Pred_\Nature^\bullet {\Pred_\Nature^\bullet X}$),
    $\mu^\DN_X (\mathcal F) \eqdef (h \in \Lform X \mapsto \mathcal F
    (F \mapsto F (h)))$, where $F$ ranges over $\Pred_\DN^\bullet X$,
    resp.\ $\Pred_\Nature^\bullet X$.
  \end{itemize}
\end{proposition}
\begin{proof}
  Unit.  The functional $(h \in \Lform X \mapsto h (x))$ is a linear
  prevision, hence also a superlinear prevision, as one checks easily.
  It is also normalized, hence subnormalized.  The map $\eta^\DN_X$ is
  continuous, since the inverse image of $[h > r]$ is
  $h^{-1} (]r, \infty])$.

  Multiplication.  For every $h \in \Lform X$, the map
  $F \mapsto F (h)$ is lower semicontinuous, since the inverse image
  of $]r, \infty]$ is $[h > r]$.  Hence $\mathcal F (F \mapsto F (h))$
  makes sense.  The map
  $(h \in \Lform X \mapsto \mathcal F (F \mapsto F (h)))$ is
  Scott-continuous, because $\mathcal F$ is.  Additionally, when
  $\mathcal F$ is superlinear (resp., linear), and $F$ ranges over
  $\Pred_\DN^\bullet X$, resp.\ $\Pred_\Nature^\bullet X$, it is clear
  that $(h \in \Lform X \mapsto \mathcal F (F \mapsto F (h)))$ is also
  superlinear (resp., linear).  Hence $\mu^\DN_X$ is a superlinear
  (resp., linear) prevision.  When $\bullet$ is ``$\leq 1$'', namely
  if $\mathcal F$ is subnormalized and $F$ ranges over
  $\Pred_\DN^{\leq 1} X$, then for every $h \in \Lform X$,
  $\mu^\DN_X (\one+h) = \mathcal F (F \mapsto F (\one+h)) \leq
  \mathcal F (F \mapsto 1 + F (h)) \leq 1 + \mathcal F (F \mapsto F
  (h)) = 1+\mu^\DN_X (h)$, so $\mu^\DN_X (\mathcal F)$ is
  subnormalized; similarly, if $\mathcal F$ is normalized, then
  $\mu^\DN_X (\mathcal F)$ is normalized.  (And yes, the dependency on
  $\bullet$ should be made explicit in the notation $\mu^\DN_X$, but
  we feel that the notation is heavy enough.)  Finally, $\mu^\DN_X$ is
  continuous since the inverse image of $[h > r]$ is
  $[(F \mapsto F (h)) > r]$.

  The monad laws are easily verified.  In fact, they are verified
  exactly as for the continuation monad, which is defined with the
  same formulae.  \qed
\end{proof}

There is a homeomorphism between $\Pred_\Nature^\bullet X$ and
$\Val_\bullet X$, for every space $X$: in one direction,
$\nu \in \Val_\bullet X$ is mapped to
$(h \in \Lform X \mapsto \int_{x \in X} h (x) \,d\nu)$, and in the
other direction any $G \in \Pred_\Nature^\bullet X$ is mapped to
$(U \in \Open X \mapsto G (\chi_U))$, where $\chi_U$ is the
characteristic map of $U$.  We note in particular that
$\int_{x \in X} \chi_U (x) \,d\nu = \nu (U)$.  This was first proved
by Kirch \cite[Satz~8.1]{kirch93}, under the additional assumption
that $X$ is core-compact; Tix later observed that this assumption was
unnecessary \cite[Satz~4.16]{Tix:bewertung}.

\section{The weak distributive law}
\label{sec:weak-distr-law}

We will show that there is a weak distributive law of $\SV$ over
$\Val_\bullet$ on the whole category of topological spaces.
This will be the collection of maps ${\lambda^\sharp_X}$ given as follows.
\begin{proposition}[Corollary~12.6 of \cite{JGL:projlim:prev}]
  \label{prop:Phi}
  For every topological space $X$, there is a continuous map
  ${\lambda^\sharp_X}$ from $\Val_\bullet {\SV X}$ to $\SV {\Val_\bullet X}$,
  which maps every $\mu \in \Val_\bullet {\SV X}$ to the collection of
  continuous valuations $\nu \in \Val_\bullet X$ such that
  $\nu (U) \geq \mu (\Box U)$ for every $U \in \Open X$.
\end{proposition}

We will also need to know how this is proved.

There is a map
$r_\DN \colon \SV (\Val_\bullet X) \to \Pred_{\DN}^\bullet X$, defined
by
$r_\DN (Q) (h) \eqdef \min_{\nu \in Q} \allowbreak \int_{x \in X} h
(x)\,d\nu$, and a map
$s_\DN^\bullet \colon \Pred_{\DN}^\bullet X \to \SV (\Val_\bullet X)$
defined by
$s_\DN^\bullet (F) \eqdef \{\nu \in \Val_\bullet X \mid \forall h \in
\Lform X, \int_{x \in X} h (x) \,d\nu \geq F (h)\}$.  Both are
continuous, and they form a retraction:
$r_\DN \circ s_\DN^\bullet = \identity {\Pred_{\DN}^\bullet X}$.  This
is the contents of Proposition~3.22 of \cite{JGL-mscs16}, only
replacing the space $\Pred_\Nature^\bullet X$ used there by
$\Val_\bullet X$, since the two are homeomorphic.  Additionally,
$r_\DN$ and $s_\DN^\bullet$ are natural in $X$ \cite[Lemma
11.2]{JGL:projlim:prev}.

This retraction even cuts down to a homeomorphism between
$\SV^{cvx} (\Val_\bullet X)$ and $\Pred_\DN^\bullet X$
\cite[Theorem~4.15]{JGL-mscs16}, where the former denotes the subspace
of $\SV (\Val_\bullet X)$ consisting of convex non-empty compact
saturated subsets of $\Val_\bullet X$.  (A subset $A$ of the latter is
\emph{convex} if and only if for all $G_1, G_2 \in A$, for every
$r \in [0, 1]$, $r G_1 + (1-r) G_2 \in A$.)

Then, there is a continuous map
$\Phi \colon \Val_\bullet {\SV X} \to \Pred_\DN^\bullet X$ defined by
$\Phi (\mu) (h) \eqdef \int_{Q \in \SV X} \min_{x \in Q} h (x) \,d\mu$
for every $h \in \Lform X$ \cite[Lemma 12.5]{JGL:projlim:prev}, and
the proof of Corollary~12.6 of \cite{JGL:projlim:prev} builds
${\lambda^\sharp_X}$ as $s_\DN^\bullet \circ \Phi$.
\begin{fact}
  \label{fact:lambda}
  ${\lambda^\sharp_X} = s_\DN^\bullet \circ \Phi$.
\end{fact}

\begin{remark}
  \label{rem:Phi}
  In other words, for every $\mu \in \Val_\bullet {\SV X}$,
  $\lambda^\sharp_X (\mu)$ is also the set $\{\nu \in \Val_\bullet X
  \mid \forall h \in \Lform X, \int_{x \in X} h (x) \,d\nu \geq
  \int_{Q \in \SV X} \min_{x \in Q} h (x) \,d\mu\}$.

  In order to illustrate what ${\lambda^\sharp_X}$ does, let us look
  at the special case where
  $\mu \eqdef \sum_{i=1}^n a_i \delta_{Q_i}$, where $n \geq 1$,
  $Q_i \eqdef \upc E_i$, and where each set $E_i$ is non-empty and
  finite.  For every $h \in \Lform X$,
  $\Phi (\mu) (h) = \int_{Q \in \SV X} \min_{x \in Q} h (x) \,d\mu =
  \sum_{i=1}^n a_i \min_{x_i \in Q_i} h (x_i) = \sum_{i=1}^n a_i
  \min_{x_i \in E_i} h (x_i) = \min_{x_1 \in E_1, \cdots, x_n \in E_n}
  \sum_{i=1}^n a_i h (x_i) = \min_{x_1 \in E_1, \cdots, x_n \in E_n}
  \allowbreak \int_{x \in X} h (x) \,d\sum_{i=1}^n a_i \delta_{x_i}$,
  so $\Phi = r_\DN (\upc \mathcal E)$ where $\mathcal E$ is the finite
  set
  $\{\sum_{i=1}^n a_i \delta_{x_i} \mid x_1 \in E_1, \cdots, x_n \in
  E_n\}$.  By \cite[Proposition~4.20]{JGL-mscs16},
  $s_\DN^\bullet \circ r_\DN$ maps every compact saturated subset $Q$
  of $\Val_\bullet X$ to its compact saturated convex hull, namely to
  $\upc conv (Q)$, where $conv$ denotes convex hull.  Hence
  ${\lambda^\sharp_X} (\mu) = \upc conv (\{\sum_{i=1}^n a_i
  \delta_{x_i} \mid x_1 \in E_1, \cdots, x_n \in E_n\})$.
\end{remark}

All those maps are continuous.  For $\Phi$, we have the following.
\begin{fact}
  \label{fact:Phi:cont}
  For all $h \in \Lform X$ and $r \in \Rp$,
  $\Phi^{-1} ([h > r]) = [h^* > r]$ where $h^* (Q)$ is defined as
  $\min_{x \in Q} h (x)$ for every $Q \in \SV X$.  The map $h^*$ is in
  $\Lform {\SV X}$, since for every $t \in \Rp$,
  ${h^*}^{-1} (]t, \infty]) = \Box {h^{-1} (]t, \infty])}$.
\end{fact}

\begin{fact}
  \label{fact:r:cont}
  For every $h \in \Lform X$, ${(r_\DN)}^{-1} ([h > 1]) = \Box {[h > 1]}$.
\end{fact}

The proof that $s_\DN^\bullet$ is continuous is complex
\cite[Lemma~3.20]{JGL-mscs16}.  We produce a simpler proof in
Lemma~\ref{lemma:sDP:inv:func} below, which additionally describes the
shape of inverse images of (sub)basic open sets by that map.  We note
that a base of open subsets of $\SV (\Val_\bullet X)$ is given by sets
of the form
$\Box {\bigcup_{i \in I} \bigcap_{j=1}^{n_i} [U_{ij} > r_{ij}]}$.
Since $\Box$ commutes with directed unions and
$\bigcup_{i \in I} \mathcal U_i = \dsup_{J \in \Pfin (I)} \bigcup_{i
  \in J} \mathcal U_i$, since finite unions and finite intersections
distribute, and since $\Box$ commutes with finite intersections, a
simpler base is given by open subsets of the form
$\Box {\bigcup_{i=1}^n [U_i > r_i]}$, where each $U_i$ is open in $X$
and $r_i \in \Rp$.  More generally, we consider open subsets of the
form $\Box {\bigcup_{i=1}^n [h_i > r_i]}$ where each $h_i$ is in
$\Lform X$ and $r_i \in \Rp \diff \{0\}$.  We may even restrict them
to $n \geq 1$, since $\Box \emptyset$ is empty.  Since
$[h_i > r_i] = [h_i/r_i > 1]$, we also restrict $r_i$ to be equal to
$1$.  Let
$\Delta_n \eqdef \{(t_1, \cdots, t_n) \in \Rp^n \mid \sum_{i=1}^n
t_i=1\}$.  We write $\vec a$ for the elements of $\Delta_n$;
and $a_i$ for component number $i$ of $\vec a$.
\begin{lemma}
  \label{lemma:sDP:inv:func}
  For every $n \geq 1$, for all $h_1, \cdots, h_n \in \Lform X$,
  \begin{align*}
    {(s_\DN^\bullet)}^{-1} (\Box {(\bigcup_{i=1}^n [h_i > 1])})
    & = \bigcup_{\vec a \in \Delta_n} \left[\sum_{i=1}^n a_i h_i >
      1\right].
  \end{align*}
\end{lemma}
\begin{proof}
  %
  Let
  $F \in \bigcup_{\vec a \in \Delta_n} \left[\sum_{i=1}^n a_i h_i >
    1\right]$.  Then $F (\sum_{i=1}^n a_i h_i) > 1$ for some
  $\vec a \in \Delta_n$.  For every $\nu \in s_\DN^\bullet (F)$, we
  then have
  $\int_{x \in X} \sum_{i=1}^n a_i h_i (x) \,d\nu \geq F (\sum_{i=1}^n
  a_i h_i) > 1$.  It follows that $\int_{i=1}^n h_i (x) \,d\nu > 1$
  for some $i \in \{1, \cdots, n\}$: otherwise
  $\int_{i=1}^n h_i (x) \,d\nu \leq 1$ for every $i$, so
  $\int_{x \in X} \sum_{i=1}^n a_i h_i (x) \,d\nu \leq 1$.  Since
  $\nu$ is arbitrary in $s_\DN^\bullet (F)$,
  $s_\DN^\bullet (F) \subseteq \bigcup_{i=1}^n [h_i > 1]$, so
  $F \in {(s_\DN^\bullet)}^{-1} (\Box {(\bigcup_{i=1}^n [h_i > 1])})$.

  Conversely, let
  $F \in {(s_\DN^\bullet)}^{-1} (\Box {(\bigcup_{i=1}^n [h_i > r_i])})$, and let
  us define $f \colon s_\DN^\bullet (F) \times \Delta_n \to \creal$ by
  $f (\nu, \vec a) \eqdef \int_{x \in X} \sum_{i=1}^n a_i h_i (x)
  \,d\nu$.  The space $s_\DN^\bullet (F)$, seen as a subspace of
  $\Val_\bullet X$, is non-empty and compact.  The map
  $f (\_, \vec a)$ is lower semicontinuous for every
  $\vec a \in \Delta_n$, since
  $f (\_, \vec a)^{-1} (]t, \infty]) = s_\DN^\bullet (F) \cap
  [\sum_{i=1}^n a_i h_i > t]$ for every $t \in \Rp$.  The Minimax
  Theorem~3.3 of \cite{JGL-minimax17} then says that: $(*)$
  $\sup_{\vec a \in \Delta_n} \inf_{\nu \in s_\DN^\bullet (F)} f (\nu,
  \vec a) = \min_{\nu \in s_\DN^\bullet (F)} \sup_{\vec a \in
    \Delta_n} f (\nu, \vec a)$ (and the $\min$ on the right is
  attained), provided we can show that $f$ is closely convex in its
  first argument and closely concave in its second argument.  Whatever
  the latter mean, Remark~3.4 of \cite{JGL-minimax17} tells us that
  maps that preserve pairwise linear combinations of their first
  arguments (with coefficients $a$ and $1-a$, $a \in [0, 1]$) are
  convex hence closely convex in their first argument, and similarly
  with ``concave'' instead of ``convex''; and it is clear that $f$ is
  linear in both its arguments.
  
  We recall that $F = r_\DN (s_\DN^\bullet (F))$, so
  $F (h) = \min_{\nu \in s_\DN^\bullet (F)} \int_{x \in X} h (x)
  \,d\nu$ for every $h \in \Lform X$.  Since
  $F \in {(s_\DN^\bullet)}^{-1} (\Box {(\bigcup_{i=1}^n [h_i > r_i])})$, for every
  $\nu \in s_\DN^\bullet (F)$, there is an $i \in \{1, \cdots, n\}$
  such that $\int_{x \in X} h_i (x) \,d\nu > 1$.  In particular, there
  is an $\vec a \in \Delta_n$ such that
  $\int_{x \in X} \sum_{i=1}^n a_i h_i (x) \,d\nu > 1$.  Therefore
  $\min_{\nu \in s_\DN^\bullet (F)} \sup_{\vec a \in \Delta_n} f (\nu,
  \vec a) > 1$.  By $(*)$,
  $\sup_{\vec a \in \Delta_n} \inf_{\nu \in s_\DN^\bullet (F)} f (\nu,
  \vec a) > 1$, so there is an $\vec a \in \Delta_n$ such that
  $\inf_{\nu \in s_\DN^\bullet (F)} \int_{x \in X} \sum_{i=1}^n a_i
  h_i (x) \,d\nu > 1$, namely such that
  $r_\DN (s_\DN^\bullet (F)) (\sum_{i=1}^n a_i h_i) > 1$, or
  equivalently such that $F (\sum_{i=1}^n a_i h_i) > 1$.  We conclude
  that
  $F \in \bigcup_{\vec a \in \Delta_n} \left[\sum_{i=1}^n a_i h_i >
    1\right]$.  \qed
\end{proof}

\begin{lemma}
  \label{lemma:lambda:inv:func}
  For every $n \geq 1$, for all $h_1, \cdots, h_n \in \Lform X$,
  \begin{align*}
    {\lambda^\sharp_X}^{-1} (\Box {(\bigcup_{i=1}^n [h_i > 1])})
    & = \bigcup_{\vec a \in \Delta_n} \left[(\sum_{i=1}^n a_i h_i)^* >
      1\right].
  \end{align*}
\end{lemma}

When $n=1$, there is just one element in $\Delta_n$, and that is
$(1)$.  We obtain the following as a special case.
\begin{corollary}
  \label{corl:lambda:inv:func:spec}
  For every $h \in \Lform X$,
  \begin{align*}
    {(s_\DN^\bullet)}^{-1} (\Box {[h > 1]}) & = [h > 1] \\
    {\lambda^\sharp_X}^{-1} (\Box {[h > 1]}) & = [h^* > 1].
  \end{align*}
\end{corollary}


\begin{lemma}
  \label{lemma:lambda:nat}
  $\lambda^\sharp$ is a natural transformation.
\end{lemma}
\begin{proof}
  Since $s_\DN^\bullet$ is natural, it suffices to show that $\Phi$ is
  natural.  Let $f \colon X \to Y$ be any continuous map.  We need to
  show that
  $\Phi \circ \Val_\bullet {\SV f} = \Pred_\DN^\bullet f \circ \Phi$.
  Let $\mu \in \Val_\bullet {\SV X}$ and $h \in \Lform Y$.  Then
  $(\Phi \circ \Val_\bullet {\SV f}) (\mu) (h) = \Phi (\SV f [\mu])
  (h) = \int_{Q \in \SV Y} \min_{y \in Q} h (y) \,d\SV f [\mu]$.  By
  the change of variable formula, this is equal to
  $\int_{Q \in \SV X} \min_{y \in \SV f (Q)} h (y) \,d\mu$.  Now
  $\min_{y \in \SV f (Q)} h (y) = \min_{y \in \upc f [Q]} h (y) =
  \min_{y \in f [Q]} h (y)$ (because $h$, being lower semicontinuous,
  is monotonic) $= \min_{x \in Q} h (f (x))$.  We therefore obtain
  that
  $(\Phi \circ \Val_\bullet {\SV f}) (\mu) (h) = \int_{Q \in \SV X}
  \min_{x \in Q} h (f (x)) \,d\mu$.  But
  $(\Pred_\DN^\bullet f \circ \Phi) (\mu) (h) = \Pred_\DN^\bullet f
  (\Phi (\mu)) (h) = \Phi (\mu) (h \circ f) = \int_{Q \in \SV X}
  \min_{x \in Q} h (f (x)) \,d\mu$.  \qed
\end{proof}

\begin{lemma}
  \label{lemma:lambda:etaS}
  ${\lambda^\sharp_X} \circ \Val_\bullet {\eta^\Smyth_X} =
  \eta^\Smyth_{\Val_\bullet X}$.
\end{lemma}
\begin{proof}
  For every $\nu \in \Val_\bullet X$, the elements of
  $({\lambda^\sharp_X} \circ \Val_\bullet {\eta^\Smyth_X}) (\nu)$ are those
  $\nu' \in \Val_\bullet X$ such that
  $\nu' (U) \geq \Val_\bullet {\eta^\Smyth_X} (\nu) (\Box U)$ for
  every $U \in \Open X$.  Now
  $\Val_\bullet {\eta^\Smyth_X} (\nu) (\Box U) = \nu
  ({(\eta^\Smyth_X)}^{-1} (\Box U)) = \nu (U)$, so
  $({\lambda^\sharp_X} \circ \Val_\bullet {\eta^\Smyth_X}) (\nu)$ is simply the
  collection of $\nu' \in \Val_\bullet X$ such that $\nu' \geq \nu$,
  namely $\eta^\Smyth_{\Val_\bullet X} (\nu)$.  \qed
\end{proof}

Lemma~\ref{lemma:lambda:etaS} is the first of the weak distributivity
laws.  The proofs of the others are more complex.  We will use the
following several times:

\vskip0.5em
\textbf{Trick A.} \emph{In order to show that two continuous maps
  $f, g \colon Y \to Z$ are equal, where $Z$ is a $T_0$ space, check
  that $f^{-1} (V) = g^{-1} (V)$ for every element $V$ of a subbase of
  the topology of $Z$.}

\vskip0.5em
Indeed, if the condition holds, then 
$f^{-1} (V) = g^{-1} (V)$ for every open subset $V$ of $Z$.
Hence, for every point $y \in Y$, $f (y)$ and $g (y)$ will have the
same open neighborhoods and therefore be equal, since $Z$ is $T_0$.

We also note that any space of the form $\Val_\bullet Y$,
$\Pred_\DN^\bullet Y$, or $\SV Y$ is $T_0$, since their specialization
preorderings are $\leq$, $\leq$, and $\subseteq$ respectively, and
they are antisymmetric.

In general, the maps $f$ and $g$ that are to be shown equal are
compositions $f_1 \circ \cdots \circ f_m$ and
$g_1 \circ \cdots \circ g_n$, and we will use Trick~A by writing lists
of the following kind:
$V \invto {f_1} V_1 \invto {f_2} \cdots \invto{f_m} V_m$ and
$V \invto {g_1} V'_1 \invto {g_2} \cdots \invto {g_n} V'_n$, meaning
that $V_1 = f_1^{-1} (V)$, etc., and finally checking that $V_m=V'_n$.

For example, an alternate proof of Lemma~\ref{lemma:lambda:etaS} along
these lines is:
\begin{align*}
  \Box {(\bigcup_{i=1}^n [h_i > 1])}
  & \invto{{\lambda^\sharp_X}}
    \bigcup_{\vec a \in \Delta_n} [(\sum_{i=1}^n a_i h_i)^* > 1]
  & \text{by Lemma~\ref{lemma:lambda:inv:func}} \\
  & \invto{\Val_\bullet {\eta^\Smyth_X}}
    \bigcup_{\vec a \in \Delta_n} [(\sum_{i=1}^n a_i h_i)^* \circ
    \eta^\Smyth_X > 1] \\
  & = \bigcup_{\vec a \in \Delta_n} [\sum_{i=1}^n a_i h_i > 1],
\end{align*}
since for every $h \in \Lform X$, $h^* \circ \eta^\Smyth_X = h$, and:
\begin{align*}
  \Box {(\bigcup_{i=1}^n [h_i > 1])}
  & \invto{\eta^\Smyth_{\Val_\bullet X}}
    \bigcup_{i=1}^n [h_i > 1].
\end{align*}
Now
$\bigcup_{\vec a \in \Delta_n} [\sum_{i=1}^n a_i h_i > 1] =
\bigcup_{i=1}^n [h_i > 1]$, as open subsets of $\Val_\bullet X$: for
every $\nu \in \Val_\bullet X$, if $\int_{x \in X} h_i (x) \,d\nu > 1$
for some $i$, then
$\int_{x \in X} \sum_{i=1}^n a_i h_i (x) \,d\nu > 1$ for the tuple
$\vec a$ whose $i$th component is $1$ and whose other components are
$0$.  Conversely, if
$\int_{x \in X} \sum_{i=1}^n a_i h_i (x) \,d\nu > 1$, then
$\int_{x \in X} h_i (x) \,d\nu > 1$ for some $i$, otherwise
$\int_{x \in X} h_i (x) \,d\nu \leq 1$ for every $i$, so
$\int_{x \in X} \sum_{i=1}^n a_i h_i (x) \,d\nu = \sum_{i=1}^n a_i
\int_{x \in X} h_i (x) \,d\nu \leq 1$.

We will also use the following.  \vskip0.5em \textbf{Trick~B.}
\emph{In order to show that $f=g$, where $f$ and $g$ are continuous
  maps from a space $Y$ to a space of the form $\SV {\Val_\bullet Z}$,
  show that $f (y)$ and $g (y)$ are convex for every $y \in Y$, and
  then show that $r_\DN \circ f = r_\DN \circ g$.}

\vskip0.5em
Indeed, we recall that $r_\DN$ is a homeomorphism of $\SV^{cvx}
(\Val_\bullet Z)$ onto $\Pred_\DN^\bullet Z$.

\begin{lemma}
  \label{lemma:lambda:muS}
  ${\lambda^\sharp_X} \circ \Val_\bullet {\mu^\Smyth_X} =
  \mu^\Smyth_{\Val_\bullet X} \circ \SV {\lambda^\sharp_X} \circ \lambda^\sharp_{\SV
    X}$.
\end{lemma}
\begin{proof}
  Since ${\lambda^\sharp_X} = s_\DN^\bullet \circ \Phi$ and the image
  of $s_\DN^\bullet$ consists of convex subsets, the left-hand side
  takes its values in convex sets.  For the right-hand side, for every
  $\xi \in \Val_\bullet {\SV {\SV X}}$, let
  $Q' \eqdef (\mu^\Smyth_{\Val_\bullet X} \circ \SV {\lambda^\sharp_X}
  \circ \lambda^\sharp_{\SV X}) (\xi)$.  We want to show that $Q'$ is
  convex.  The elements $\nu$ of $Q'$ are those of
  $\bigcup (\SV {\lambda^\sharp_X} \circ \lambda^\sharp_{\SV X})
  (\xi)$, namely those that are in some element $\mathcal Q$ of
  $(\SV {\lambda^\sharp_X} \circ \lambda^\sharp_{\SV X}) (\xi) = \upc
  {\lambda^\sharp_X} [\lambda^\sharp_{\SV X} (\xi)] = \{\mathcal Q \in \SV
  {\Val_\bullet X} \mid \exists \mu \in \lambda^\sharp_{\SV X} (\xi),
  \mathcal Q \subseteq {\lambda^\sharp_X} (\mu)\}$.  (Remember that
  the specialization ordering on $\SV {\Val_\bullet X}$ is reverse
  inclusion $\supseteq$.)  Simplifying this, and since $\upc \mu$ is
  always non-empty, compact and saturated, hence can take the place of
  $\mathcal Q$, the elements $\nu$ of $Q'$ are those such that there
  is a $\mu \in \lambda^\sharp_{\SV X} (\xi)$ such that
  $\nu \in {\lambda^\sharp_X} (\mu)$.

  Let $a \in [0, 1]$, and $\nu_1, \nu_2 \in Q'$.  Hence there is a
  $\mu_1 \in \lambda^\sharp_{\SV X} (\xi)$ such that
  $\nu_1 \in {\lambda^\sharp_X} (\mu_1)$ and there is a
  $\mu_2 \in \lambda^\sharp_{\SV X} (\xi)$ such that
  $\nu_2 \in {\lambda^\sharp_X} (\mu_2)$.  We note that
  $\lambda^\sharp_{\SV X} (\xi) = (s_\DN^\bullet \circ \Phi) (\xi)$ is
  convex, so $a\mu_1 + (1-a)\mu_2 \in \lambda^\sharp_{\SV X} (\xi)$.
  Additionally,
  $a\nu_1 + (1-a)\nu_2 \in {\lambda^\sharp_X} (a\mu_1 + (1-a)\mu_2)$: for every
  $U \in \Open X$, $\nu_1 (U) \geq \mu_1 (\Box U)$ and
  $\nu_2 (U) \geq \mu_2 (\Box U)$ by definition of ${\lambda^\sharp_X}$, so
  $(a\nu_1 + (1-a)\nu_2) (U) \geq (a\mu_1 + (1-a) \mu_2) (\Box U)$.
  Therefore there is a $\mu \in \lambda^\sharp_{\SV X} (\xi)$, namely
  $\mu \eqdef a\mu_1 + (1-a) \mu_2$, such that
  $a \nu_1 + (1-a) \nu_2 \in {\lambda^\sharp_X} (\mu)$.  Equivalently,
  $a \nu_1 + (1-a) \nu_2$ is in $Q'$.  Hence $Q'$ is convex, as
  promised.

  It follows that we can use Trick~B, and so we prove the following
  instead:
  \begin{align*}
    \Phi \circ \Val_\bullet {\mu^\Smyth_X}
    & = r_\DN \circ  \mu^\Smyth_{\Val_\bullet X} \circ \SV {\lambda^\sharp_X} \circ \lambda^\sharp_{\SV X}.
  \end{align*}
  Indeed,
  $r_\DN \circ {\lambda^\sharp_X} = r_\DN \circ s_\DN^\bullet \circ \Phi =
  \Phi$, which allows us to simplify the left-hand side.

  We now use Trick~A:
  \begin{align*}
    [h > 1]
    & \invto{\Phi} [h^* > 1] \\
    & \invto{\Val_\bullet {\mu^\Smyth_X}}
      [h^* \circ \mu^\Smyth_X > 1] = [h^{**} > 1],
  \end{align*}
  where for every $\mathcal Q \in \SV {\SV X}$,
  $h^{**} (\mathcal Q) = \min_{Q \in \mathcal Q} h^* (Q) = \min_{Q \in
    \mathcal Q} \min_{x \in Q} h (x)$; and indeed, the latter is equal
  to
  $\min_{x \in \bigcup \mathcal Q} h (x) = h^* (\mu^\Smyth_X (\mathcal
  Q))$.
  \begin{align*}
    [h > 1]
    & \invto{r_\DN}
      \Box {[h > 1]}
    & \text{by Fact~\ref{fact:r:cont}} \\
    & \invto{\mu^\Smyth_{\Val_\bullet X}}
      \Box {\Box {[h > 1]}} \\
    & \invto{\SV {{\lambda^\sharp_X}}}
      \Box {{\lambda^\sharp_X}^{-1} (\Box {[h > 1]})}
      = \Box {[h^* > 1]}
    & \text{by Corollary~\ref{corl:lambda:inv:func:spec}}\\
    & \invto{\lambda^\sharp_{\SV X}}
      [h^{**} > 1]
    & \text{by Corollary~\ref{corl:lambda:inv:func:spec}}.
  \end{align*}
  \qed
\end{proof}

\begin{lemma}
  \label{lemma:lambda:muT}
  ${\lambda^\sharp_X} \circ \mu_{\SV X} = \SV {\mu_X} \circ
  \lambda^\sharp_{\Val_\bullet X} \circ \Val_\bullet {\lambda^\sharp_X}$.
\end{lemma}
\begin{proof}
  The left-hand side takes its values in convex sets.  We claim that
  the right-hand side does, too.  For every
  $\xi \in \Val_\bullet {\Val_\bullet {\SV X}}$, the elements of
  $A \eqdef (\SV {\mu_X} \circ \lambda^\sharp_{\Val_\bullet X} \circ
  \Val_\bullet {\lambda^\sharp_X}) (\xi)$ are the continuous
  valuations $\nu \in \Val_\bullet X$ such that $\nu \geq \mu_X (\mu)$
  for some
  $\mu \in \lambda^\sharp_{\Val_\bullet X} ({\lambda^\sharp_X}
  [\xi]))$.  For all $a \in [0, 1]$, $\nu_1, \nu_2 \in A$, there are
  $\mu_1, \mu_2 \in \lambda^\sharp_{\Val_\bullet X}
  ({\lambda^\sharp_X} [\xi])$ such that $\nu_1 \geq \mu_X (\mu_1)$ and
  $\nu_2 \geq \mu_X (\mu_2)$.  Then $\mu \eqdef a \mu_1 + (1-a) \mu_2$
  is in $\lambda^\sharp_{\Val_\bullet X} ({\lambda^\sharp_X} [\xi])$,
  since $\lambda^\sharp_{\Val_\bullet X}$ takes its values in convex
  sets.  In turn,
  $a \nu_1 + (1-a) \nu_2 \geq a \mu_X (\mu_1) + (1-a) \mu_X (\mu_2) =
  a (\mathcal U \in \Open {\Val_\bullet X} \mapsto \int_{\nu \in
    \Val_\bullet X} \nu (\mathcal U) \,d\mu_1) + (1-a) (\mathcal U \in
  \Open {\Val_\bullet X} \mapsto \int_{\nu \in \Val_\bullet X} \nu
  (\mathcal U) \,d\mu_2) = \mu_X (\mu)$, so
  $a \nu_1 + (1-a) \nu_2 \in A$, showing that $A$ is convex.

  We can therefore use Trick~B, and since $r_\DN \circ {\lambda^\sharp_X} =
  r_\DN \circ s_\DN^\bullet \circ \Phi = \Phi$, it remains to show:
  \begin{align*}
    \Phi \circ \mu_{\SV X}
    & = r_\DN \circ \SV {\mu_X} \circ \lambda^\sharp_{\Val_\bullet X} \circ \Val_\bullet {\lambda^\sharp_X}.
  \end{align*}
  To this end, we use Trick~A:
  \begin{align*}
    [h > 1]
    & \invto{\Phi}
      [h^* > 1]
    & \text{by Fact~\ref{fact:Phi:cont}} \\
    & \invto{\mu_{\SV X}}
      \left[\left(\mu \in \Val_\bullet {\SV X} \mapsto \int_{Q \in \SV X} h^*
      (Q) \,d\mu\right) > 1\right]
    & \text{by (\ref{eq:mu-1})}
  \end{align*}
  while:
  \begin{align*}
    [h > 1]
    & \invto{r_\DN}
      \Box {[h > 1]}
    & \text{by Fact~\ref{fact:r:cont}} \\
    & \invto{\SV {\mu_X}}
      \Box {\mu_X^{-1} ([h > 1])}
      = \Box {\left[\left(\nu \in \Val_\bullet X \mapsto \int_{x \in X} h (x)
      \,d\nu\right) > 1\right]}
      \mskip-80mu
    & \text{by (\ref{eq:mu-1})} \\
    & \invto{\lambda^\sharp_{\Val_\bullet X}}
      \left[\left(\nu \in \Val_\bullet X \mapsto \int_{x \in X} h (x)
      \,d\nu\right)^* > 1\right]
    & \text{by Corollary~\ref{corl:lambda:inv:func:spec}} \\
    & \invto{\Val_\bullet {\lambda^\sharp_X}}
      \left[\left(\nu \in \Val_\bullet X \mapsto \int_{x \in X} h (x)
      \,d\nu\right)^* \circ {\lambda^\sharp_X} > 1\right].
  \end{align*}
  In order to conclude, it suffices to show that:
  \begin{align*}
    \left(\nu \in \Val_\bullet X \mapsto \int_{x \in X} h (x)
    \,d\nu\right)^* \circ {\lambda^\sharp_X}
    & = (\mu \in \Val_\bullet {\SV X} \mapsto \int_{Q \in \SV X} h^*
      (Q) \,d\mu).
  \end{align*}
  The left-hand side maps every $\mu \in \Val_\bullet {\SV X}$ to
  \begin{align*}
    \min_{\nu \in s_\DN^\bullet (\Phi (\mu))} \int_{x \in X} h (x)
    \,d\nu
    & = r_\DN (s_\DN^\bullet (\Phi (\mu))) (h) \\
    & = \Phi (\mu) (h) = \int_{Q \in \SV X} h^*
      (Q),
  \end{align*}
  and this completes the proof.  \qed
\end{proof}

Putting Lemma~\ref{lemma:lambda:nat}, Lemma~\ref{lemma:lambda:etaS},
Lemma~\ref{lemma:lambda:muS} and Lemma~\ref{lemma:lambda:muT}
together, we obtain the following.
\begin{theorem}
  \label{thm:lambda}
  Let $\bullet$ be nothing, ``$\leq 1$'', or ``$1$''.  Then
  $\lambda^\sharp$ is a weak distributive law of $\SV$ over
  $\Val_\bullet$ on the category $\Topcat$ of topological spaces.
\end{theorem}

\section{The associated weak composite monad}
\label{sec:associated-monad}

We will show that the associated weak composite monad (see
\cite[Section~3.2]{Garner:weak:distr}, or \cite[page~46]{Goy:PhD}) is
isomorphic to the monad $\Pred_\DN^\bullet$ of superlinear (resp., and
subnormalized, resp., and normalized) previsions, or equivalently to
the isomorphic monad $\SV^{cvx} \Val_\bullet$.

There is an idempotent:
\begin{equation}
  \label{eq:idemp}
  \xymatrix{
    \SV {\Val_\bullet X} \ar[r]^{\eta_{\SV {\Val_\bullet X}}}
    & \Val_\bullet {\SV {\Val_\bullet X}}
    \ar[r]^{\lambda^\sharp_{\Val_\bullet X}}
    & \SV {\Val_\bullet {\Val_\bullet X}}
    \ar[r]^{\SV {\mu_X}}
    & \SV {\Val_\bullet X}
  }
\end{equation}
and splitting it defines the action of the (functor part of the)
associated weak composite monad on each object $X$.  \emph{Splitting}
an idempotent $e$ means writing it as $\iota \circ \pi$, for a pair of
morphisms such that $\pi \circ \iota = \identity \relax$.

\begin{lemma}
  \label{lemma:idemp:split}
  A splitting of (\ref{eq:idemp}) is
  $\xymatrix{\SV {\Val_\bullet X} \ar[r]^{r_\DN} & \Pred_\DN^\bullet X
    \ar[r]^{s_\DN^\bullet} & \SV {\Val_\bullet X}}$.
\end{lemma}
\begin{proof}
  Let $f$ denote the idempotent (\ref{eq:idemp}).  We recall that
  $\lambda^\sharp_{\Val_\bullet X} = s_\DN^\bullet \circ \Phi$.  In
  particular, any element in its range is convex.  The image of any
  convex element $Q$ of a space $\SV {\Val_\bullet}$ under $\SV f$,
  where $f$ is any linear continuous map from
  $\Val_\bullet Y \to \Val_\bullet Z$, is convex; and $\mu_X$ is
  linear.  Therefore $f$ takes its values in
  $\SV^{cvx} {\Val_\bullet X}$.

  Using Trick~B, and simplifying using the fact that $r_\DN \circ
  s_\DN^\bullet$ is the identity, it remains to show that $r_\DN$ is
  equal to $r_\DN \circ f$, namely to:
  \[
    \xymatrix{
      \SV {\Val_\bullet X} \ar[r]^{\eta_{\SV {\Val_\bullet X}}}
      & \Val_\bullet {\SV {\Val_\bullet X}}
      \ar[r]^{\lambda^\sharp_{\Val_\bullet X}}
      & \SV {\Val_\bullet {\Val_\bullet X}}
      \ar[r]^{\SV {\mu_X}}
      & \SV {\Val_\bullet X}
      \ar[r]^{r_\DN}
      & \Pred_\DN^\bullet X
    }
  \]
  Using Trick~A,
  \begin{align*}
    [h > 1]
    & \invto{r_\DN}
      \Box {[h > 1]}
    & \text{by Fact~\ref{fact:r:cont}} \\
    & \invto{\SV {\mu_X}}
      \Box {\mu_X^{-1} ([h > 1])} \\
    & \qquad= \Box {\left[\left(\nu \in \Val_\bullet X \mapsto \int_{x \in
      X} h (x) \,d\nu\right) > 1\right]}
    & \text{by (\ref{eq:mu-1})} \\
    & \invto{\lambda^\sharp_{\Val_\bullet X}}
      \left[\left(\nu \in \Val_\bullet X \mapsto \int_{x \in
      X} h (x) \,d\nu\right)^* > 1\right]
    & \text{by Corollary~\ref{corl:lambda:inv:func:spec}} \\
    & \invto{\eta_{\SV {\Val_\bullet X}}}
      \left\{Q \in \SV {\Val_\bullet X} \mid \left(\nu \in \Val_\bullet X \mapsto \int_{x \in
      X} h (x) \,d\nu\right)^* (Q) > 1\right\}
    \mskip-360mu \\
    & \qquad= \left\{Q \in \SV {\Val_\bullet X} \mid r_\DN (Q) (h) > 1 \right\}
      = r_\DN^{-1} ([h > 1]).
      \mskip-250mu
  \end{align*}
  \qed
\end{proof}
We can always split an idempotent $e \colon Z \to Z$ in $\Topcat$ as
the composition of its corestriction to $e [Z]$ with the inclusion map
$e [Z] \to Z$.  In Lemma~\ref{lemma:idemp:split}, we might have chosen
the homeomorphic space $\SV^{cvx} {\Val_\bullet X}$ instead of
$\Pred_\DN^\bullet X$.

Generalizing (\ref{eq:idemp}), there is an idempotent:
\begin{equation}
  \label{eq:idemp:alpha}
  \xymatrix{
    \SV X \ar[r]^{\eta_{\SV X}}
    & \Val_\bullet {\SV X}
    \ar[r]^{\lambda^\sharp_X}
    & \SV {\Val_\bullet X}
    \ar[r]^{\SV \alpha}
    & \SV X
  }
\end{equation}
for every $\Val_\bullet$-algebra $\alpha$.  Finding out what its
splitting is is harder in general, especially because we have no
complete characterization of $\Val_\bullet$-algebras in $\Topcat$.  We
will only need to compute this splitting for one additional
$\Val_\bullet$-algebra.
\begin{lemma}
  \label{lemma:idemp:P}
  For every topological space $X$, there is a $\Val_\bullet$-algebra
  $\alpha_\DN^\bullet \colon \allowbreak \Val_\bullet
  {\Pred_\DN^\bullet X} \to \Pred_\DN^\bullet X$ defined by
  $\alpha_\DN^\bullet (\xi) (h) \eqdef \int_{F \in \Pred_\DN^\bullet
    X} F (h) \,d\xi$ for all
  $\xi \in \Val_\bullet {\Pred_\DN^\bullet X}$ and $h \in \Lform X$.

  The idempotent (\ref{eq:idemp:alpha}) with
  $\alpha \eqdef \alpha_\DN^\bullet$ splits as:
  \begin{align*}
    \xymatrix{\SV {\Pred_\DN^\bullet X} \ar[r]
    & \SV^{cvx} {\Pred_\DN^\bullet X} \ar[r] & \SV {\Pred_\DN^\bullet X}},
  \end{align*}
  where the arrow on the right is the inclusion map, and the arrow on
  the left is the corestriction of (\ref{eq:idemp:alpha}), and maps
  every $Q \in \SV^{cvx} {\Pred_\DN^\bullet X}$ to itself.
\end{lemma}
\begin{proof}
  The map $F \mapsto F (h)$ is lower semicontinuous, since the inverse
  image of $]t, \infty]$ is $[h > t]$, for all $h \in \Lform X$ and
  $t \in \Rp$.  Hence the integral
  $\int_{F \in \Pred_\DN^\bullet X} F (h) \,d\xi$ makes sense.  The
  map ${\alpha_\DN^\bullet} (\xi)$ is Scott-continuous because each
  $F \in \Pred_\DN^\bullet X$ is and because integration with respect
  to $\nu$ is Scott-continuous; it is superlinear because $F$ is and
  integration with respect to $\xi$ is linear.  Additionally,
  ${(\alpha_\DN^\bullet)}^{-1} ([h > r]) = [(F \mapsto F (h)) > r]$,
  so ${\alpha_\DN^\bullet}$ is continuous.

  We claim that the map $\alpha_\DN^\bullet$ is a
  $\Val_\bullet$-algebra.  We have
  $\alpha_\DN^\bullet (\eta_{\Pred_\DN^\bullet X} (F)) (h) = \int_{F'
    \in \Pred_\DN^\bullet X} F' (h)\,d\delta_F = F (h)$, so
  $\alpha_\DN^\bullet \circ \eta_{\Pred_\DN^\bullet X}$ is the
  identity map.  In order to show that
  $\alpha_\DN^\bullet \circ \mu_{\Pred_\DN^\bullet X} =
  \alpha_\DN^\bullet \circ \Val_\bullet {\alpha_\DN^\bullet}$, we use
  Trick~A:
  \begin{align*}
    [h > 1]
    & \invto{\alpha_\DN^\bullet}
      [(F \in \Pred_\DN^\bullet X \mapsto F (h)) > 1] \\
    & \invto{\mu_{\Pred_\DN^\bullet X}}
      [(\xi \in \Val_\bullet {\Pred_\DN^\bullet X} \mapsto
      \int_{F \in \Pred_\DN^\bullet X} F (h) \,d\xi) > 1]
    & \text{by (\ref{eq:mu-1})} \\
    [h > 1]
    & \invto{\alpha_\DN^\bullet}
      [(F \in \Pred_\DN^\bullet X \mapsto F (h)) > 1] \\
    & \invto{\Val_\bullet {\alpha_\DN^\bullet}}
      [(F \in \Pred_\DN^\bullet X \mapsto F (h)) \circ
      \alpha_\DN^\bullet > 1] \\
    & = [(\xi \in \Val_\bullet {\Pred_\DN^\bullet X} \mapsto
      \alpha_\DN^\bullet (\xi) (h)) > 1] \\
    & = [(\xi \in \Val_\bullet {\Pred_\DN^\bullet X} \mapsto
      \int_{F \in \Pred_\DN^\bullet X} F (h) \,d\xi) > 1].
  \end{align*}
  In order to justify the announced splitting of
  (\ref{eq:idemp:alpha}), we need to show that the image of
  (\ref{eq:idemp:alpha}) is exactly $\SV^{cvx} {\Pred_\DN^\bullet X}$.
  For every $Q \in \SV X$, the image of $Q$ by (\ref{eq:idemp:alpha})
  is $\upc \alpha_\DN^\bullet [Q']$ where
  $Q' \eqdef \lambda^\sharp_{\Pred_\DN^\bullet X} (\delta_Q)$ is
  convex, and $\alpha_\DN^\bullet$ is linear (in its $\xi$ argument),
  so that image is convex.

  In order to see that every element $Q$ of
  $\SV^{cvx} {\Pred_\DN^\bullet X}$ is reached as the image of some
  element by (\ref{eq:idemp:alpha}), we simply show that $Q$ is its
  own image by (\ref{eq:idemp:alpha}), or equivalently, that
  (\ref{eq:idemp:alpha}) restricts to the identity on
  $\SV^{cvx} {\Pred_\DN^\bullet X}$.  We use Trick~A:
  \begin{align*}
    \Box {\bigcup_{i=1}^n [h_i > 1]}
    & \invto{\SV {\alpha_\DN^\bullet}}
      \Box {\bigcup_{i=1}^n {(\alpha_\DN^\bullet)}^{-1} ([h_i > 1])}
    \\
    & = \Box {\bigcup_{i=1}^n [(F \mapsto F (h_i)) > 1]} \\
    & \invto{\lambda^\sharp_{\Pred_\DN^\bullet X}}
      \bigcup_{\vec a \in \Delta_n} \left[\left(F \mapsto \sum_{i=1}^n
      a_i F (h_i)\right)^* > 1\right]
    & \text{by Lemma~\ref{lemma:lambda:inv:func}} \\
    & \invto{\eta_{\SV {\Pred_\DN^\bullet X}}}
      \bigcup_{\vec a \in \Delta_n} \left(\left(F \mapsto \sum_{i=1}^n
      a_i F (h_i)\right)^*\right)^{-1} (]1, \infty]),
  \end{align*}
  and we claim that the convex elements $Q$ of $\SV {\Pred_\DN^\bullet
    X}$ that are in the latter open set are exactly those of $\Box
  {\bigcup_{i=1}^n [h_i > 1]}$, which will show the claim.

  We have
  $Q \in \bigcup_{\vec a \in \Delta_n} \left(\left(F \mapsto
      \sum_{i=1}^n a_i F (h_i)\right)^*\right)^{-1} (]1, \infty])$ if
  and only there is an $\vec a \in \Delta_n$ such that
  $\min_{F \in Q} \sum_{i=1}^n a_i F (h_i) > 1$.  For every
  $Q \in \SV^{cvx} {\Pred_\DN^\bullet X}$ and for every
  $i \in \{1, \cdots, n\}$, let
  $f (F, \vec a) \eqdef \sum_{i=1}^n a_i F (h_i)$.  This defines a
  function from $Q \times \Delta_n$ to $\creal$ that is lower
  semicontinuous in $Q$ (since $f$ is the composition of various sums
  and products with the functions $F \mapsto F (h_i)$, which are lower
  semicontinuous), is such that
  $f (a F_1 + (1-a) F_2, i) = a f (F_1, i) + (1-a) f (F_2, i)$ for all
  $a \in [0, 1]$ and $F_1, F_2 \in Q$ (we need $Q$ to be convex simply
  to make sure that $a F_1 + (1-a) F_2$ still lies in $Q$) hence is
  convex (the same property, with $\geq$ instead of $=$) hence closely
  convex in its first argument (see Remark~3.4 of
  \cite{JGL-minimax17}).  It is also concave hence closely concave in
  its second argument. By the Minimax Theorem~3.3 of
  \cite{JGL-minimax17}, therefore,
  $\sup_{\vec a \in \Delta_n} \min_{F \in Q} f (F, \vec a) = \min_{F
    \in Q} \sup_{\vec a \in \Delta_n} f (F, \vec a)$.  For every
  $Q \in \bigcup_{\vec a \in \Delta_n} \left(\left(F \mapsto
      \sum_{i=1}^n a_i F (h_i)\right)^*\right)^{-1} \allowbreak (]1,
  \infty])$, the left-hand side is strictly larger than $1$, so the
  right-hand side is, too.  Hence, for every $F \in Q$, there is an
  $\vec a \in \Delta_n$ such that $\sum_{i=1}^n a_i F (h_i) > 1$.  It
  follows easily that $F (h_i) > 1$ for some $i \in \{1, \cdots, n\}$,
  so $Q \in \Box {\bigcup_{i=1}^n [h_i > 1]}$.  Conversely, if
  $Q \in \Box {\bigcup_{i=1}^n [h_i > 1]}$, then for every $F \in Q$,
  there is an $i \in \{1, \cdots, n\}$ such that $F (h_i) > 1$, so
  there is an $\vec a \in \Delta_n$ such that
  $\sum_{i=1}^n a_i F (h_i) > 1$, and therefore
  $\sup_{\vec a \in \Delta_n} \min_{F \in Q} f (F, \vec a) > 1$,
  showing that $Q$ is in
  $\bigcup_{\vec a \in \Delta_n} \left(\left(F \mapsto \sum_{i=1}^n
      a_i F (h_i)\right)^*\right)^{-1} \allowbreak (]1, \infty])$.
%
  \qed
\end{proof}

Following B\"ohm, Garner showed that, for all monads $S$ and $T$ on a
category where idempotent splits, every weak distributive law
$\lambda \colon TS \to ST$ yields a weak lifting $\overline S$ of $S$
to the category of $T$-algebras, built as follows.

A $T$-semialgebra is a morphism $\alpha \colon TX \to X$ such that
$\alpha \circ \mu^T_X = \alpha \circ T \alpha$.  Given a weak
distributive law $\lambda$,
$S \alpha \circ \lambda_X \circ \eta^T_{SX} \colon SX \to SX$ is an
idempotent, and also a morphism of $T$-semialgebras.  We can split it
in $\Topcat$, as:
\[
  \xymatrix{
    SX \ar[r]_{\pi_\alpha}
    & \overline S X \ar[r]_{\iota_\alpha}
    & SX
  }
\]
since idempotents split in $\Topcat$.  We recall that splitting means
that the idempotent is equal to $\iota_\alpha \circ \pi_\alpha$, and
$\pi_\alpha \circ \iota_\alpha$ is the identity.  This actually splits
the idempotent in the category of $T$-semialgebras, as follows:
\[
  \xymatrix{
    TSX
    \ar[d]_{\lambda_X}
    \ar[r]^{T \pi_\alpha}
    & T\overline S X
    \ar[r]^{T \iota_\alpha}
    \ar[d]|{\lambda_X \circ T \iota_\alpha}
    & TSX
    \ar[d]_{\lambda_X}
    \\
    STX
    \ar[d]_{S \alpha}
    &
    STX
    \ar[d]|{\pi_\alpha \circ S\alpha}
    &
    STX
    \ar[d]_{S \alpha}
    \\
    SX \ar[r]_{\pi_\alpha}
    & \overline S X \ar[r]_{\iota_\alpha}
    & SX
  }
\]
Additionally, the middle vertical arrow $\overline S \alpha \eqdef
\pi_\alpha \circ S\alpha \circ \lambda_X \circ T\iota_\alpha$ is not
just a $T$-semialgebra, but a $T$-algebra.

This defines the object part of a endofunctor $\overline S$ on the
category of $T$-algebras.
On morphisms, the action of $\overline S$ is given by: for all
$T$-algebras $\alpha \colon TX \to X$ and $\beta \colon TY \to Y$, for
every $T$-algebra morphism $f \colon \alpha \to \beta$,
$\overline S f \eqdef \pi_\beta \circ S f \circ \iota_\alpha$.

$\overline S$ is part of a monad, whose unit
$\eta^{\overline S}_\alpha$ is $\pi_\alpha \circ \eta^S_X$, and whose
multiplication is
$\mu^{\overline S}_\alpha \eqdef \pi_\alpha \circ \mu^S_X \circ S
\iota_\alpha \circ \iota_{\overline S \alpha}$.  All of that was
proved by B\"ohm and Garner.

\begin{lemma}
  \label{lemma:Sbar}
  For $S \eqdef \SV$, $T \eqdef \Val_\bullet$,
  $\lambda \eqdef \lambda^\sharp$, the monad $\overline S$ on the
  category of $\Val_\bullet$-algebras has the following properties:
  \begin{enumerate}
  \item For every free $\Val_\bullet$-algebra
    $\mu_X \colon \Val_\bullet {\Val_\bullet X} \to \Val_\bullet X$,
    $\overline S \mu_X \colon \Val_\bullet {\Pred_\DN^\bullet X} \to
    \Pred_\DN^\bullet X$ is $\alpha_\DN^\bullet$ (see
    Lemma~\ref{lemma:idemp:P}).
  \item The unit $\eta^{\overline S}_{\mu_X}$ evaluated at the free
    $\Val_\bullet$-algebra
    $\mu_X \colon \Val_\bullet {\Val_\bullet X} \to \Val_\bullet X$ is
    $r_\DN \circ \eta^\Smyth_{\Val_\bullet X}$.
  \item The multiplication $\mu^{\overline S}_{\mu_X}$ evaluated at
    the free $\Val_\bullet$-algebra
    $\mu_X \colon \Val_\bullet {\Val_\bullet X} \to \Val_\bullet X$ is
    the function that maps every $Q \in \SV^{cvx} {\Pred_\DN^\bullet
      X}$ to $(h \in \Lform X \mapsto \min_{F \in Q} F (h)) \in
    \Pred_\DN^\bullet X$; the inverse image of $[h > 1]$ by that
    function is $\Box {[h > 1]} \cap \SV^{cvx} {\Pred_\DN^\bullet X}$,
    for every $h \in \Lform X$.
  \item The counit $\epsilon^T$ of the adjunction $F^T \dashv U^T$
    between $\Topcat$ and the category of $T$-algebras
    ($\Val_\bullet$-algebras) is given at each algebra $\beta \colon
    \Val_\bullet X \to X$ as $\beta$ itself.
  \end{enumerate}
\end{lemma}
\begin{proof}
  1. We know that $\overline S \mu_X$ must be a $\Val_\bullet$-algebra
  on $\overline S {\Val_\bullet X}$, and the latter can be taken as
  $\Pred_\DN^\bullet X$, by Lemma~\ref{lemma:idemp:split}.  (All
  splittings are up to isomorphism.)  We know that
  $\overline S \mu_X = \pi_{\mu_X} \circ S \mu_X \circ \lambda_X \circ
  T\iota_{\mu_X} = r_\DN \circ \SV \mu_X \circ {\lambda^\sharp_X}
  \circ \Val_\bullet {s_\DN^\bullet}$, which we elucidate using
  Trick~A:
  \begin{align*}
    [h > 1]
    & \invto{r_\DN}
      \Box {[h > 1]}
    & \text{by Fact~\ref{fact:r:cont}} \\
    & \invto{\SV {\mu_X}}
      \Box {\mu_X^{-1} ([h > 1])}
      = \Box {\left[\left(\nu \in \Val_\bullet X \mapsto \int_{x \in X} h (x)
      \,d\nu\right) > 1\right]}
      \mskip-80mu
    & \text{by (\ref{eq:mu-1})} \\
    & \invto{\lambda^\sharp_{\Val_\bullet X}}
      \left[\left(\nu \in \Val_\bullet X \mapsto \int_{x \in X} h (x)
      \,d\nu\right)^* > 1\right]
    & \text{by Corollary~\ref{corl:lambda:inv:func:spec}} \\
    & \invto{\Val_\bullet {s_\DN^\bullet}}
      \left[\left(\nu \in \Val_\bullet X \mapsto \int_{x \in X} h (x)
      \,d\nu\right)^* \circ s_\DN^\bullet > 1\right] \\
    & = [(F \in \Pred_\DN^\bullet X \mapsto F (h)) > 1].
  \end{align*}
  The last equality is justified as follows.  For every
  $F \in \Pred_\DN^\bullet X$,
  \begin{align*}
    \left(\left(\nu \in \Val_\bullet X \mapsto \int_{x \in X} h (x)
    \,d\nu\right)^* \circ s_\DN^\bullet\right) (F)
    & = \min_{\nu \in s_\DN^\bullet (F)} \int_{x \in X} h (x)
      \,d\nu \\
    & = r_\DN (s_\DN^\bullet (F)) (h) = F (h).
  \end{align*}
  Finally, $[(F \in \Pred_\DN^\bullet X \mapsto F (h)) > 1]$ is the
  collection of continuous valuations $\xi \in \Val_\bullet
  {\Pred_\DN^\bullet X}$ such that $\int_{F \in \Pred_\DN^\bullet X} F
  (h) \,d\xi > 1$, namely ${(\alpha_\DN^\bullet)}^{-1} ([h > 1])$.

  2.
  $\eta^{\overline S}_{\mu_X} = \pi_{\mu_X} \circ \eta^S_{\Val_\bullet
    X} = r_\DN \circ \eta^\Smyth_{\Val_\bullet X}$.

  3. In general, $\mu^{\overline S}_\alpha$ is a morphism from
  $\overline S {\overline S \alpha}$ to $\overline S \alpha$ in the
  category of $\Val_\bullet$-algebras.  When
  $\alpha = \mu_X \colon \Val_\bullet {\Val_\bullet X} \to
  \Val_\bullet X$, it is therefore one from
  $\Pred_\DN^\bullet {\Pred_\DN^\bullet X}$ to $\Pred_\DN^\bullet X$.
  We have
  $\mu^{\overline S}_{\mu_X} \eqdef \pi_{\mu_X} \circ
  \mu^S_{\Val_\bullet X} \circ S \iota_{\mu_X} \circ \iota_{\overline
    S \mu_X}$.  We may take
  $\iota_{\overline S \mu_X} = \iota_{\alpha_\DN^\bullet}$ to be the
  right arrow, from $\SV^{cvx} {\Pred_\DN^\bullet X}$ to
  $\SV {\Pred_\DN^\bullet X}$ in Lemma~\ref{lemma:idemp:P}---just
  subspace inclusion.  Then
  $S \iota_{\mu_X} = \SV {s_\DN^\bullet} \colon \SV {\Pred_\DN^\bullet
    X} \to \SV {\SV {\Val_\bullet X}}$, by
  Lemma~\ref{lemma:idemp:split}.  We compose that with
  $\mu^\Smyth_{\Val_\bullet X} \colon \SV {\SV {\Val_\bullet X}} \to
  \SV {\Val_\bullet X}$, then with
  $r_\DN \colon \SV {\Val_\bullet X} \to \Pred_\DN^\bullet X$.
  We elucidate what this composition is by relying on Trick~A:
  \begin{align*}
    [h > 1]
    & \invto{r_\DN}
      \Box {[h > 1]}
    & \text{by Fact~\ref{fact:r:cont}} \\
    & \invto{\mu^\Smyth_{\Val_\bullet X}}
      \Box {\Box {[h > 1]}} \\
    & \invto{\SV {s_\DN^\bullet}}
      \Box {{(s_\DN^\bullet)}^{-1} (\Box {[h > 1]})} \\
    & = \Box {[h > 1]}
    & \text{by Corollary~\ref{corl:lambda:inv:func:spec}} \\
    & \invto{\iota_{\alpha_\DN^\bullet}}
      \Box {[h > 1]} \cap \SV^{cvx} {\Pred_\DN^\bullet X}.
  \end{align*}
  Now this looks just like $r_\DN$, and we imitate it by letting
  $r \colon \SV^{cvx} {\Pred_\DN^\bullet X} \to \Pred_\DN^\bullet X$
  be defined by $r (Q) (h) \eqdef \min_{F \in Q} F (h)$.  This is
  well-defined and lower semicontinuous, and
  $r^{-1} ([h > 1]) = \Box {[h > 1]} \cap \SV^{cvx} {\Pred_\DN^\bullet
    X}$, so $\mu^{\overline S}_{\mu_X} = r$.

  4. Standard category theory, see \cite[Chapter~VI, Section~2,
  Theorem~1]{McLane:cat:math} for example.  \qed
\end{proof}

Finally, the weak composite monad of $S$ with $T$ is
$U^T \overline S F^T$, with unit
$U^T \eta^{\overline S} F^T \circ \eta^T$ and multiplication
$U^T \mu^{\overline S} F^T \circ U^T \overline S \epsilon^T \overline
S F^T$, where $F^T \dashv U^T$ is the adjunction between the base
category ($\Topcat$ in our case) and the category of $T$-algebras, and
$\epsilon^T$ is its counit.

\begin{theorem}
  \label{thm:weaklift}
  The weak composite monad associated with
  $\lambda^\sharp \colon \Val_\bullet {\SV} \to \SV {\Val_\bullet}$ on
  $\Topcat$ is $(\Pred_\DN^\bullet, \eta^\DN, \mu^\DN)$.
\end{theorem}
\begin{proof}
  We take $T \eqdef \Val_\bullet$, $S \eqdef \SV$,
  $\lambda \eqdef \lambda^\sharp$.  The functor $U^T$ takes every
  $\Val_\bullet$-algebra $\alpha \colon \Val_\bullet X \to X$ to $X$
  and every $\Val_\bullet$-algebra morphism
  $f \colon (\alpha \colon \Val_\bullet X \to X) \to (\beta \colon
  \Val_\bullet Y \to Y)$ to the underlying continuous map
  $f \colon X \to Y$.  The functor $F^T$ takes every space $X$ to
  $\mu_X \colon \Val_\bullet {\Val_\bullet X} \to \Val_\bullet X$, and
  every continuous map $f \colon X \to Y$ to $\Val_\bullet f$, seen as
  a morphism of $\Val_\bullet$-algebras.

  Then $U^T \overline S F^T$ maps every space $X$ to the splitting
  $\Pred_\DN^\bullet X$ of
  $\SV {\mu_X} \circ {\lambda^\sharp_X} \circ \eta_{\SV X}$ we have obtained in
  Lemma~\ref{lemma:idemp:split}.  We also obtain
  $\pi_{\mu_X} \eqdef r_\DN$, $\iota_{\mu_X} \eqdef s_\DN^\bullet$.

  Given any continuous map $f \colon X \to Y$, $U^T \overline S F^T
  (f)$ is equal to (the map underlying the $\Val_\bullet$-algebra)
  $\overline S (\Val_\bullet f)$, namely $\pi_{\mu_Y} \circ \SV
  {\Val_\bullet f} \circ \iota_{\mu_X} = r_\DN \circ \SV
  {\Val_\bullet f} \circ s_\DN^\bullet$.  We claim that this is equal
  to $\Pred_\DN^\bullet (f)$, namely that it maps every $F \in
  \Pred_\DN^\bullet X$ to $(h \in \Lform X \mapsto F (f \circ h))$.
  We use Trick~A:
  \begin{align*}
    [h > 1]
    & \invto{r_\DN}
      \Box {[h > 1]}
    & \text{by Fact~\ref{fact:r:cont}} \\
    & \invto{\SV {\Val_\bullet f}}
      \Box {(\Val_\bullet f)^{-1} ([h > 1])}
      = \Box {[h \circ f > 1]} \\
    & \invto{s_\DN^\bullet}
      [h \circ f > 1]
    & \text{by Corollary~\ref{corl:lambda:inv:func:spec}} \\
    & = {(\Pred_\DN^\bullet (f))}^{-1} ([h > 1]).
  \end{align*}
  From Lemma~\ref{lemma:Sbar}, item~2, the unit of $\overline S$ at
  the algebra $\mu_X$ is
  $\eta^{\overline S}_{\mu_X} = r_\DN \circ \eta^\Smyth_{\Val_\bullet
    X}$, which is the map
  $\nu \in \Val_\bullet X \mapsto r_\DN (\upc \nu)$, namely
  $\nu \in \Val_\bullet X \mapsto (h \in \Lform X \mapsto \int_{x \in
    X} h (x) \,d\nu)$ from $\Val_\bullet X$ to $\Pred_\DN^\bullet X$.

  The unit of the weak composite monad is
  $U^T \eta^{\overline S} F^T \circ \eta^T$.  At object $X$, this maps
  every $x \in X$ to
  $(h \in \Lform X \mapsto \int_{x' \in X} h (x') \,d\delta_x)$,
  namely to $(h \in \Lform X \mapsto h (x))$, and that is exactly the
  unit of the $\Pred_\DN^\bullet$ monad.

  The multiplication is
  $U^T \mu^{\overline S} F^T \circ U^T \overline S \epsilon^T
  \overline S F^T$, and here is what this means, explicitly.  At any
  object $X$, $\overline S F^T X$ is the $T$-algebra
  $\overline S\mu_X \colon T \overline S T X \to \overline S T X$.
  Then $\epsilon^T$ evaluated at this $T$-algebra is the bottom arrow
  in the following diagram:
  \[
    \xymatrix@C+10pt{
      TT\overline S T X
      \ar[d]_{\mu_{\overline S T X}}
      \ar[r]^{T \epsilon^T_{\overline S F^T X}}
      & T \overline S T X
      \ar[d]^{\overline S \mu_X}
      \\
      T \overline S T X
      \ar[r]_{\epsilon^T_{\overline S F^T X}}
      & \overline S T X
    }
  \]
  from the $T$-algebra $\alpha \eqdef \mu_{\overline S T X}$ that is
  the leftmost vertical arrow to the $T$-algebra
  $\beta \eqdef \overline S \mu_X$ on the right.  We apply $\overline
  S$ to the morphism $\epsilon^T_{\overline S F^T X}$ (i.e., we now
  apply $\overline S$ to a morphism of $T$-algebras, not to a
  $T$-algebra, as we did before), and we obtain the composition:
  \begin{equation}
    \label{eq:mu:part}
    \xymatrix{
      \overline S T \overline S T X
      \ar[r]^{\iota_\alpha}
      &
      S T \overline S T X
      \ar[r]^{S \epsilon^T_{\overline S F^T X}}
      &
      S \overline S T X
      \ar[r]^{\pi_\beta}
      &
      \overline S \overline S T X.
    }
  \end{equation}
  Applying $U^T$ to the latter, we obtain the same composition, this
  time seen as a morphism in the base category, instead of as a
  morphism in the category of $T$-algebras.  We finally compose it
  with $U^T \mu^{\overline S}_{F^T X}$, which is
  $\mu^{\overline S}_{\mu_X} \colon \overline S \overline S T X \to
  \overline S T X$.

  By Lemma~\ref{lemma:Sbar}, item~4,
  $\epsilon^T_{\overline S F^T X} \colon \alpha \to \beta$ is $\beta$
  itself, where $\beta = \overline S \mu_X = \alpha_\DN^\bullet$
  (Lemma~\ref{lemma:Sbar}, item~1) and
  $\alpha = \mu_{\overline S T X} = \mu_{\Pred_\DN^\bullet X}$: we
  have $\overline S T X = \Pred_\DN^\bullet X$, using
  Lemma~\ref{lemma:idemp:split}, and $\iota_\alpha = s_\DN^\bullet$.
  Using Lemma~\ref{lemma:idemp:P},
  $\pi_\beta \colon \SV {\Pred_\DN^\bullet X} \to \SV^{cvx}
  {\Pred_\DN^\bullet X}$ is the corestriction of
  (\ref{eq:idemp:alpha}) (call it $f$), which maps every convex
  element of $\SV {\Pred_\DN^\bullet X}$ to itself.

  Hence the multiplication of the weak composite monad, evaluated at
  $X$, is the composition of (\ref{eq:mu:part}) with $\mu^{\overline
    S}_{\mu_X}$ (given in Lemma~\ref{lemma:Sbar}, item~3), namely:
  \begin{equation}
    \label{eq:Sbarmu}
    \xymatrix{ 
      \Pred_\DN^\bullet {\Pred_\DN^\bullet X}
      \ar[r]^{s_\DN^\bullet}
      & \SV {\Val_\bullet {\Pred_\DN^\bullet X}}
      \ar[r]^{\SV {\alpha_\DN^\bullet}}
      & \SV {\Pred_\DN^\bullet X}
      \ar[r]^f
      & \SV^{cvx} {\Pred_\DN^\bullet X}
      \ar[r]^{\mu^{\overline S}_{\mu_X}}
      & \Pred_\DN^\bullet X
      }
  \end{equation}
  For every $\mathcal F \in \Pred_\DN^\bullet {\Pred_\DN^\bullet X}$,
  $\SV {\alpha_\DN^\bullet} (s_\DN^\bullet (\mathcal F)) = \upc
  \{\alpha_\DN^\bullet (\xi) \mid \xi \in s_\DN^\bullet (\mathcal F)\}
  = \upc \{(h \in \Lform X \mapsto \int_{F \in \Pred_\DN^\bullet X} F
  (h) \,d\xi) \mid \xi \in s_\DN^\bullet (\mathcal F)\}$.  It is easy
  to see that this is a convex set, because
  $s_\DN^\bullet (\mathcal F)$ is convex and integration is linear in
  the continuous valuation.  Therefore $f$ maps it to itself, and,
  using Lemma~\ref{lemma:Sbar}, item~3, $\mu^{\overline S}_{\mu_X}$
  maps it to:
  \begin{align*}
    (h \in \Lform X \mapsto \min_{F \in \SV {\alpha_\DN^\bullet}
    (s_\DN^\bullet (\mathcal F))} F (h))
    & = (h \in \Lform X \mapsto \min_{\xi \in s_\DN^\bullet (\mathcal
      F), F \geq \alpha_\DN^\bullet (\xi)} F (h)) \\
    & = (h \in \Lform X \mapsto \min_{\xi \in s_\DN^\bullet (\mathcal
      F)} \alpha_\DN^\bullet (\xi) (h)) \\
    & = (h \in \Lform X \mapsto \min_{\xi \in s_\DN^\bullet (\mathcal
      F)} \int_{F \in \Pred_\DN^\bullet X} F (h) \,d\xi) \\
    & = (h \in \Lform X \mapsto  r_\DN (s_\DN^\bullet (\mathcal F)) (F \in \Pred_\DN^\bullet X
      \mapsto F (h))) \\
    & = (h \in \Lform X \mapsto  \mathcal F (F \in \Pred_\DN^\bullet X
      \mapsto F (h)),
  \end{align*}
  and this is exactly what the multiplication of the
  $\Pred_\DN^\bullet$ monad, evaluated at $X$ and applied to $\mathcal
  F$, produces.
  \qed
\end{proof}


\part{A weak distributive law between $\Val_\bullet$ and $\HV$}

Just as $\SV$ is a model of demonic non-determinism, the Hoare
powerdomain is a model of \emph{angelic} non-determinism, and is
obtained as follows.  For every topological space $X$, let $\Hoare X$
be the set of non-empty closed subsets of $X$.  The \emph{lower
  Vietoris} topology on that set has subbasic open subsets $\Diamond
U$ consisting of those non-empty closed subsets that intersect $U$,
for each open subset $U$ of $X$.  We write $\HV X$ for the resulting
topological space.  Its specialization ordering is (ordinary)
inclusion $\subseteq$.

The $\HV$ construction was studied by Schalk
\cite[Section~6]{schalk:diss}, as well as a localic variant and a
variant with the Scott topology.  See also \cite[Sections~6.2.2,
6.2.3]{AJ:domains} or \cite[Section~IV-8]{GHKLMS:contlatt}.

There is a monad $(\HV, \eta^\Hoare, \mu^\Hoare)$ on $\Topcat$.  For
every continuous map $f \colon X \to Y$, $\HV f$ maps every
$C \in \SV X$ to $cl (f [C])$.  The unit $\eta^\Hoare$ is defined by
$\eta^\Hoare_X (x) \eqdef \dc x$ for every $x \in X$.  For every
continuous map $f \colon X \to \SV Y$, there is an extension
$f^\flat \colon \HV X \to \HV Y$, defined by
$f^\flat (C) \eqdef cl (\bigcup_{x \in C} f (x))$.  The multiplication
$\mu^\Hoare$ is defined by
$\mu^\Hoare_X \eqdef \identity {\HV X}^\flat$, namely
$\mu^\Hoare_X (\mathcal C) \eqdef cl (\bigcup \mathcal C)$.

We note the equalities ${\eta^\Hoare}^{-1} (\Diamond U) = U$,
${(\HV f)}^{-1} (\Diamond V) = \Diamond {f^{-1} (V)}$, and
${\mu^\Hoare}^{-1} (\Diamond U) = \Diamond {\Diamond U}$.

The remainder of this part is more or less directly copied from
Part~I, and is therefore pretty boring.  For the most part, we replace
$\Box$ with $\Diamond$, $\supseteq$ by $\subseteq$, $\lambda^\sharp_X$
by $\lambda^\flat_X$, etc.  But there are a few differences, if we
look closely enough.  For example, the sets $\Diamond U$ form a
\emph{subbase}, while the sets $\Box U$ formed a \emph{base}.  Perhaps
more importantly, we will not reason on the whole category $\Topcat$,
but on a full subcategory $\Topcat^\flat$ of spaces $X$ satisfying
certain conditions.  The reason is that there are maps $r_\AN$ and
$s_\AN^\bullet$, similar to the maps $r_\DN$ and $s_\DN^\bullet$, but
they are only known to enjoy similar properties on spaces $X$ that are
$\AN_\bullet$-friendly (see \cite{JGL-mscs16} and its errata
\cite{JGL:mscs16:errata})---we introduce $\AN_\bullet$-friendliness
below.  Also, the formula(e) for inverse images of subbasic open sets
by $s_\AN^\bullet$ will be trickier.

We import the following from \cite{Keimel:topcones2}.  A \emph{cone}
is a set with a scalar multiplication operation, by scalars from
$\Rp$, and with an addition operation, satisfying the expected laws.
A \emph{semitopological cone} is a cone with a topology that makes
both scalar multiplication and addition separately continuous, where
$\Rp$ is given the Scott topology.  For example, $\Lform X$, $\Val X$,
$\Pred_\DN X$, $\Pred_\AN X$ are semitopological cones, and
$\Val_\bullet X$, $\Pred_\DN^\bullet X$, $\Pred_\AN^\bullet X$ are
convex subspaces of the latter three.  In a semitopological cone,
scalar multiplication is always jointly continuous, but addition may
fail to be.  A \emph{topological cone} is one where addition is
jointly continuous.  A semitopological cone $C$ is \emph{locally
  convex} if and only if for every $x \in C$, every open neighborhood
of $x$ contains a convex open neighborhood of $x$.  It is
\emph{locally convex-compact} if and only if for every $x \in C$,
every open neighborhood of $x$ contains a convex compact saturated
neighborhood of $x$.

A space $X$ is \emph{$\AN_\bullet$-friendly}
\cite[Definition~1]{JGL:mscs16:errata}
if and only if:
\begin{itemize}[label=---]
\item $\bullet$ is nothing or ``$\leq 1$'', and $\Lform X$ is locally
  convex;
\item or $\bullet$ is ``$1$'', and either:
  \begin{enumerate}
  \item $\Lform X$ is locally convex and $X$ is compact;
  \item or $\Lform X$ is a locally convex, locally convex-compact,
    sober topological cone;
  \item or $X$ is LCS-complete.
  \end{enumerate}
\end{itemize}
Every locally compact space is core-compact, and every core-compact
space is $\AN_\bullet$-friendly, for any value of $\bullet$
\cite[Remark~2]{JGL:mscs16:errata}.
Every LCS-complete space is $\AN_\bullet$-friendly for any value of
$\bullet$
\cite[Remark~3]{JGL:mscs16:errata}.
We state this formally here.
\begin{fact}
  \label{fact:assum}
  Let $\bullet$ be nothing, ``$\leq 1$'', or ``$1$''.  Every locally
  compact space $X$ (even every core-compact space), every
  LCS-complete space $X$ is $\AN_\bullet$-friendly.
\end{fact}

The categorical machinery requires the subcategory $\Topcat^\flat$ to
be closed under the $\HV$ and $\Val_\bullet$ functors, as well as
under retracts (so that $\Topcat^\flat$, just like $\Topcat$, splits
idempotents).  For example, we may take $\Topcat^\flat$ to be the full
subcategory of locally compact spaces.  First, by
Fact~\ref{fact:assum}, such spaces are $\AN_\bullet$-friendly.
Second, for every locally compact space $X$, $\HV X$ is locally
compact, by \cite[Proposition 6.11]{schalk:diss}, and $\Val_\bullet X$
is locally compact if $\bullet$ is nothing or ``$\leq 1$'' (see
\cite[Theorem 12.2]{JGL:projlim:prev}, or
\cite[Theorem~4.1]{JGL:vlc}).  When $\bullet$ is ``$1$'',
$\Val_\bullet X$ is locally compact and compact if $X$ locally compact
and compact, and $\HV X$ is also locally compact, as we have seen, and
compact: the proof of \cite[Proposition 6.11]{schalk:diss} shows in
particular that
$\Diamond K \eqdef \{C \in \HV X \mid C \cap K \neq \emptyset\}$ is
compact for every compact subset $K$ of $X$, and the claim follows by
taking $K$ equal to $X$.  Finally, any retract of a locally compact
(resp., compact) space is locally compact (resp., compact); this
follows from the arguments developed in the proof of \cite[Proposition
2.17]{Jung:scs:prob}.  In brief, we work under the following
assumption.
\begin{definition}
  \label{defn:Top:flat}
  Let $\Topcat^\flat$ be any full subcategory of $\Topcat$ consisting
  of $\AN_\bullet$-friendly spaces, and closed under $\HV$,
  $\Val_\bullet$ and under retracts---for example the category of
  locally compact spaces if $\bullet$ is nothing or ``$\leq 1$'', or
  the category of locally compact, compact spaces if $\bullet$ is
  ``$1$''.
\end{definition}

\section{The monad of sublinear previsions}
\label{sec:monad-subl-prev}

A \emph{sublinear} prevision on a space $X$ is a prevision $F$ such
that $F (h+h') \leq F (h)+F (h')$ for all $h, h' \in \Lform X$.  We
write $\Pred_\AN X$ for the space of sublinear previsions on $X$,
$\Pred_\AN^{\leq 1} X$ for its subspace of subnormalized sublinear
previsions, and $\Pred_\AN^1 X$ for its subspace of normalized
sublinear previsions, and we use the synthetic notation
$\Pred_\AN^\bullet X$ for each.  The topology is generated by sets
that we continue to write as $[h > r]$, with $h \in \Lform X$ and
$r \in \Rp$, and defined as
$\{F \in \Pred_\AN^\bullet X \mid F (h) > r\}$.  Its specialization
ordering is the pointwise ordering: $F \leq F'$ if and only if
$F (h) \leq F (h')$ for every $h \in \Lform X$.

There is a $\Pred_\AN^\bullet$ endofunctor on $\Topcat$, which is
defined with the exact same formulae as for $\Pred_\DN^\bullet$ or
$\Pred_\Nature^\bullet$.  Its action on morphisms $f \colon X \to Y$
is given by
$\Pred_\AN^\bullet f (F) = (h \in \Lform Y \mapsto F (h \circ f))$.

As with superlinear and linear previsions, sublinear previsions form a
monad on $\Topcat$, which is a natural counterpart of the monad of
sublinear previsions on the category of dcpos given in
\cite[Proposition~2]{Gou-csl07}.  The following, which is
Proposition~\ref{prop:DN:monad} with ``$\DN$'' replaced by ``$\AN$'',
is proved in exactly the same way.
\begin{proposition}
  \label{prop:AN:monad}
  Let $\bullet$ be nothing, ``$\leq 1$'', or ``$1$''.  There is a
  monad
  $(\Pred_\AN^\bullet, \allowbreak \eta^\AN, \allowbreak \mu^\AN)$ on
  $\Topcat$, whose unit and multiplication are defined at every
  topological space $X$ by:
  \begin{itemize}
  \item for every $x \in X$, $\eta^\AN_X (x) \eqdef (h \in \Lform X
    \mapsto h (x))$,
  \item for every
    $\mathcal F \in \Pred_\AN^\bullet {\Pred_\AN^\bullet X}$,
    $\mu^\AN_X (\mathcal F) \eqdef (h \in \Lform X \mapsto \mathcal F
    (F \in \Pred_\AN^\bullet X \mapsto F (h)))$.
  \end{itemize}
\end{proposition}

For every $\AN_\bullet$-friendly space $X$, for example, for every
locally compact or LCS-complete space $X$ (see Fact~\ref{fact:assum}),
there is a function
$r_\AN \colon \HV (\Val_\bullet X) \to \Pred_{\AN}^\bullet X$, defined
by
$r_\AN (C) (h) \eqdef \sup_{\nu \in C} \allowbreak \int_{x \in X} h
(x)\,d\nu$, and a map
$s_\AN^\bullet \colon \Pred_{\AN}^\bullet X \to \HV (\Val_\bullet X)$
defined by
$s_\AN^\bullet (F) \eqdef \{\nu \in \Val_\bullet X \mid \forall h \in
\Lform X, \int_{x \in X} h (x) \,d\nu \leq F (h)\}$.  Both are
continuous, and
$r_\AN \circ s_\AN^\bullet = \identity {\Pred_{\AN}^\bullet X}$.  This
is the content of Corollary~3.12 of \cite{JGL-mscs16}, up to two
changes: first, we have replaced the space $\Pred_\Nature^\bullet X$
used there by $\Val_\bullet X$, since the two are homeomorphic;
second, we need to assume $X$ to be $\AN_\bullet$-friendly for this,
as indicated in the errata \cite[Corollary~3.12, repaired]{JGL:mscs16:errata}.

Additionally, $r_\AN$ and $s_\AN^\bullet$ are natural in $X$ on the
category of $\AN_\bullet$-spaces \cite[Lemma~13.2]{JGL:projlim:prev},
in particular on the category $\Topcat^\flat$.

\begin{corollary}
  \label{corl:AN:monad:flat}
  The monad $(\Pred_\AN^\bullet, \allowbreak \eta^\AN, \allowbreak
  \mu^\AN)$ restricts to a monad on any subcategory $\Topcat^\flat$ as
  given in Definition~\ref{defn:Top:flat}.
\end{corollary}
\begin{proof}
  It suffices to verify that $\Pred_\AN^\bullet X$ is an object of
  $\Topcat^\flat$ for every object $X$ of $\Topcat^\flat$, and this
  follows from the fact that it arises as a retract of $\HV
  {\Val_\bullet X}$ through $r_\AN$ and $s_\AN^\bullet$.  \qed
\end{proof}

This retraction even cuts down to a homeomorphism between
$\HV^{cvx} (\Val_\bullet X)$ and $\Pred_\AN^\bullet X$ provided that
$X$ is $\AN_\bullet$-friendly (see Theorem~4.11 of
\cite[Theorem~4.11]{JGL-mscs16} and the errata
\cite{JGL:mscs16:errata}); $\HV^{cvx} (\Val_\bullet X)$ denotes the
subspace of $\HV (\Val_\bullet X)$ consisting of convex non-empty
closed subsets of $\Val_\bullet X$.

As with $s_\DN^\bullet$, we propose the following simpler proof that
$s_\AN^\bullet$ is continuous, exhibiting the shape of inverse images
of subbasic open subsets $\Diamond {\bigcap_{i=1}^n [h_i > 1]}$ of
$\HV (\Val_\bullet X)$ ($n \geq 1$, $h_i \in \Lform X$).
\begin{lemma}
  \label{lemma:sAP:inv:func}
  For every $n \geq 1$, for every natural number $N > n$, let
  $\frac 1 N \nat$ denote the set of non-negative integer multiples of
  $\frac 1 N$, and $\Delta_n^N$ be the finite set
  $\{(b_1, \cdots, b_n) \in \frac 1 N \nat \mid 1 - \frac n N <
  \sum_{i=1}^n b_i \leq 1\}$.

  Let $\bullet$ be nothing, ``$\leq 1$'' or ``$1$'', and $X$ be an
  $\AN_\bullet$-friendly space.  For every
  $n \geq 1$, for all $h_1, \cdots, h_n \in \Lform X$,
  \begin{align*}
    {(s_\AN^\bullet)}^{-1} (\Diamond {(\bigcap_{i=1}^n [h_i > 1])})
    & = \bigcap_{\vec a \in \Delta_n} \left[\sum_{i=1}^n a_i h_i >
      1\right] \\
    & = \bigcup_{N > n} \bigcap_{\vec b \in \Delta_n^N} \left[\sum_{i=1}^n b_i h_i >
      1\right].
  \end{align*}
\end{lemma}
The first formula produces an intersection of open sets
$[\sum_{i=1}^n a_i h_i > 1]$ over all $\vec a \in \Delta_n$.  That is
an infinite intersection in general, and therefore it is not
immediately clear that this intersection is open.  The next line is an
infinite union of finite intersections of open sets, and is therefore
open.

\begin{proof}
  Letting (1), (2), (3) be the three values above, in order to prove
  $(1)=(2)=(3)$, we show $(1) \subseteq (3) \subseteq (2) \subseteq
  (1)$.

  $(1) \subseteq (3)$.  Let
  $F \in {(s_\AN^\bullet)}^{-1} (\Diamond {(\bigcap_{i=1}^n [h_i >
    1])})$.  There is a $\nu \in s_\AN^\bullet (F)$ such that
  $\int_{x \in X} h_i (x) \,d\nu > 1$ for every
  $i \in \{1, \cdots, n\}$.  We pick a natural number $N > n$ so large
  than $\frac N {N-n} < \min_{i=1}^n \int_{x \in X} h_i (x) \,d\nu$.
  Then, for every $\vec b \in \Delta_n^N$,
  $F (\sum_{i=1}^n b_i h_i) \geq \int_{x \in X} \sum_{i=1}^n b_i h_i
  (x) \,d\nu$ (because $\nu \in s_\AN^\bullet (F)$)
  $= \sum_{i=1}^n b_i \int_{x \in X} h_i (x) \,d\nu > \sum_{i=1}^n b_i
  \frac N {N-n} > (1 - \frac n N) \frac N {N-n} = 1$.

  $(3) \subseteq (2)$.  Let $F \in \Pred_\AN^\bullet X$ be such that
  there is a natural number $N > n$ such that, for every
  $\vec b \in \Delta_n^N$, $F (\sum_{i=1}^n b_i h_i) > 1$.  For every
  $\vec a \in \Delta_n$, we define $\vec b$ by
  $b_i \eqdef \frac 1 N \lfloor N a_i \rfloor$ for each
  $i \in \{1, \cdots, n\}$.  Then $b_i \in \frac 1 N \nat$, and
  $\sum_{i=1}^n b_i$ lies between
  $\sum_{i=1}^n \frac {Na_i - 1} N = 1 - \frac n N$ (strictly) and
  $\sum_{i=1}^n a_i = 1$, so $\vec b$ is in $\Delta_n^N$.  Also,
  $a_i \geq b_i$ for every $i$, so, using the monotonicity of $F$,
  $F (\sum_{i=1}^n a_i h_i) \geq F (\sum_{i=1}^n b_i h_i) > 1$; this
  holds for every $\vec a \in \Delta_n$.

  $(2) \subseteq (1)$.  Let
  $F \in \bigcap_{\vec a \in \Delta_n} [\sum_{i=1}^n a_i h_i > 1]$.
  Since $r_\AN \circ s_\AN^\bullet$ is the identity,
  $F (\sum_{i=1}^n a_i h_i) = \sup_{\nu \in s_\AN^\bullet (F)} \int_{x
    \in X} \sum_{i=1}^n a_i h_i (x) \,d\nu$---and
  $F (\sum_{i=1}^n a_i h_i) > 1$---for every $\vec a \in \Delta_n$.

  Let $f \colon \Delta_n \times s_\AN^\bullet (F) \to \creal$ map
  $(\vec a, \nu)$ to $\int_{x \in X} \sum_{i=1}^n a_i h_i (x) \,d\nu$.
  We have just shown that
  $\sup_{\nu \in s_\AN^\bullet (F)} f (\vec a, \nu) > 1$ for every
  $\vec a \in \Delta_n$.
  
  We equip $\Delta_n$ with the subspace topology induced by the
  inclusion in $(\creal)^n$, where each copy of $\creal$ has the Scott
  topology: this makes $f$ lower semicontinuous in its first argument,
  since addition and scalar multiplication are Scott-continuous on
  $\creal$.  The topology on $\Delta_n$ has a base of open subsets
  $\Delta_n \cap \prod_{i=1}^n ]b_i, \infty]$, where each
  $b_i \in \Rp$ and $\sum_{i=1}^n b_i < 1$ (otherwise the intersection
  with $\Delta_n$ is empty, and can therefore be disregarded).  But
  $\Delta_n \cap \prod_{i=1}^n ]b_i, \infty] = \Delta_n \cap
  \prod_{i=1}^n ]b_i, 1+\epsilon[$, where $\epsilon$ is any strictly
  positive real number; therefore $\Delta_n$ also has the subspace
  topology induced by the inclusion in $\real^n$ with its usual metric
  topology.  We then know that $\Delta_n$ is compact (and Hausdorff).
  The map $f$ preserves pairwise linear combinations (with
  coefficients $a$ and $1-a$, $a \in [0, 1]$) of its first arguments,
  and similarly with its second arguments, so by Remark~3.4 and the
  Minimax Theorem~3.3 of \cite{JGL-minimax17},
  $\sup_{\nu \in s_\AN^\bullet (F)} \inf_{\vec a \in \Delta_n} f (\vec
  a, \nu) = \min_{\vec a \in \Delta_n} \sup_{\nu \in s_\AN^\bullet
    (F)} f (\vec a, \nu)$.

  But $\sup_{\nu \in s_\AN^\bullet (F)} f (\vec a, \nu) > 1$ for every
  $\vec a \in \Delta_n$, in other words
  $\min_{\vec a \in \Delta_n} \allowbreak \sup_{\nu \in s_\AN^\bullet
    (F)} f (\vec a, \nu) > 1$, so
  $\sup_{\nu \in s_\AN^\bullet (F)} \inf_{\vec a \in \Delta_n} f (\vec
  a, \nu) > 1$.  Hence there is a continuous valuation
  $\nu \in s_\AN^\bullet (F)$ such that
  $\inf_{\vec a \in \Delta_n} f (\vec a, \nu) > 1$, in particular such
  that $\int_{x \in X} \sum_{i=1}^n a_i h_i (x) \,d\nu > 1$ for every
  $\vec a \in \Delta_n$, in particular such that
  $\int_{x \in X} h_i (x) \,d\nu > 1$ for every
  $i \in \{1, \cdots n\}$.  Therefore
  $\nu \in \bigcap_{i=1}^n [h_i > 1]$, showing that
  $s_\AN^\bullet (F) \in \Diamond {\bigcap_{i=1}^n [h_i > 1]}$.  \qed
\end{proof}

The following is immediate.
\begin{fact}
  \label{fact:r:cont:flat}
  For every topological space $X$, for every $h \in \Lform X$,
  ${(r_\AN)}^{-1} ([h > 1]) = \Diamond {[h > 1]}$.
\end{fact}

\section{The weak distributive law}
\label{sec:weak-distr-law:flat}

We will show that there is a weak distributive law of $\HV$ over
$\Val_\bullet$ on $\Topcat^\flat$.  This will be the collection of
maps ${\lambda^\flat_X}$ as $s_\AN^\bullet \circ \Psi$,
where $\Psi$ is given below.

\begin{lemma}
  \label{lemma:Psi}
  Let $X$ be a topological space.  For every $h \in \Lform X$, the
  function $h_* \colon C \mapsto \sup_{x \in C} h (x)$ is in
  $\Lform {\HV X}$, and for every $r \in \Rp$,
  $h_*^{-1} (]r, \infty]) = \Diamond {h^{-1} (]r, \infty])}$.  There
  is a continuous map
  $\Psi \colon \Val_\bullet {\HV X} \to \Pred_\AN^\bullet X$ such
  that:
  \begin{align*}
    \Psi (\mu) (h) & \eqdef \int_{C \in \HV X} h_* (C) \,d\mu
  \end{align*}
  for every $h \in \Lform X$.  For every subbasic open subset
  $[h > r]$ of $\Pred_\AN^\bullet X$,
  \begin{align*}
    \Psi^{-1} ([h > r]) & = [h_* > r].
  \end{align*}
\end{lemma}
\begin{proof}
  For every $r \in \Rp$, $h_*^{-1} (]r, \infty]) = \{C \in \HV X \mid
  \sup_{x \in C} h (x) > r\} = \{C \in \HV X \mid \exists x \in C, h
  (x) > r\} = \Diamond {h^{-1} (]r, \infty])}$.  This entails that
  $h_*$ is lower semicontinuous.

  In particular, $\Psi (\mu) (h)$ is well-defined for every
  $\mu \in \Val_\bullet X$ and for every $h \in \Lform X$.  It is
  clear that $\Psi$ is monotonic in $h$.  If $h$ is the (pointwise)
  supremum of a directed family ${(h_i)}_{i \in I}$ in $\Lform X$,
  then for every $C \in \HV X$,
  $h_* (C) = \sup_{x \in C} \dsup_{i \in I} h_i (x) = \dsup_{i \in I}
  h_i^* (C)$.  Since integration with respect to the continuous
  valuation $\mu$ is Scott-continuous,
  $\Psi (\mu) (h) = \dsup_{i \in I} \Psi (\mu) (h_i)$, showing that
  $\Psi (\mu)$ is Scott-continuous.  For all $a \in \Rp$,
  $h, h' \in \Lform X$, we have $(ah)_* = a h_*$ and
  $(h+h')_* \leq h_* + {h'}_*$, so $\Psi (\mu)$ is sublinear.
  Additionally, $(\one + h)_* = \one + h_*$, so $\Psi (\mu)$ is
  (sub)normalized if $\mu$ is a (sub)probability valuation.  Hence
  $\Psi (\mu) \in \Pred_\AN^\bullet X$.

  Finally, $\Psi^{-1} ([h > r])$ is the collection of all continuous
  valuations $\mu \in \Val_\bullet X$ such that
  $\int_{C \in \HV X} h_* (C) \,d\mu > r$, namely $[h^* > r]$.  In
  particular, $\Psi$ is continuous.  \qed
\end{proof}

\begin{proposition}
  \label{prop:lambda:flat}
  Let $\bullet$ be nothing, ``$\leq 1$'' or ``$1$'', and $X$ be an
  $\AN_\bullet$-friendly space.  The continuous function
  ${\lambda^\flat_X} \eqdef s_\AN^\bullet \circ \Psi \colon
  \Val_\bullet {\HV X} \to \HV {\Val_\bullet X}$ maps every
  $\mu \in \Val_\bullet {\HV X}$ to the collection of continuous
  valuations $\nu \in \Val_\bullet X$ such that
  $\nu (U) \leq \mu (\Diamond U)$ for every $U \in \Open X$.  On
  subbasic open sets,
  \begin{align*}
    {\lambda^\flat_X}^{-1} (\Diamond {\bigcap_{i=1}^n [h_i > 1]})
    & = \bigcap_{\vec a \in \Delta_n} \left[\left(\sum_{i=1}^n a_i h_i\right)_* >
      1\right] \\
    & = \bigcup_{N > n} \bigcap_{\vec b \in \Delta_n^N} \left[\left(\sum_{i=1}^n b_i h_i\right)_* >
      1\right].
  \end{align*}
\end{proposition}
\begin{proof}
  The formula for inverse images is a direct consequence of
  Lemma~\ref{lemma:sAP:inv:func} and of Lemma~\ref{lemma:Psi}.

  Let $\mu \in \Val_\bullet {\HV X}$.
  The elements $\nu$ of ${\lambda^\flat_X} (\mu)$ are those $\nu \in
  \Val_\bullet X$ such that $\int_{x \in X} h (x)\,d\nu \leq \int_{C
    \in \HV X} h_* (C) \,d\mu$ for every $h \in \Lform X$.
  In particular, for $h \eqdef \chi_U$, where $U \in \Open X$, it is
  easy to
  check that $h_* =
  \chi_{\Diamond U}$, so $\nu (U) \leq \mu (\Diamond U)$.
  Conversely, if $\nu (U) \leq \mu (\Diamond U)$ for every $U \in
  \Open X$, for every $h \in \Lform X$,
  $\int_{x \in X} h (x)\,d\nu = \int_0^\infty \nu (h^{-1} (]t,
  \infty])) \,dt \leq \int_0^\infty \mu (\Diamond {h^{-1} (]t,
    \infty])})\,dt = \int_0^\infty \mu (h_*^{-1} (]t, \infty])) \,dt =
  \int_{x \in X} h_* (x) \,d\mu$.  \qed
\end{proof}

When $n = 1$, $\Delta_n = \{(1)\}$, and $\Delta_n^N = \{(1)\}$ for
every $N > n$.  Lemma~\ref{lemma:sAP:inv:func} and
Proposition~\ref{prop:lambda:flat} then specialize to the following.
\begin{corollary}
  \label{corl:lambda:flat:inv:func:spec}
  Let $\bullet$ be nothing, ``$\leq 1$'' or ``$1$'', and let $X$ be an
  $\AN_\bullet$-friendly space.  For every $h \in \Lform X$,
  \begin{align*}
    {(s_\AN^\bullet)}^{-1} (\Diamond {[h > 1]})
    & = [h > 1] \\
    {\lambda^\flat_X}^{-1} (\Diamond {[h > 1]})
    & = [h_* >1].
  \end{align*}
\end{corollary}

\begin{remark}
  \label{rem:Psi}
  Let $X$ be $\AN_\bullet$-friendly.  For every
  $\mu \in \Val_\bullet {\SV X}$, $\lambda^\flat_X (\mu)$ is also the
  set
  $\{\nu \in \Val_\bullet X \mid \forall h \in \Lform X, \int_{x \in
    X} h (x) \,d\nu \leq \int_{C \in \HV X} \sup_{x \in C} h (x)
  \,d\mu\}$.

  In order to illustrate what ${\lambda^\flat_X}$ does, let us look at
  the special case where $\mu \eqdef \sum_{i=1}^n a_i \delta_{C_i}$,
  where $n \geq 1$, $C_i \eqdef \dc E_i$ and where each set $E_i$ is
  non-empty and finite.  For every $h \in \Lform X$,
  $\int_{C \in \HV X} \sup_{x \in C} h (x) \,d\mu = \sum_{i=1}^n a_i
  \sup_{x_i \in C_i} h (x_i) = \sum_{i=1}^n a_i \max_{x_i \in E_i} h
  (x_i) = \max_{x_1 \in E_1, \cdots, x_n \in E_n} \sum_{i=1}^n a_i h
  (x_i)$, which is equal to
  $\max_{x_1 \in E_1, \cdots, x_n \in E_n} \int_{x \in X} h (x)
  \,d\sum_{i=1}^n a_i \delta_{x_i}$, so
  $\Psi = r_\AN (\dc \mathcal E)$ where $\mathcal E$ is the finite set
  $\{\sum_{i=1}^n a_i \delta_{x_i} \mid x_1 \in E_1, \cdots, x_n \in
  E_n\}$.  By Proposition~4.19 of \cite{JGL-mscs16} (see also the
  errata \cite{JGL:mscs16:errata}), which applies since $X$ is
  $\AN_\bullet$-friendly, $s_\AN^\bullet \circ r_\AN$ maps every
  closed subset $C$ of $\Val_\bullet X$ to its closed convex hull,
  namely to $cl (conv (C)$, where $conv$ denotes convex hull.  Hence
  ${\lambda^\flat_X} (\mu) = cl (conv (\{\sum_{i=1}^n a_i \delta_{x_i}
  \mid x_1 \in E_1, \cdots, x_n \in E_n\}))$.
\end{remark}

We will use the following easily proved fact.
\begin{fact}
  \label{fact:sup:cl}
  For every topological space $X$, for every $h \in \Lform X$,
  for every subset $A$ of $X$, $\sup_{x \in cl (A)} h (x) = \sup_{x
    \in A} h (x)$.
\end{fact}
Indeed, we use the fact that an open subset intersects $cl (A)$ if and
only if it intersects $A$.  Then, for every $t \in \Rp$,
$\sup_{x \in cl (A)} h (x) > t$ if and only if the open set
$h^{-1} (]t, \infty])$ intersects $cl (A)$, if and only if it
intersects $A$, if and only if $\sup_{x \in A} h (x) > t$.

\begin{lemma}
  \label{lemma:lambda:flat:nat}
  $\lambda^\flat$ is a natural transformation on $\Topcat^\flat$.
\end{lemma}
\begin{proof}
  Since $s_\AN^\bullet$ is natural, it suffices to show that $\Psi$ is
  natural.  Let $f \colon X \to Y$ be any continuous map.  We need to
  show that
  $\Psi \circ \Val_\bullet {\HV f} = \Pred_\AN^\bullet f \circ \Psi$.
  Let $\mu \in \Val_\bullet {\HV X}$ and $h \in \Lform Y$.  Then
  $(\Psi \circ \Val_\bullet {\HV f}) (\mu) (h) = \Psi (\HV f [\mu])
  (h) = \int_{C \in \HV Y} \sup_{y \in C} h (y) \,d\HV f [\mu]$.  By
  the change of variable formula, this is equal to
  $\int_{C \in \HV X} \sup_{y \in \HV f (C)} h (y) \,d\mu$.  Now
  $\sup_{y \in \HV f (C)} h (y) = \sup_{y \in cl (f [C])} h (y) =
  \sup_{y \in f [C]} h (y)$ (by Fact~\ref{fact:sup:cl})
  $= \sup_{x \in C} h (f (x))$.  We therefore obtain that
  $(\Psi \circ \Val_\bullet {\HV f}) (\mu) (h) = \int_{C \in \HV X}
  \sup_{x \in C} h (f (x)) \,d\mu$.  But
  $(\Pred_\AN^\bullet f \circ \Psi) (\mu) (h) = \Pred_\AN^\bullet f
  (\Psi (\mu)) (h) = \Psi (\mu) (h \circ f) = \int_{C \in \HV X}
  \sup_{x \in C} h (f (x)) \,d\mu$.  \qed
\end{proof}

We prove the first of the weak distributivity laws.
\begin{lemma}
  \label{lemma:lambda:flat:etaS}
  For every space $X$ in $\Topcat^\flat$,
  ${\lambda^\flat_X} \circ \Val_\bullet {\eta^\Hoare_X} =
  \eta^\Hoare_{\Val_\bullet X}$.
\end{lemma}
\begin{proof}
  For every $\nu \in \Val_\bullet X$, the elements of
  $({\lambda^\flat_X} \circ \Val_\bullet {\eta^\Hoare_X}) (\nu)$ are
  those $\nu' \in \Val_\bullet X$ such that
  $\nu' (U) \leq \Val_\bullet {\eta^\Hoare_X} (\nu) (\Diamond U)$ for
  every $U \in \Open X$.  Now
  $\Val_\bullet {\eta^\Hoare_X} (\nu) (\Diamond U) = \nu
  ({(\eta^\Hoare_X)}^{-1} (\Diamond U)) = \nu (U)$, so
  $({\lambda^\flat_X} \circ \Val_\bullet {\eta^\Hoare_X}) (\nu)$ is
  simply the collection of $\nu' \in \Val_\bullet X$ such that
  $\nu' \leq \nu$, namely $\eta^\Hoare_{\Val_\bullet X} (\nu)$.  \qed
\end{proof}

For the remaining proofs, we will use Trick~A, and the following
variant of Trick~B.  \vskip0.5em \textbf{Trick~C.} \emph{In order to
  show that $f=g$, where $f$ and $g$ are continuous map from a space
  $Y$ to a space of the form $\HV {\Val_\bullet Z}$, where $Z$ is an
  object of $\Topcat^\flat$, show that $f (y)$ and $g (y)$ are convex
  for every $y \in Y$, and then show that
  $r_\AN \circ f = r_\AN \circ g$.}

\vskip0.5em Indeed, we recall that $r_\AN$ is a homeomorphism of
$\HV^{cvx} (\Val_\bullet Z)$ onto $\Pred_\AN^\bullet Z$.

We also note the following.  In a semitopological cone, the closure of
a convex subsets is convex \cite[Lemma~4.10~(a)]{Keimel:topcones2},
and we obtain the following as an easy consequence.
\begin{fact}
  \label{fact:cl:conv}
  Given any convex subspace $Z$ of a semitopological cone, the closure
  of any convex subset of $Z$ in $Z$ is convex.
\end{fact}

\begin{lemma}
  \label{lemma:lambda:flat:muS}
  For every object $X$ of $\Topcat^\flat$,
  ${\lambda^\flat_X} \circ \Val_\bullet {\mu^\Hoare_X} =
  \mu^\Hoare_{\Val_\bullet X} \circ \HV {\lambda^\flat_X} \circ
  \lambda^\flat_{\HV X}$.
\end{lemma}
The presence of the term $\lambda^\flat_{\HV X}$ is why we require
$\Topcat^\flat$ to be closed under $\HV$, see
Definition~\ref{defn:Top:flat}.

\begin{proof}
  Since ${\lambda^\flat_X} = s_\AN^\bullet \circ \Psi$ and the image
  of $s_\AN^\bullet$ consists of convex subsets, the left-hand side
  takes its values in convex sets.  For the right-hand side, for every
  $\xi \in \Val_\bullet {\HV {\HV X}}$,
  $(\mu^\Smyth_{\Val_\bullet X} \circ \HV {\lambda^\flat_X} \circ
  \lambda^\flat_{\HV X}) (\xi)$ is the closure of
  $\bigcup (\HV {\lambda^\flat_X} \circ \lambda^\flat_{\HV X}) (\xi) =
  \bigcup cl (\{\lambda^\flat_X (\mu) \mid \mu \in \lambda^\flat_{\HV X}
  (\xi)\})$, hence also of
  $\bigcup \{\lambda^\flat_X (\mu) \mid \mu \in \lambda^\flat_{\HV X}
  (\xi)\}$.  In order to show that
  $(\mu^\Smyth_{\Val_\bullet X} \circ \HV {\lambda^\flat_X} \circ
  \lambda^\flat_{\HV X}) (\xi)$ is convex, it suffices to show that
  $A \eqdef \bigcup \{\lambda^\flat_X (\mu) \mid \mu \in \lambda^\flat_{\HV X}
  (\xi)\}$ is, by Fact~\ref{fact:cl:conv}.
  
  Let $a \in [0, 1]$, and $\nu_1, \nu_2 \in A$.  there is a
  $\mu_1 \in \lambda^\flat_{\HV X} (\xi)$ such that
  $\nu_1 \in {\lambda^\flat_X} (\mu_1)$ and there is a
  $\mu_2 \in \lambda^\flat_{\HV X} (\xi)$ such that
  $\nu_2 \in {\lambda^\flat_X} (\mu_2)$.  We note that
  $\lambda^\flat_{\HV X} (\xi) = (s_\AN^\bullet \circ \Psi) (\xi)$ is
  convex, so $a\mu_1 + (1-a)\mu_2 \in \lambda^\flat_{\HV X} (\xi)$.
  Additionally,
  $a\nu_1 + (1-a)\nu_2 \in {\lambda^\flat_X} (a\mu_1 + (1-a)\mu_2)$:
  for every $U \in \Open X$, $\nu_1 (U) \leq \mu_1 (\Diamond U)$ and
  $\nu_2 (U) \leq \mu_2 (\Diamond U)$ by definition of
  ${\lambda^\flat_X}$, so
  $(a\nu_1 + (1-a)\nu_2) (U) \leq (a\mu_1 + (1-a) \mu_2) (\Diamond
  U)$.  Therefore there is a $\mu \in \lambda^\flat_{\HV X} (\xi)$,
  namely $\mu \eqdef a\mu_1 + (1-a) \mu_2$, such that
  $a \nu_1 + (1-a) \nu_2 \in {\lambda^\flat_X} (\mu)$.  Equivalently,
  $a \nu_1 + (1-a) \nu_2$ is in $A$.  Hence $A$ is convex, as
  promised.

  It follows that we can use Trick~C, and so we prove the following
  instead:
  \begin{align*}
    \Psi \circ \Val_\bullet {\mu^\Hoare_X}
    & = r_\AN \circ  \mu^\Hoare_{\Val_\bullet X} \circ \HV {\lambda^\flat_X} \circ \lambda^\flat_{\HV X}.
  \end{align*}
  Indeed,
  $r_\AN \circ {\lambda^\flat_X} = r_\AN \circ s_\AN^\bullet \circ \Psi =
  \Psi$, which allows us to simplify the left-hand side.

  We now use Trick~A:
  \begin{align*}
    [h > 1]
    & \invto{\Psi} [h_* > 1]
    & \text{by Lemma~\ref{lemma:Psi}} \\
    & \invto{\Val_\bullet {\mu^\Hoare_X}}
      [h_* \circ \mu^\Hoare_X > 1] = [h_{**} > 1],
  \end{align*}
  where for every $\mathcal C \in \HV {\HV X}$,
  $h_{**} (\mathcal C) = \sup_{C \in \mathcal C} h_* (C) = \sup_{C \in
    \mathcal C} \sup_{x \in C} h (x)$; indeed, the latter is equal to
  $\sup_{x \in \bigcup \mathcal C} h (x) = \sup_{x \in cl (\bigcup
    \mathcal C)} h (x)$ (by Fact~\ref{fact:sup:cl})
  $= h_* (\mu^\Hoare_X (\mathcal C))$.
  \begin{align*}
    [h > 1]
    & \invto{r_\AN}
      \Diamond {[h > 1]}
    & \text{by Fact~\ref{fact:r:cont:flat}} \\
    & \invto{\mu^\Hoare_{\Val_\bullet X}}
      \Diamond {\Diamond {[h > 1]}} \\
    & \invto{\HV {{\lambda^\flat_X}}}
      \Diamond {{\lambda^\flat_X}^{-1} (\Diamond {[h > 1]})}
      = \Diamond {[h_* > 1]}
    & \text{by Corollary~\ref{corl:lambda:flat:inv:func:spec}}\\
    & \invto{\lambda^\flat_{\HV X}}
      [h_{**} > 1]
    & \text{by Corollary~\ref{corl:lambda:flat:inv:func:spec}}.
  \end{align*}
  \qed
\end{proof}

\begin{lemma}
  \label{lemma:lambda:flat:muT}
  For every object $X$ of $\Topcat^\flat$,
  ${\lambda^\flat_X} \circ \mu_{\HV X} = \HV {\mu_X} \circ
  \lambda^\flat_{\Val_\bullet X} \circ \Val_\bullet
  {\lambda^\flat_X}$.
\end{lemma}
The term $\lambda^\flat_{\Val_\bullet X}$ is why we require
$\Topcat^\flat$ to be closed under $\Val_\bullet$ (see
Definition~\ref{defn:Top:flat}).

\begin{proof}
  The left-hand side takes its values in convex sets.  We claim that
  the right-hand side does, too.  For every
  $\xi \in \Val_\bullet {\Val_\bullet {\HV X}}$,
  $(\HV {\mu_X} \circ \lambda^\flat_{\Val_\bullet X} \circ
  \Val_\bullet {\lambda^\flat_X}) (\xi)$ is the closure of
  $A \eqdef \bigcup \{\mu_X (\mu) \mid \mu \in
  \lambda^\flat_{\Val_\bullet X} ({\lambda^\flat_X} [\xi])\}$ so it
  suffices to show that $A$ is convex, and to rely on
  Fact~\ref{fact:cl:conv}.  For all $a \in [0, 1]$,
  $\mu_1, \mu_2 \in \lambda^\flat_{\Val_\bullet X} ({\lambda^\flat_X}
  [\xi])$, $\mu \eqdef a \mu_1 + (1-a) \mu_2$ is in
  $\lambda^\flat_{\Val_\bullet X} ({\lambda^\flat_X} [\xi])$, since
  $\lambda^\flat_{\Val_\bullet X}$ takes its values in convex sets.
  Then
  $\mu_X (\mu) = (\mathcal U \in \Open {\Val_\bullet X} \mapsto \int_{\nu \in
    \Val_\bullet X} \nu (\mathcal U) \,d\mu) = a \mu_X (\mu_1) + (1-a) \mu_X
  (\mu_2)$ is in $A$, showing that $A$ is convex.

  We can therefore use Trick~C, and since $r_\AN \circ {\lambda^\flat_X} =
  r_\AN \circ s_\AN^\bullet \circ \Psi = \Psi$, it remains to show:
  \begin{align*}
    \Psi \circ \mu_{\HV X}
    & = r_\AN \circ \HV {\mu_X} \circ \lambda^\flat_{\Val_\bullet X} \circ \Val_\bullet {\lambda^\flat_X}.
  \end{align*}
  To this end, we use Trick~A:
  \begin{align*}
    [h > 1]
    & \invto{\Psi}
      [h_* > 1]
    & \text{by Lemma~\ref{lemma:Psi}} \\
    & \invto{\mu_{\HV X}}
      \left[\left(\mu \in \Val_\bullet {\HV X} \mapsto \int_{C \in \HV X} h_*
      (C) \,d\mu\right) > 1\right]
    & \text{by (\ref{eq:mu-1})}
  \end{align*}
  while:
  \begin{align*}
    [h > 1]
    & \invto{r_\AN}
      \Diamond {[h > 1]}
    & \text{by Fact~\ref{fact:r:cont:flat}} \\
    & \invto{\HV {\mu_X}}
      \Diamond {\mu_X^{-1} ([h > 1])}
      = \Diamond {\left[\left(\nu \in \Val_\bullet X \mapsto \int_{x \in X} h (x)
      \,d\nu\right) > 1\right]}
      \mskip-80mu
    & \text{by (\ref{eq:mu-1})} \\
    & \invto{\lambda^\flat_{\Val_\bullet X}}
      \left[\left(\nu \in \Val_\bullet X \mapsto \int_{x \in X} h (x)
      \,d\nu\right)_* > 1\right]
    & \text{by Corollary~\ref{corl:lambda:flat:inv:func:spec}} \\
    & \invto{\Val_\bullet {\lambda^\flat_X}}
      \left[\left(\nu \in \Val_\bullet X \mapsto \int_{x \in X} h (x)
      \,d\nu\right)_* \circ {\lambda^\flat_X} > 1\right].
  \end{align*}
  In order to conclude, it suffices to show that:
  \begin{align*}
    \left(\nu \in \Val_\bullet X \mapsto \int_{x \in X} h (x)
    \,d\nu\right)_* \circ {\lambda^\flat_X}
    & = (\mu \in \Val_\bullet {\HV X} \mapsto \int_{C \in \HV X} h_*
      (C) \,d\mu).
  \end{align*}
  The left-hand side maps every $\mu \in \Val_\bullet {\HV X}$ to
  \begin{align*}
    \sup_{\nu \in s_\AN^\bullet (\Psi (\mu))} \int_{x \in X} h (x)
    \,d\nu
    & = r_\AN (s_\AN^\bullet (\Psi (\mu))) (h) \\
    & = \Psi (\mu) (h) = \int_{C \in \HV X} h_*
      (C),
  \end{align*}
  and this completes the proof.  \qed
\end{proof}

Putting Lemma~\ref{lemma:lambda:flat:nat}, Lemma~\ref{lemma:lambda:flat:etaS},
Lemma~\ref{lemma:lambda:flat:muS} and Lemma~\ref{lemma:lambda:flat:muT}
together, we obtain the following.
\begin{theorem}
  \label{thm:lambda:flat}
  Let $\bullet$ be nothing, ``$\leq 1$'', or ``$1$''.  Then
  $\lambda^\flat$ is a weak distributive law of $\HV$ over
  $\Val_\bullet$ on any category $\Topcat^\flat$ as given in
  Definition~\ref{defn:Top:flat}.
\end{theorem}

\section{The associated weak composite monad}
\label{sec:associated-monad:flat}

We proceed as in Section~\ref{sec:associated-monad}.
We consider the following idempotent instead of (\ref{eq:idemp}):
\begin{equation}
  \label{eq:idemp:flat}
  \xymatrix{
    \HV {\Val_\bullet X} \ar[r]^{\eta_{\HV {\Val_\bullet X}}}
    & \Val_\bullet {\HV {\Val_\bullet X}}
    \ar[r]^{\lambda^\flat_{\Val_\bullet X}}
    & \HV {\Val_\bullet {\Val_\bullet X}}
    \ar[r]^{\HV {\mu_X}}
    & \HV {\Val_\bullet X}
  }
\end{equation}

\begin{lemma}
  \label{lemma:idemp:split:flat}
  For every object $X$ of $\Topcat^\flat$, a splitting of the
  idempotent (\ref{eq:idemp:flat}) is
  $\xymatrix{\HV {\Val_\bullet X} \ar[r]^{r_\AN} & \Pred_\AN^\bullet X
    \ar[r]^{s_\AN^\bullet} & \HV {\Val_\bullet X}}$.
\end{lemma}
We note that the middle object $\Pred_\AN^\bullet X$ occurs as a
retract of $\HV {\Val_\bullet X}$ in $\Topcat$, through $r_\AN$.
Hence it is in $\Topcat^\flat$, since $\Topcat^\flat$ is closed under
$\HV$, $\Val_\bullet$ and retracts (Definition~\ref{defn:Top:flat}).

\begin{proof}
  Let $f$ denote the idempotent (\ref{eq:idemp:flat}).  We recall that
  $\lambda^\flat_{\Val_\bullet X} = s_\AN^\bullet \circ \Psi$.  (And
  we remember that $\Val_\bullet X$ is in $\Topcat^\flat$, since
  $\Topcat^\flat$ is closed under $\Val_\bullet$.)  In particular, any
  element in its range is convex.  The element of any convex element
  $C$ of a space $\HV {\Val_\bullet}$ by $\HV f$, where $f$ is any
  linear continuous map from $\Val_\bullet Y \to \Val_\bullet Z$, is
  convex; and $\mu_X$ is linear.  Therefore $f$ takes its values in
  $\HV^{cvx} {\Val_\bullet X}$.

  Using Trick~C, and simplifying using the fact that $r_\AN \circ
  s_\AN^\bullet$ is the identity, it remains to show that $r_\AN$ is
  equal to $r_\AN \circ f$, namely to:
  \[
    \xymatrix{
      \HV {\Val_\bullet X} \ar[r]^{\eta_{\HV {\Val_\bullet X}}}
      & \Val_\bullet {\HV {\Val_\bullet X}}
      \ar[r]^{\lambda^\flat_{\Val_\bullet X}}
      & \HV {\Val_\bullet {\Val_\bullet X}}
      \ar[r]^{\HV {\mu_X}}
      & \HV {\Val_\bullet X}
      \ar[r]^{r_\AN}
      & \Pred_\AN^\bullet X
    }
  \]
  Using Trick~A,
  \begin{align*}
    [h > 1]
    & \invto{r_\AN}
      \Diamond {[h > 1]}
    & \text{by Fact~\ref{fact:r:cont:flat}} \\
    & \invto{\HV {\mu_X}}
      \Diamond {\mu_X^{-1} ([h > 1])} \\
    & \qquad= \Diamond {\left[\left(\nu \in \Val_\bullet X \mapsto \int_{x \in
      X} h (x) \,d\nu\right) > 1\right]}
    & \text{by (\ref{eq:mu-1})} \\
    & \invto{\lambda^\flat_{\Val_\bullet X}}
      \left[\left(\nu \in \Val_\bullet X \mapsto \int_{x \in
      X} h (x) \,d\nu\right)_* > 1\right]
    & \text{by Corollary~\ref{corl:lambda:flat:inv:func:spec}} \\
    & \invto{\eta_{\HV {\Val_\bullet X}}}
      \left\{C \in \HV {\Val_\bullet X} \mid \left(\nu \in \Val_\bullet X \mapsto \int_{x \in
      X} h (x) \,d\nu\right)_* (C) > 1\right\}
    \mskip-360mu \\
    & \qquad= \left\{C \in \HV {\Val_\bullet X} \mid r_\AN (C) (h) > 1 \right\}
      = r_\AN^{-1} ([h > 1]).
      \mskip-250mu
  \end{align*}
  \qed
\end{proof}

Generalizing (\ref{eq:idemp:flat}), there is an idempotent:
\begin{equation}
  \label{eq:idemp:alpha:flat}
  \xymatrix{
    \HV X \ar[r]^{\eta_{\HV X}}
    & \Val_\bullet {\HV X}
    \ar[r]^{\lambda^\flat_X}
    & \HV {\Val_\bullet X}
    \ar[r]^{\HV \alpha}
    & \HV X
  }
\end{equation}
for every $\Val_\bullet$-algebra $\alpha$.  Finding out what its
splitting is is harder in general, especially because we have no
complete characterization of $\Val_\bullet$-algebras in $\Topcat$.  As
before, we will only need to compute this splitting for one additional
$\Val_\bullet$-algebra.
\begin{lemma}
  \label{lemma:idemp:P:flat}
  For every topological space $X$, there is a $\Val_\bullet$-algebra
  $\alpha_\AN^\bullet \colon \Val_\bullet {\Pred_\AN^\bullet X} \to
  \Pred_\AN^\bullet X$ defined by
  $\alpha_\AN^\bullet (\xi) (h) \eqdef \int_{F \in \Pred_\AN^\bullet
    X} F (h) \,d\xi$ for all
  $\xi \in \Val_\bullet {\Pred_\DN^\bullet X}$ and $h \in \Lform X$.

  If $X$ is an object of $\Topcat^\flat$, then a splitting of the
  idempotent (\ref{eq:idemp:alpha:flat}) where
  $\alpha \eqdef \alpha_\AN^\bullet$ is
  \begin{align*}
    \xymatrix{\HV {\Pred_\AN^\bullet X} \ar[r]
    & \HV^{cvx} {\Pred_\AN^\bullet X} \ar[r]
    & \HV {\Pred_\AN^\bullet X}},
  \end{align*}
  where the arrow on the right is the inclusion map, and the arrow on
  the left is the corestriction of (\ref{eq:idemp:alpha:flat}), and
  maps every $C \in \HV^{cvx} {\Pred_\AN^\bullet X}$ to itself.
\end{lemma}
In the discussion after Lemma~\ref{lemma:idemp:split:flat}, we have
noted that $\Pred_\AN^\bullet X$ is an object of $\Topcat^\flat$, for
every object $X$ of $\Topcat^\flat$.  Since $\Topcat^\flat$ is closed
under $\HV$ and under retracts (see Definition~\ref{defn:Top:flat}),
the middle object $\HV^{cvx} {\Pred_\AN^\bullet X}$ is in
$\Topcat^\flat$, too.

\begin{proof}
  The map $F \mapsto F (h)$ is lower semicontinuous, since the inverse
  image of $]t, \infty]$ is $[h > t]$, for all $h \in \Lform X$ and
  $t \in \Rp$.  Hence the integral
  $\int_{F \in \Pred_\AN^\bullet X} F (h) \,d\xi$ makes sense.  The
  map ${\alpha_\AN^\bullet} (\xi)$ is Scott-continuous because each
  $F \in \Pred_\AN^\bullet X$ is and because integration with respect
  to $\nu$ is Scott-continuous; it is sublinear because $F$ is and
  integration with respect to $\xi$ is linear.  Additionally,
  ${(\alpha_\AN^\bullet)}^{-1} ([h > r]) = [(F \mapsto F (h)) > r]$,
  so ${\alpha_\AN^\bullet}$ is continuous.

  We claim that the map $\alpha_\AN^\bullet$ is a
  $\Val_\bullet$-algebra.  We have
  $\alpha_\AN^\bullet (\eta_{\Pred_\AN^\bullet X} (F)) (h) = \int_{F'
    \in \Pred_\AN^\bullet X} F' (h)\,d\delta_F = F (h)$, so
  $\alpha_\AN^\bullet \circ \eta_{\Pred_\AN^\bullet X}$ is the
  identity map.  In order to show that
  $\alpha_\AN^\bullet \circ \mu_{\Pred_\AN^\bullet X} =
  \alpha_\AN^\bullet \circ \Val_\bullet {\alpha_\AN^\bullet}$, we use
  Trick~A:
  \begin{align*}
    [h > 1]
    & \invto{\alpha_\AN^\bullet}
      [(F \in \Pred_\AN^\bullet X \mapsto F (h)) > 1] \\
    & \invto{\mu_{\Pred_\AN^\bullet X}}
      [(\xi \in \Val_\bullet {\Pred_\AN^\bullet X} \mapsto
      \int_{F \in \Pred_\AN^\bullet X} F (h) \,d\xi) > 1]
    & \text{by (\ref{eq:mu-1})} \\
    [h > 1]
    & \invto{\alpha_\AN^\bullet}
      [(F \in \Pred_\AN^\bullet X \mapsto F (h)) > 1] \\
    & \invto{\Val_\bullet {\alpha_\AN^\bullet}}
      [(F \in \Pred_\AN^\bullet X \mapsto F (h)) \circ
      \alpha_\AN^\bullet > 1] \\
    & = [(\xi \in \Val_\bullet {\Pred_\AN^\bullet X} \mapsto
      \alpha_\AN^\bullet (\xi) (h)) > 1] \\
    & = [(\xi \in \Val_\bullet {\Pred_\AN^\bullet X} \mapsto
      \int_{F \in \Pred_\AN^\bullet X} F (h) \,d\xi) > 1].
  \end{align*}
  In order to justify the announced splitting of
  (\ref{eq:idemp:alpha:flat}), we need to show that the image of
  (\ref{eq:idemp:alpha:flat}) is exactly $\HV^{cvx} {\Pred_\AN^\bullet X}$.
  For every $C \in \HV X$, the image of $C$ by (\ref{eq:idemp:alpha:flat})
  is $cl (\alpha_\AN^\bullet [C'])$ where
  $C' \eqdef \lambda^\flat_{\Pred_\AN^\bullet X} (\delta_C)$ is
  convex, and $\alpha_\AN^\bullet$ is linear (in its $\xi$ argument),
  so that image is convex, using Fact~\ref{fact:cl:conv}.

  In order to see that every element $C$ of
  $\HV^{cvx} {\Pred_\AN^\bullet X}$ is reached as the image of some
  element by (\ref{eq:idemp:alpha:flat}), we simply show that $C$ is its
  own image by (\ref{eq:idemp:alpha:flat}), or equivalently, that
  (\ref{eq:idemp:alpha:flat}) restricts to the identity on
  $\HV^{cvx} {\Pred_\AN^\bullet X}$.  We use Trick~A:
  \begin{align*}
    \Diamond {\bigcap_{i=1}^n [h_i > 1]}
    & \invto{\HV {\alpha_\AN^\bullet}}
      \Diamond {\bigcap_{i=1}^n {(\alpha_\AN^\bullet)}^{-1} ([h_i > 1])}
    \\
    & = \Diamond {\bigcap_{i=1}^n [(F \mapsto F (h_i)) > 1]} \\
    & \invto{\lambda^\flat_{\Pred_\AN^\bullet X}}
      \bigcap_{\vec a \in \Delta_n} \left[\left(F \mapsto \sum_{i=1}^n
      a_i F (h_i)\right)_* > 1\right]
    & \text{by Proposition~\ref{prop:lambda:flat}} \\
    & \invto{\eta_{\HV {\Pred_\AN^\bullet X}}}
      \bigcap_{\vec a \in \Delta_n}  \left(\left(F \mapsto \sum_{i=1}^n
      a_i F (h_i)\right)_*\right)^{-1} (]1, \infty]),
      \mskip-50mu
  \end{align*}
  and we claim that the convex elements $C$ of $\HV {\Pred_\AN^\bullet
    X}$ that are in the latter open set are exactly those of $\Diamond
  {\bigcup_{i=1}^n [h_i > 1]}$, which will show the claim.

  For every $C \in \HV {\Pred_\AN^\bullet X}$,
  $C \in \bigcap_{\vec a \in \Delta_n} \left(\left(F \mapsto
      \sum_{i=1}^n a_i F (h_i)\right)_*\right)^{-1} (]1, \infty])$ if
  and only for every $\vec a \in \Delta_n$,
  $\sup_{F \in C} \sum_{i=1}^n a_i F (h_i) > 1$.  Let
  $f \colon \Delta_n \times C \to \creal$ map $(\vec a, F)$ to
  $\int_{x \in X} \sum_{i=1}^n a_i F (h_i)$.  This is a lower
  semicontinuous map in its first argument, $\Delta_n$ is compact (and
  Hausdorff), as in the proof of Lemma~\ref{lemma:sAP:inv:func}, and
  $f$ preserves pairwise linear combinations (with coefficients $a$
  and $1-a$, $a \in [0, 1]$) of its first arguments, and similarly
  with its second arguments (which exist because $C$ is convex), so by
  Remark~3.4 and the Minimax Theorem~3.3 of \cite{JGL-minimax17},
  $\sup_{F \in C} \inf_{\vec a \in \Delta_n} f (\vec a, F) =
  \min_{\vec a \in \Delta_n} \sup_{F \in C} f (\vec a, F)$.  But the
  elements $C \in \HV {\Pred_\AN^\bullet X}$ that are in
  $\bigcap_{\vec a \in \Delta_n} \left(\left(F \mapsto \sum_{i=1}^n
      a_i F (h_i)\right)_*\right)^{-1} (]1, \infty])$ are those such
  that $\min_{\vec a \in \Delta_n} \sup_{F \in C} f (\vec a, F) > 1$,
  hence such that
  $\sup_{F \in C} \inf_{\vec a \in \Delta_n} f (\vec a, F) > 1$, hence
  such that for some $F \in C$,
  $\inf_{\vec a \in \Delta_n} f (\vec a, F) > 1$.  Since
  $f (\vec a, F)$ is linear in $\vec a$ (and lower semicontinuous on a
  compact set), the latter inf is reached at one of the vertices of
  $\Delta_n$, so $\inf_{\vec a \in \Delta_n} f (\vec a, F) > 1$ is
  equivalent to $F (h_i) > 1$ for every $i \in \{1, \cdots, n\}$.
  Therefore, for every $C \in \HV {\Pred_\AN^\bullet X}$,
  $C \in \bigcap_{\vec a \in \Delta_n} \left(\left(F \mapsto
      \sum_{i=1}^n a_i F (h_i)\right)_*\right)^{-1} (]1, \infty])$ if
  and only if there is an $F \in C$ such that $F (h_i) > 1$ for every
  $i \in \{1, \cdots, n\}$, if and only if
  $C \in \Diamond {\bigcap_{i=1}^n [h_i > 1]}$.  \qed
\end{proof}

\begin{lemma}
  \label{lemma:Sbar:flat}
  For $S \eqdef \HV$, $T \eqdef \Val_\bullet$,
  $\lambda \eqdef \lambda^\flat$, the monad $\overline S$ on the
  category of $\Val_\bullet$-algebras over $\Topcat^\flat$ has the
  following properties:
  \begin{enumerate}
  \item For every free $\Val_\bullet$-algebra
    $\mu_X \colon \Val_\bullet {\Val_\bullet X} \to \Val_\bullet X$,
    $\overline S \mu_X \colon \Val_\bullet {\Pred_\AN^\bullet X} \to
    \Pred_\AN^\bullet X$ is $\alpha_\AN^\bullet$ (see
    Lemma~\ref{lemma:idemp:P:flat}).
  \item The unit $\eta^{\overline S}_{\mu_X}$ evaluated at the free
    $\Val_\bullet$-algebra
    $\mu_X \colon \Val_\bullet {\Val_\bullet X} \to \Val_\bullet X$ is
    $r_\AN \circ \eta^\Hoare_{\Val_\bullet X}$.
  \item The multiplication $\mu^{\overline S}_{\mu_X}$ evaluated at
    the free $\Val_\bullet$-algebra
    $\mu_X \colon \Val_\bullet {\Val_\bullet X} \to \Val_\bullet X$ is
    the function that maps every $C \in \HV^{cvx} {\Pred_\AN^\bullet
      X}$ to $(h \in \Lform X \mapsto \sup_{F \in C} F (h)) \in
    \Pred_\AN^\bullet X$; the inverse image of $[h > 1]$ by that
    function is $\Diamond {[h > 1]} \cap \HV^{cvx} {\Pred_\AN^\bullet X}$,
    for every $h \in \Lform X$.
  \item The counit $\epsilon^T$ of the adjunction $F^T \dashv U^T$
    between $\Topcat$ and the category of $T$-algebras
    ($\Val_\bullet$-algebras) is given at each algebra $\beta \colon
    \Val_\bullet X \to X$ as $\beta$ itself.
  \end{enumerate}
\end{lemma}
\begin{proof}
  1. We know that $\overline S \mu_X$ must be a $\Val_\bullet$-algebra
  on $\overline S {\Val_\bullet X}$, and the latter can be taken as
  $\Pred_\AN^\bullet X$, by Lemma~\ref{lemma:idemp:split}.  We know
  that
  $\overline S \mu_X = \pi_{\mu_X} \circ S \mu_X \circ \lambda_X \circ
  T\iota_{\mu_X} = r_\AN \circ \HV \mu_X \circ {\lambda^\flat_X} \circ
  \Val_\bullet {s_\AN^\bullet}$, which we elucidate using Trick~A:
  \begin{align*}
    [h > 1]
    & \invto{r_\AN}
      \Diamond {[h > 1]}
    & \text{by Fact~\ref{fact:r:cont:flat}} \\
    & \invto{\HV {\mu_X}}
      \Diamond {\mu_X^{-1} ([h > 1])}
      = \Diamond {\left[\left(\nu \in \Val_\bullet X \mapsto \int_{x \in X} h (x)
      \,d\nu\right) > 1\right]}
      \mskip-80mu
    & \text{by (\ref{eq:mu-1})} \\
    & \invto{\lambda^\flat_{\Val_\bullet X}}
      \left[\left(\nu \in \Val_\bullet X \mapsto \int_{x \in X} h (x)
      \,d\nu\right)_* > 1\right]
    & \text{by Corollary~\ref{corl:lambda:flat:inv:func:spec}} \\
    & \invto{\Val_\bullet {s_\AN^\bullet}}
      \left[\left(\nu \in \Val_\bullet X \mapsto \int_{x \in X} h (x)
      \,d\nu\right)_* \circ s_\AN^\bullet > 1\right] \\
    & = [(F \in \Pred_\AN^\bullet X \mapsto F (h)) > 1].
  \end{align*}
  The last equality is justified as follows.  For every
  $F \in \Pred_\AN^\bullet X$,
  \begin{align*}
    \left(\left(\nu \in \Val_\bullet X \mapsto \int_{x \in X} h (x)
    \,d\nu\right)_* \circ s_\AN^\bullet\right) (F)
    & = \sup_{\nu \in s_\AN^\bullet (F)} \int_{x \in X} h (x)
      \,d\nu \\
    & = r_\AN (s_\AN^\bullet (F)) (h) = F (h).
  \end{align*}
  Finally, $[(F \in \Pred_\AN^\bullet X \mapsto F (h)) > 1]$ is the
  collection of continuous valuations $\xi \in \Val_\bullet
  {\Pred_\AN^\bullet X}$ such that $\int_{F \in \Pred_\AN^\bullet X} F
  (h) \,d\xi > 1$, namely ${(\alpha_\AN^\bullet)}^{-1} ([h > 1])$.

  2.
  $\eta^{\overline S}_{\mu_X} = \pi_{\mu_X} \circ \eta^S_{\Val_\bullet
    X} = r_\AN \circ \eta^\Hoare_{\Val_\bullet X}$.

  3. In general, $\mu^{\overline S}_\alpha$ is a morphism from
  $\overline S {\overline S \alpha}$ to $\overline S \alpha$ in the
  category of $\Val_\bullet$-algebras.  When
  $\alpha = \mu_X \colon \Val_\bullet {\Val_\bullet X} \to
  \Val_\bullet X$, it is therefore one from
  $\Pred_\AN^\bullet {\Pred_\AN^\bullet X}$ to $\Pred_\AN^\bullet X$.
  We have
  $\mu^{\overline S}_{\mu_X} \eqdef \pi_{\mu_X} \circ
  \mu^S_{\Val_\bullet X} \circ S \iota_{\mu_X} \circ \iota_{\overline
    S \mu_X}$.  We may take
  $\iota_{\overline S \mu_X} = \iota_{\alpha_\AN^\bullet}$ to be the
  right arrow, from $\HV^{cvx} {\Pred_\AN^\bullet X}$ to
  $\HV {\Pred_\AN^\bullet X}$ in Lemma~\ref{lemma:idemp:P:flat}---just
  subspace inclusion.  Then
  $S \iota_{\mu_X} = \HV {s_\AN^\bullet} \colon \HV {\Pred_\AN^\bullet
    X} \to \HV {\HV {\Val_\bullet X}}$, by
  Lemma~\ref{lemma:idemp:split}.  We compose that with
  $\mu^\Hoare_{\Val_\bullet X} \colon \HV {\HV {\Val_\bullet X}} \to
  \HV {\Val_\bullet X}$, then with
  $r_\AN \colon \HV {\Val_\bullet X} \to \Pred_\AN^\bullet X$.  We
  elucidate what this composition is by relying on Trick~A:
  \begin{align*}
    [h > 1]
    & \invto{r_\AN}
      \Diamond {[h > 1]}
    & \text{by Fact~\ref{fact:r:cont:flat}} \\
    & \invto{\mu^\Hoare_{\Val_\bullet X}}
      \Diamond {\Diamond {[h > 1]}} \\
    & \invto{\HV {s_\AN^\bullet}}
      \Diamond {{(s_\AN^\bullet)}^{-1} (\Diamond {[h > 1]})} \\
    & = \Diamond {[h > 1]}
    & \text{by Corollary~\ref{corl:lambda:flat:inv:func:spec}} \\
    & \invto{\iota_{\alpha_\AN^\bullet}}
      \Diamond {[h > 1]} \cap \HV^{cvx} {\Pred_\AN^\bullet X}.
  \end{align*}
  Now this looks just like $r_\AN$, and we imitate it by letting
  $r \colon \HV^{cvx} {\Pred_\AN^\bullet X} \to \Pred_\AN^\bullet X$
  be defined by $r (C) (h) \eqdef \sup_{F \in C} F (h)$.  This is
  well-defined and lower semicontinuous, and
  $r^{-1} ([h > 1]) = \Diamond {[h > 1]} \cap \HV^{cvx} {\Pred_\AN^\bullet
    X}$, so $\mu^{\overline S}_{\mu_X} = r$.

  4. Standard category theory, see \cite[Chapter~VI, Section~2,
  Theorem~1]{McLane:cat:math} for example.  \qed
\end{proof}

Finally, the weak composite monad of $S$ with $T$ is
$U^T \overline S F^T$, with unit
$U^T \eta^{\overline S} F^T \circ \eta^T$ and multiplication
$U^T \mu^{\overline S} F^T \circ U^T \overline S \epsilon^T \overline
S F^T$, where $F^T \dashv U^T$ is the adjunction between the base
category ($\Topcat$ in our case) and the category of $T$-algebras, and
$\epsilon^T$ is its counit.

\begin{theorem}
  \label{thm:weaklift:flat}
  The weak composite monad associated with
  $\lambda^\flat \colon \Val_\bullet {\HV} \to \HV {\Val_\bullet}$ on
  $\Topcat^\flat$ (see Definition~\ref{defn:Top:flat}) is
  $(\Pred_\AN^\bullet, \eta^\AN, \mu^\AN)$.
\end{theorem}
\begin{proof}
  We take $T \eqdef \Val_\bullet$, $S \eqdef \HV$,
  $\lambda \eqdef \lambda^\flat$.  The functor $U^T$ takes every
  $\Val_\bullet$-algebra $\alpha \colon \Val_\bullet X \to X$ to $X$
  and every $\Val_\bullet$-algebra morphism
  $f \colon (\alpha \colon \Val_\bullet X \to X) \to (\beta \colon
  \Val_\bullet Y \to Y)$ to the underlying continuous map
  $f \colon X \to Y$.  The functor $F^T$ takes every space $X$ to
  $\mu_X \colon \Val_\bullet {\Val_\bullet X} \to \Val_\bullet X$, and
  every continuous map $f \colon X \to Y$ to $\Val_\bullet f$, seen as
  a morphism of $\Val_\bullet$-algebras.

  Then $U^T \overline S F^T$ maps every space $X$ to the splitting
  $\Pred_\AN^\bullet X$ of
  $\HV {\mu_X} \circ {\lambda^\flat_X} \circ \eta_{\HV X}$ we have
  obtained in Lemma~\ref{lemma:idemp:split:flat}.  We also obtain
  $\pi_{\mu_X} \eqdef r_\AN$, $\iota_{\mu_X} \eqdef s_\AN^\bullet$.

  Given any continuous map $f \colon X \to Y$, $U^T \overline S F^T
  (f)$ is equal to (the map underlying the $\Val_\bullet$-algebra)
  $\overline S (\Val_\bullet f)$, namely $\pi_{\mu_Y} \circ \HV
  {\Val_\bullet f} \circ \iota_{\mu_X} = r_\AN \circ \HV
  {\Val_\bullet f} \circ s_\AN^\bullet$.  We claim that this is equal
  to $\Pred_\AN^\bullet (f)$, namely that it maps every $F \in
  \Pred_\AN^\bullet X$ to $(h \in \Lform X \mapsto F (f \circ h))$.
  We use Trick~A:
  \begin{align*}
    [h > 1]
    & \invto{r_\AN}
      \Diamond {[h > 1]}
    & \text{by Fact~\ref{fact:r:cont:flat}} \\
    & \invto{\HV {\Val_\bullet f}}
      \Diamond {(\Val_\bullet f)^{-1} ([h > 1])}
      = \Diamond {[h \circ f > 1]} \\
    & \invto{s_\AN^\bullet}
      [h \circ f > 1]
    & \text{by Corollary~\ref{corl:lambda:flat:inv:func:spec}} \\
    & = {(\Pred_\AN^\bullet (f))}^{-1} ([h > 1]).
  \end{align*}
  From Lemma~\ref{lemma:Sbar:flat}, item~2, the unit of $\overline S$
  at the algebra $\mu_X$ is
  $\eta^{\overline S}_{\mu_X} = r_\AN \circ \eta^\Hoare_{\Val_\bullet
    X}$, which is the map
  $\nu \in \Val_\bullet X \mapsto r_\AN (\dc \nu)$, namely
  $\nu \in \Val_\bullet X \mapsto (h \in \Lform X \mapsto \int_{x \in
    X} h (x) \,d\nu)$ from $\Val_\bullet X$ to $\Pred_\AN^\bullet X$.

  The unit of the weak composite monad is
  $U^T \eta^{\overline S} F^T \circ \eta^T$.  At object $X$, this maps
  every $x \in X$ to
  $(h \in \Lform X \mapsto \int_{x' \in X} h (x') \,d\delta_x)$,
  namely to $(h \in \Lform X \mapsto h (x))$, and that is exactly the
  unit of the $\Pred_\AN^\bullet$ monad.

  The multiplication is
  $U^T \mu^{\overline S} F^T \circ U^T \overline S \epsilon^T
  \overline S F^T$.  We recall what that means.  At any object $X$,
  $\overline S F^T X$ is the $T$-algebra
  $\overline S\mu_X \colon T \overline S T X \to \overline S T X$.
  Then $\epsilon^T$ evaluated at this $T$-algebra is the bottom arrow
  in the following diagram:
  \[
    \xymatrix@C+10pt{
      TT\overline S T X
      \ar[d]_{\mu_{\overline S T X}}
      \ar[r]^{T \epsilon^T_{\overline S F^T X}}
      & T \overline S T X
      \ar[d]^{\overline S \mu_X}
      \\
      T \overline S T X
      \ar[r]_{\epsilon^T_{\overline S F^T X}}
      & \overline S T X
    }
  \]
  from the $T$-algebra $\alpha \eqdef \mu_{\overline S T X}$ that is
  the leftmost vertical arrow to the $T$-algebra
  $\beta \eqdef \overline S \mu_X$ on the right.  We apply $\overline
  S$ to the morphism $\epsilon^T_{\overline S F^T X}$ (i.e., we now
  apply $\overline S$ to a morphism of $T$-algebras, not to a
  $T$-algebra, as we did before), and we obtain the composition:
  \begin{equation}
    \label{eq:mu:part:flat}
    \xymatrix{
      \overline S T \overline S T X
      \ar[r]^{\iota_\alpha}
      &
      S T \overline S T X
      \ar[r]^{S \epsilon^T_{\overline S F^T X}}
      &
      S \overline S T X
      \ar[r]^{\pi_\beta}
      &
      \overline S \overline S T X.
    }
  \end{equation}
  Applying $U^T$ to the latter, we obtain the same composition, this
  time seen as a morphism in the base category, instead of as a
  morphism in the category of $T$-algebras.  We finally compose it
  with $U^T \mu^{\overline S}_{F^T X}$, which is
  $\mu^{\overline S}_{\mu_X} \colon \overline S \overline S T X \to
  \overline S T X$.

  By Lemma~\ref{lemma:Sbar:flat}, item~4,
  $\epsilon^T_{\overline S F^T X} \colon \alpha \to \beta$ is $\beta$
  itself, where $\beta = \overline S \mu_X = \alpha_\AN^\bullet$
  (Lemma~\ref{lemma:Sbar:flat}, item~1) and
  $\alpha = \mu_{\overline S T X} = \mu_{\Pred_\AN^\bullet X}$: we
  have $\overline S T X = \Pred_\AN^\bullet X$, using
  Lemma~\ref{lemma:idemp:split:flat}, and $\iota_\alpha = s_\AN^\bullet$.
  Using Lemma~\ref{lemma:idemp:P:flat},
  $\pi_\beta \colon \HV {\Pred_\AN^\bullet X} \to \HV^{cvx}
  {\Pred_\AN^\bullet X}$ is the corestriction of
  (\ref{eq:idemp:alpha:flat}) (call it $f$), which maps every convex
  element of $\HV {\Pred_\AN^\bullet X}$ to itself.

  Hence the multiplication of the weak composite monad, evaluated at
  $X$, is the composition of (\ref{eq:mu:part:flat}) with
  $\mu^{\overline S}_{\mu_X}$ (given in Lemma~\ref{lemma:Sbar:flat},
  item~3), namely:
  \begin{equation}
    \label{eq:Sbarmu:flat}
    \xymatrix{
      \Pred_\AN^\bullet {\Pred_\AN^\bullet X}
      \ar[r]^{s_\AN^\bullet}
      & \HV {\Val_\bullet {\Pred_\AN^\bullet X}}
      \ar[r]^{\HV {\alpha_\AN^\bullet}}
      & \HV {\Pred_\AN^\bullet X}
      \ar[r]^f
      & \HV^{cvx} {\Pred_\AN^\bullet X}
      \ar[r]^{\mu^{\overline S}_{\mu_X}}
      & \Pred_\AN^\bullet X
      }
  \end{equation}
  For every $\mathcal F \in \Pred_\AN^\bullet {\Pred_\AN^\bullet X}$,
  $\HV {\alpha_\AN^\bullet} (s_\AN^\bullet (\mathcal F)) = cl
  (\{\alpha_\AN^\bullet (\xi) \mid \xi \in s_\AN^\bullet (\mathcal
  F)\}) = cl (\{(h \in \Lform X \mapsto \int_{F \in \Pred_\AN^\bullet
    X} F (h) \,d\xi) \mid \xi \in s_\AN^\bullet (\mathcal F)\})$.  It
  is easy to see that this is a convex set, because
  $s_\AN^\bullet (\mathcal F)$ is convex, integration is linear in the
  continuous valuation, and closure preserves convexity
  (Fact~\ref{fact:cl:conv}).  Therefore $f$ maps it to itself, and,
  using Lemma~\ref{lemma:Sbar:flat}, item~3,
  $\mu^{\overline S}_{\mu_X}$ maps it to:
  \begin{align*}
    (h \in \Lform X \mapsto \sup_{F \in \HV {\alpha_\AN^\bullet}
    (s_\AN^\bullet (\mathcal F))} F (h))
    & = (h \in \Lform X \mapsto \sup_{F \in cl (\alpha_\AN^\bullet
      [s_\AN^\bullet (\mathcal F)])}
      F (h)) \\
    & = (h \in \Lform X \mapsto \sup_{F \in \alpha_\AN^\bullet
      [s_\AN^\bullet (\mathcal F)]}
      F (h))
      \quad \text{by Fact~\ref{fact:sup:cl}} \\
    & = (h \in \Lform X \mapsto \sup_{\xi \in s_\AN^\bullet (\mathcal
      F)} \alpha_\AN^\bullet (\xi) (h)) \\
    & = (h \in \Lform X \mapsto \sup_{\xi \in s_\AN^\bullet (\mathcal
      F)} \int_{F \in \Pred_\AN^\bullet X} F (h) \,d\xi) \\
    & = (h \in \Lform X \mapsto  r_\AN (s_\AN^\bullet (\mathcal F)) (F \in \Pred_\AN^\bullet X
      \mapsto F (h))) \\
    & = (h \in \Lform X \mapsto  \mathcal F (F \in \Pred_\AN^\bullet X
      \mapsto F (h)),
  \end{align*}
  and this is exactly what the multiplication of the
  $\Pred_\AN^\bullet$ monad, evaluated at $X$ and applied to
  $\mathcal F$, produces.  \qed
\end{proof}

\part{A weak distributive law between $\Val_\bullet$ and $\QLV$,
  resp.\ $\Plotkinn$}

Combining $\SV$ and $\HV$, there is a monad of lenses and a monad of
quasi-lenses, which we define below.  We start with lenses, which are
probably more well-known.  Quasi-lenses are easier to work with.

A \emph{lens} is a non-empty set of the form $Q \cap C$ where $Q$ is
compact saturated and $C$ is closed in $X$.  The \emph{Vietoris
  topology} has subbasic open subsets of the form $\Box U$ (the set of
lenses included in $U$) and $\Diamond U$ (the set of lenses that
intersect $U$), for each open subset $U$ of $X$.  We write
$\Plotkin X$ for the set of lenses on $X$, and we let $\Plotkinn X$
denote $\Plotkin X$ with the Vietoris topology.  The specialization
ordering of $\Plotkinn X$ is the \emph{topological Egli-Milner
  ordering}: $L \TEMleq L'$ if and only if $\upc L \supseteq \upc L'$
and $cl (L) \subseteq cl (L')$ \cite[Discussion before
Fact~4.1]{GL:duality}.  This is an ordering, not just a preordering,
hence $\Plotkinn X$ is $T_0$.

A \emph{quasi-lens} on a topological space $X$ is a pair $(Q, C)$ of a
compact saturated subset $Q$ and a closed subset $C$ of $X$ such that:
\begin{enumerate}
\item $Q$ intersects $C$;
\item $Q \subseteq \upc (Q \cap C)$;
\item for every open neighborhood $U$ of $Q$, $C \subseteq cl (U
  \cap C)$.
\end{enumerate}
The notion originates from \cite[Theorem~9.6]{heckmann:2ndorder}, and
was rediscovered in \cite[Section~5]{GL:duality}.  We write $\QL X$
for the space of quasi-lenses on $X$.  The \emph{Vietoris topology} on
$\QL X$ is generated by the subbasic open sets
$\Box^\quasi U \eqdef \{(Q, C) \in \QL X \mid Q \subseteq U\}$ and
$\Diamond^\quasi U \eqdef \{(Q, C) \in \QL X \mid C \cap U \neq
\emptyset\}$.  We write $\QLV X$ for $\QL X$ with the Vietoris
topology.  Its specialization preordering is
$\supseteq \times \subseteq$, which is antisymmetric, so $\QLV X$ is
$T_0$ \cite[Lemma~8.2, item~2]{JGL:projlim:prev}; also, the inclusion
of $\QLV X$ into $\SV X \times \HV X$ is a topological embedding
\cite[Lemma~8.2, item~1]{JGL:projlim:prev}, and $\QLV$ extends to an
endofunctor on $\Topcat$, whose action on morphisms $f \colon X \to Y$
is defined by
$\QLV f (Q, C) \eqdef (\SV f (Q), \HV f (C)) = (\upc f [Q], cl (f
[C]))$, and is such that
${(\QLV f)}^{-1} (\Box^\quasi V) = \Box^\quasi {f^{-1} (V)}$ and
${(\QLV f)}^{-1} (\Diamond^\quasi V) = \Diamond^\quasi {f^{-1} (V)}$
for every $V \in \Open Y$ \cite[Lemma~8.3]{JGL:projlim:prev}.

The relationship between spaces of lenses and of quasi-lenses is as
follows.  For every topological space $X$, there is a topological
embedding $\iota_X \colon \Plotkinn X \to \QLV X$, defined by
$\iota_X (L) \eqdef (\upc L, cl (L))$.  There is a map
$\varrho_X \colon \QLV X \to \Plotkinn X$ defined by
$\varrho_X (Q, C) \eqdef Q \cap C$, and
$\varrho_X \circ \iota_X = \identity X$.  When $X$ is weakly Hausdorff
\cite[Lemma~6.1, Proposition~6.2, Theorem~6.3]{JGL:wHaus} or
quasi-Polish \cite[Theorem 9.6]{JGL:projlim:prev}, those define
inverse continuous maps, hence a homeomorphism between $\Plotkinn X$
and $\QLV X$.  A space is \emph{weakly Hausdorff} if and only if given
any two points $x$ and $y$, every open neighborhood $W$ of
$\upc x \cap \upc y$ contains an intersection $U \cap V$ of an open
neighborhood $U$ of $x$ and an open neighborhood $V$ of $y$
\cite[Lemma~6.6]{KL:measureext}.  Every Hausdorff space is weakly
Hausdorff, every stably locally compact space is weakly Hausdorff
\cite[Lemma~8.1]{KL:measureext}.  The \emph{quasi-Polish spaces} are
the topological spaces underlying a separable Smyth-complete
quasi-metric space \cite{deBrecht:qPolish}.  We will not define them
any further, except than to notice that every $\omega$-continuous dcpo
is quasi-Polish in its Scott topology.  (See Section~2.2 of
\cite{AJ:domains} for continuous and $\omega$-continuous dcpos.)

We will write $\pi_1 \colon \QLV X \to \SV X$ and
$\pi_2 \colon \QLV X \to \HV X$ for the respective projections.  For
every $U \in \Open X$, $\pi_1^{-1} (\Box U) = \Box^\quasi U$ and
$\pi_2^{-1} (\Diamond U) = \Diamond^\quasi U$, so they are continuous.
Those define natural transformations $\pi_1 \colon \QLV \to \SV$ and
$\pi_2 \colon \QLV \to \HV$, as one can see easily using Trick~A.

\begin{proposition}
  \label{prop:QLV:monad}
  There is a monad $(\QLV, \eta^\quasi, \mu^\quasi)$ on $\Topcat$,
  \begin{itemize}
  \item whose unit $\eta^\quasi_X \colon X \to \QLV X$ maps every
    $x \in X$ to
    $(\eta^\Smyth_X (x), \eta^\Hoare_X (x)) = (\upc x, \dc x)$, and is
    such that
    ${(\eta^\quasi_X)}^{-1} (\Box^\quasi U) = {(\eta^\quasi_X)}^{-1}
    (\Diamond^\quasi U) = U$ for every $U \in \Open X$, and is
    characterized uniquely by the commutativity of the diagram:
    \begin{equation}
      \label{eq:etaq}
      \xymatrix{
        X
        \ar[d]_{\eta^\Smyth_X}
        &
        X
        \ar@{=}[r]
        \ar[d]|{\eta^\quasi_X}
        \ar@{=}[l]
        &
        X
        \ar[d]^{\eta^\Hoare_X}
        \\
        \SV X
        &
        \QLV X
        \ar[l]^{\pi_1}
        \ar[r]_{\pi_2}
        &
        \HV X
      }
    \end{equation}
  \item and whose multiplication
    $\mu^\quasi_X \colon \QLV {\QLV X} \to \QLV X$ maps every
    quasi-lens $(\mathcal Q, \mathcal C)$ on $\QLV X$ to
    $(\bigcup \pi_1 [\mathcal Q], cl (\bigcup \pi_2 [\mathcal C]))$,
    is such that
    ${(\mu^\quasi_X)}^{-1} (\Box^\quasi U) = \Box^\quasi {\Box^\quasi
      U}$ and
    ${(\mu^\quasi_X)}^{-1} (\Diamond^\quasi U) = \Diamond^\quasi
    {\Diamond^\quasi U}$ for every $U \in \Open X$, and is
    characterized uniquely by the commutativity of the diagram:
    \begin{equation}
      \label{eq:muq}
      \xymatrix{
        \SV {\SV X}
        \ar[d]_{\mu^\Smyth_X}
        &
        \SV {\QLV X}
        \ar[l]_{\SV {\pi_1}}
        &
        \QLV {\QLV X}
        \ar[l]_{\pi_1}
        \ar[d]|{\mu^\quasi_X}
        \ar[r]^{\pi_2}
        &
        \HV {\QLV X}
        \ar[r]^{\HV {\pi_2}}
        &
        \HV {\HV X}
        \ar[d]^{\mu^\Hoare_X}
        \\
        \SV X
        &
        &
        \QLV X
        \ar[ll]^{\pi_1}
        \ar[rr]_{\pi_2}
        &
        &
        \HV X
      }
    \end{equation}
  \end{itemize}
\end{proposition}
\begin{proof}
  Unit.  For every $x \in X$, we claim that
  $(Q, C) \eqdef (\upc x, \dc x)$ is a quasi-lens.  First,
  $\upc x \cap \dc x$ is non-empty, since it contains $x$.  Second,
  $Q \subseteq \upc (Q \cap C)$, since every element of $Q$ is larger
  than or equal to $x$, which is in $Q \cap C$.  Third, since
  $x \in Q \cap C$, $C = \dc x = cl (\{x\}) \subseteq cl (Q \cap C)$,
  from which is follows that for every open neighborhood $U$ of $Q$,
  $C \subseteq (U \cap C)$.  The commutativity of (\ref{eq:etaq}) is
  an equivalent way of stating the definition of $\eta^\quasi_X$
  (modulo the fact that all arrows are indeed continuous maps, which
  we show next).

  For every $U \in \Open X$, for every $x \in X$,
  $x \in {(\eta^\quasi_X)}^{-1} (\Box^\quasi U)$ if and only if
  $\eta^\Smyth_X (x) \in \Box U$, if and only if $x \in U$, and the
  equality ${(\eta^\quasi_X)}^{-1} (\Diamond^\quasi U) = U$ is proved
  similarly.  In particular, $\eta^\quasi_X$ is continuous.

  Finally, the commutativity of (\ref{eq:etaq}) characterizes
  $\eta^\quasi_X$, since the pairing
  $\langle \pi_1, \pi_2 \rangle \colon \QLV X \to \SV X \times \HV X$
  is the inclusion map, hence is injective.

  Multiplication.  For every
  $(\mathcal Q, \mathcal C) \in \QLV {\QLV X}$, let
  $(Q, C) \eqdef (\bigcup \pi_1 [\mathcal Q], \allowbreak cl (\bigcup
  \pi_2 [\mathcal C]))$.  We need to check that this is a quasi-lens.
  
  We have
  $\mu^\Smyth_X (\SV {\pi_1} (\mathcal Q)) = \bigcup \upc \pi_1
  [\mathcal Q] = \bigcup \pi_1 [\mathcal Q] = Q$ (because
  $\pi_1 [\mathcal Q]$ is upwards-closed, as one checks easily).  In
  particular, $Q$ is in $\SV X$.  Also,
  $\mu^\Hoare_X (\HV {\pi_2} (\mathcal C)) = cl (\bigcup cl (\pi_2
  [\mathcal C]))$.  The outer closure is taken in $X$, while the inner
  closure is taken in $\HV X$.  We claim that this is equal to $C$.
  Clearly, $C \subseteq cl (\bigcup cl (\pi_2 [\mathcal C]))$.
  Conversely, every open set $U$ that intersects
  $cl (\bigcup cl (\pi_2 [\mathcal C]))$ intersects
  $\bigcup cl (\pi_2 [\mathcal C])$, so there is an element of
  $cl (\pi_2 [\mathcal C])$ that intersects $U$; in other words,
  $cl (\pi_2 [\mathcal C])$ intersects $\Diamond U$, so
  $\pi_2 [\mathcal C]$ intersects $\Diamond U$, equivalently,
  $\bigcup \pi_2 [\mathcal C] = C$ intersects $U$.  Since
  $C = \mu^\Hoare_X (\HV {\pi_2} (\mathcal C))$, in particular
  $C \in \HV X$.  Also, we have just shown that the diagram
  (\ref{eq:muq}) commutes (modulo the fact that all arrows are
  continuous maps, which we will check later).

  Since $(\mathcal Q, \mathcal C)$ is a quasi-lens, there is an
  element $(Q', C')$ in $\mathcal Q \cap \mathcal C$.  Then there is
  an element $x$ in $Q' \cap C'$, because $(Q', C')$ is a quasi-lens.
  Since $(Q', C') \in \mathcal Q$, $Q'$ is in $\pi_1 [\mathcal Q]$, so
  $Q' \subseteq \bigcup \pi_1 [\mathcal Q] = Q$; then
  $x \in Q' \subseteq Q$.  Since $(Q', C') \in \mathcal C$, $C'$ is in
  $\pi_2 [\mathcal C]$, so
  $C' \subseteq cl (\bigcup \pi_2 [\mathcal C]) =C$; then
  $x \in C' \subseteq C$.  It follows that $Q \cap C \neq \emptyset$.
  
  Let us show that $Q \subseteq \upc (Q \cap C)$.  For every
  $x \in Q = \bigcup \pi_1 [\mathcal Q]$, there is quasi-lens
  $(Q', C') \in \mathcal Q$ such that $x \in Q'$.  Since
  $(\mathcal Q, \mathcal C)$ is itself a quasi-lens,
  $\mathcal Q \subseteq \upc (\mathcal Q \cap \mathcal C)$, so there
  is a quasi-lens $(Q'', C'') \in \mathcal Q \cap \mathcal C$ such
  that $(Q'', C'') \mathrel{(\supseteq \times \subseteq)} (Q', C')$.
  We have $Q'' \supseteq Q'$ and $x \in Q'$, so $x \in Q''$; also,
  $Q'' \subseteq \upc (Q'' \cap C'')$, so there is a point
  $x'' \in Q'' \cap C''$ such that $x'' \leq x$.  Since $x'' \in Q''$
  and $(Q'', C'') \in \mathcal Q$, $x''$ is in
  $\bigcup \pi_1 [\mathcal Q] = Q$, and since $x'' \in C''$ and
  $(Q'', C'') \in \mathcal C$, $x''$ is in
  $\bigcup \pi_2 [\mathcal C]$, hence in its closure $C$.  We have
  obtained that $x$ is larger than or equal to a point $x''$ in
  $Q \cap C$, so $x \in \upc (Q \cap C)$.

  We now show that for every open neighborhood $U$ of $Q$,
  $C \subseteq cl (U \cap C)$.  It suffices to show that for every
  open neighborhood $U$ of $Q$, every open set $V$ that intersects $C$
  also intersects $U \cap C$.  Since
  $Q = \bigcup \pi_1 [\mathcal Q] \subseteq U$, every quasi-lens
  $(Q', C') \in \mathcal Q$ is such that $Q' \subseteq U$, namely is
  in $\Box^\quasi U$.  Therefore $\mathcal Q \subseteq \Box^\quasi U$.
  Since $(\mathcal Q, \mathcal C)$ is a quasi-lens,
  $\mathcal C \subseteq cl (\mathcal C \cap \mathcal \Box^\quasi U)$.
  Hence every open subset $\mathcal V$ of $\QLV X$ that intersects
  $\mathcal C$ also intersects
  $\mathcal C \cap \mathcal \Box^\quasi U$.  Since
  $C = cl (\bigcup \pi_2 [\mathcal C])$ and $V$ intersects $C$, $V$
  intersects $\bigcup \pi_2 [\mathcal C]$, hence $V$ intersects some
  closed set $C'$ such that $(Q', C') \in \mathcal C$.  In other
  words, $\Diamond^\quasi V$ intersects $\mathcal C$.  Taking
  $\mathcal V \eqdef \Diamond^\quasi V$, we obtain that
  $\Diamond^\quasi V$ intersects
  $\mathcal C \cap \mathcal \Box^\quasi U$.  In other words, there is
  a quasi-lens $(Q', C')$ in $\mathcal C$ such that $Q' \subseteq U$
  and $C'$ intersects $V$.  Since $(Q', C')$ is a quasi-lens,
  $C' \subseteq cl (C' \cap U)$, so $V$ intersects $C' \cap U$.  Hence
  $C'$ intersects $U \cap V$.  But, since $(Q', C') \in \mathcal C$,
  $C'$ is in $\pi_2 [\mathcal C]$, so
  $C' \subseteq C = cl (\bigcup \pi_2 [\mathcal C])$.  Therefore $C$
  also intersects $U \cap V$, showing that $V$ intersects $U \cap C$.

  It remains to show that $\mu^\quasi_X$ is continuous.  For every
  open subset $U$ of $X$,
  ${(\mu^\quasi_X)}^{-1} (\Box^\quasi U) = \{(\mathcal Q, \mathcal C)
  \in \QLV {\QLV X} \mid \bigcup \pi_1 [\mathcal Q] \subseteq U\} =
  \{(\mathcal Q, \mathcal C) \in \QLV {\QLV X} \mid \forall (Q, C) \in
  \mathcal Q, Q \subseteq U\} = \Box^\quasi {\Box^\quasi U}$, and
  ${(\mu^\quasi_X)}^{-1} (\Diamond^\quasi U) = \{(\mathcal Q, \mathcal
  C) \in \QLV {\QLV X} \mid cl (\bigcup \pi_2 [\mathcal C]) \cap U
  \neq \emptyset\} = \{(\mathcal Q, \mathcal C) \in \QLV {\QLV X} \mid
  \bigcup \pi_2 [\mathcal C] \cap U \neq \emptyset\} = \{(\mathcal Q,
  \mathcal C) \in \QLV {\QLV X} \mid \exists (Q, C) \in \mathcal C, C
  \cap U \neq \emptyset\} = \Diamond^\quasi {\Diamond^\quasi U}$.

  There are two easy ways to verify that $\eta^\quasi$ and
  $\mu^\quasi$ are natural, and the monad laws.  One is by using
  Trick~A, and the other one is as follows.  Given any two maps
  $f, g \colon Y \to \QLV Z$ (for some spaces $Y$ and $Z$), $f=g$ if
  and only if $\pi_1 \circ f = \pi_1 \circ g$ and
  $\pi_2 \circ f = \pi_2 \circ g$, since the pairing
  $\langle \pi_1, \pi_2 \rangle$ is the inclusion map, hence is
  injective.  We use this new trick, as it will be clear that this
  reduces the question for $\QLV$ to those for $\SV$ and $\HV$.
  
  Naturality.  Given any continuous map $f \colon X \to Y$, we check
  that $\eta^\quasi_Y \circ f = \QLV f \circ \eta^\quasi_X$:
  $\pi_1 \circ \eta^\quasi_Y \circ f = \eta^\Smyth_Y \circ f$ (by
  (\ref{eq:etaq}) $= \SV f \circ \eta^\Smyth_X$ (by naturality of
  $\eta^\Smyth$) $= \pi_1 \circ \QLV f \circ \eta^\Smyth_X$, and
  similarly
  $\pi_2 \circ \eta^\quasi_Y \circ f = \pi_2 \circ \QLV f \circ
  \eta^\Smyth_X$.

  We also check that $\QLV f \circ \mu^\quasi_X = \mu^\quasi_Y \circ
  \QLV {\QLV f}$ by:
  \begin{align*}
    \pi_1 \circ \QLV f \circ \mu^\quasi_X
    & = \SV f \circ \pi_1 \circ \mu^\quasi_X
    & \text{def.\ of $\QLV f$} \\
    & = \SV f \circ \mu^\Smyth_X \circ \SV {\pi_1} \circ \pi_1
    & \text{by (\ref{eq:muq})} \\
    & = \mu^\Smyth_Y \circ \SV {\SV f} \circ \SV {\pi_1} \circ \pi_1
    & \text{nat.\ of $\mu^\Smyth$} \\
    & = \mu^\Smyth_Y \circ \SV {\pi_1} \circ \SV {\QLV f} \circ \pi_1
    & \text{nat.\ of $\pi_1$} \\
    & = \mu^\Smyth_Y \circ \SV {\pi_1} \circ \pi_1 \circ \QLV {\QLV f}
    & \text{nat.\ of $\pi_1$} \\
    & = \pi_1 \circ \mu^\quasi_Y \circ \QLV {\QLV f}
    & \text{by (\ref{eq:muq})},
  \end{align*}
  and similarly with $\pi_2$ instead of $\pi_1$, $\Hoare$ in lieu of
  $\Smyth$, and $\HV$ instead of $\SV$.
  
  The monad laws.  We proceed similarly.  We verify that $\mu^\quasi_X
  \circ \eta^\quasi_{\QLV X} = \identity {\QLV X}$ by:
  \begin{align*}
    \pi_1 \circ \mu^\quasi_X \circ \eta^\quasi_{\QLV X}
    & = \mu^\Smyth_X \circ \SV {\pi_1} \circ \pi_1 \circ
      \eta^\quasi_{\QLV X}
    & \text{by (\ref{eq:muq})} \\
    & = \mu^\Smyth_X \circ \SV {\pi_1} \circ \eta^\Smyth_{\QLV X}
    & \text{by (\ref{eq:etaq})} \\
    & = \mu^\Smyth_X \circ \eta_{\SV X} \circ \pi_1
    & \text{nat.\ of $\pi_1$} \\
    & = \pi_1,
  \end{align*}
  by the monad law
  $\mu^\Smyth_X \circ \eta^\Smyth_{\SV X} = \identity {\SV X}$ for
  $\SV$; similarly with $\pi_2$ instead of $\pi_1$.

  We verify that
  $\mu^\quasi_X \circ \QLV {\eta^\quasi_X} = \identity {\QLV X}$ by:
  \begin{align*}
    \pi_1 \circ \mu^\quasi_X \circ \QLV {\eta^\quasi_X}
    & = \mu^\Smyth_X \circ \SV {\pi_1} \circ \pi_1 \circ \QLV {\eta^\quasi_X}
    & \text{by (\ref{eq:muq})} \\
    & = \mu^\Smyth_X \circ \SV {\pi_1} \circ \SV {\eta^\quasi_X} \circ
      \pi_1
    & \text{def.\ of $\QLV$ on morphisms} \\
    & = \mu^\Smyth_X \circ \SV {\eta^\Smyth_X} \circ \pi_1
    & \text{by (\ref{eq:etaq})} \\
    & = \pi_1,
  \end{align*}
  by the monad law
  $\mu^\Smyth_X \circ \SV {\eta^\Smyth_X} = \identity {\SV X}$ for
  $\SV$; similarly with $\pi_2$ instead of $\pi_1$.

  Finally, we verify that $\mu^\quasi_X \circ \mu^\quasi_{\QLV X} =
  \mu^\quasi_X \circ \QLV {\mu^\quasi_X}$.  Once again, we do this for
  $\pi_1$ only:
  \begin{align*}
    \pi_1 \circ \mu^\quasi_X \circ \mu^\quasi_{\QLV X}
    & = \mu^\Smyth_X \circ \SV {\pi_1} \circ \pi_1
      \circ \mu^\quasi_{\QLV X}
    & \text{by (\ref{eq:muq})} \\
    & = \mu^\Smyth_X \circ \SV {\pi_1}
      \circ \mu^\Smyth_{\QLV X} \circ \SV {\pi_1} \circ \pi_1
    & \text{by (\ref{eq:muq})} \\
    & = \mu^\Smyth_X
      \circ \mu^\Smyth_{\QLV X} \circ \SV {\SV {\pi_1}} \circ \SV
      {\pi_1} \circ \pi_1
    & \text{nat.\ of $\mu^\Smyth$}
  \end{align*}
  while:
  \begin{align*}
    \pi_1 \circ \mu^\quasi_X \circ \QLV {\mu^\quasi_X}
    & = \mu^\Smyth_X \circ \SV {\pi_1} \circ \pi_1
      \circ \QLV {\mu^\quasi_X}
    & \text{by (\ref{eq:muq})} \\
    & = \mu^\Smyth_X \circ \SV {\pi_1}
      \circ \SV {\mu^\quasi_X} \circ \pi_1
    & \text{def.\ of $\QLV$ on morphisms} \\
    & = \mu^\Smyth_X
      \circ \SV {\mu^\Smyth_X} \circ \SV {\SV {\pi_1}} \circ \SV
      {\pi_1} \circ \pi_1
      \mskip-150mu
    & \text{by (\ref{eq:muq})},
  \end{align*}
  and we conclude since
  $\mu^\Smyth_X \circ \mu^\Smyth_{\QLV X} = \mu^\Smyth_X \circ \SV
  {\mu^\Smyth_X}$, by the final monad law on $\SV$.  \qed
\end{proof}

\begin{remark}
  \label{rem:Plotkinn}
  There is also a $\Plotkinn$ monad (of lenses) on $\Topcat$, but we
  will not care about it except on full subcategories of weakly
  Hausdorff of quasi-Polish spaces, where we can transport the monad
  structure of $\QLV$ over to $\Plotkinn$.  We obtain that for every
  continuous map $f \colon X \to Y$,
  $\Plotkinn f = \varrho_Y \circ \QLV f \circ \iota_X$ maps every lens
  $L$ on $X$ to $\upc f [\upc L] \cap cl (f [cl (L)])$ (which we may
  simplify as $\upc f [L] \cap cl (f [L])$), the unit
  $\eta^{\Plotkin}_X$ maps every $x \in X$ to
  $\varrho_X (\upc x, \dc x) = \upc x \cap \dc x$ (which simplifies to
  $\{x\}$ if $X$ is $T_0$), and the multiplication
  $\mu^{\Plotkin}_X = \varrho_X \circ \mu^\quasi_X \circ \QLV
  {\iota_X} \circ \iota_{\Plotkin X}$ maps every lens
  $\mathcal L \in \Plotkinn {\Plotkinn X}$ to
  $\upc \bigcup \mathcal L \cap cl (\bigcup \mathcal L)$.
\end{remark}

In order to obtain a distributive law over $\Val_\bullet$, we need to
further restrict the category of topological spaces we can work on.
Not only do we need our spaces $X$ to be such that $\Lform X$ is
locally convex, but addition on $\Lform X$ needs to be \emph{almost
  open}, meaning that given any two open subsets $\mathcal U$ and
$\mathcal V$ of $\Lform X$, $\upc (\mathcal U + \mathcal V)$ is open.
(We will also need $X$ to be compact when $\bullet$ is ``$1$''.)  This
is the case if $X$ is stably locally compact \cite[Lemma
3.24]{JGL-mscs16}.  (We have already stated that every stably locally
compact space $X$ is weakly Hausdorff \cite[Lemma~8.1]{KL:measureext},
so that $\Plotkinn X \cong \QLV X$ in this case.)

We also need our category of spaces to be closed under $\Val_\bullet$,
$\QLV$, and under retracts.  The category of stably compact spaces is
closed under retracts \cite[Proposition, bottom of p.153, and
subsequent discussion]{Lawson:versatile}, see also
\cite[Proposition~2.17]{Jung:scs:prob}.


As far as $\Val_\bullet X$ is concerned, it was shown in
\cite[Theorem~39]{AMJK:scs:prob} that $\Val_{\leq 1} X$ is stably
compact for every stably compact space $X$.  The same proof shows that
$\Val X$ is stably compact under the same assumption.  Additionally,
$\Val_1 X$ occurs as the subspace $\langle X \geq 1\rangle$, where
$\langle Q \geq r \rangle$ denotes
$\{\nu \in \Val_{\leq 1} X \mid \forall U \in \Open X, Q \subseteq U
\limp \nu (U) \geq 1\}$, and is compact saturated for every compact
saturated subset $Q$ of the stably compact space $X$ and for every
$r \in \Rp$ \cite[Lemma 6.6]{GL:duality}.  But any compact saturated
subset of a stably compact space is stably compact; more generally,
any patch-closed subset of a stably compact space, namely, any subset
that we can obtain as an intersection of finite unions of closed
subsets and of compact saturated subsets, is stably compact
\cite[Proposition~2.16]{Jung:scs:prob}.

As for $\QLV X$, it was noted in \cite[Fact~5.2]{GL:duality} that for
every sober space $X$, $\QLV X$ is homeomorphic to Heckmann's space of
$\mathbf A$-valuations on $X$ \cite{heckmann96}.  If $X$ is stably
compact, then, both spaces are stably compact
\cite[Proposition~5.13]{GL:duality}.
\begin{definition}
  \label{defn:Top:nat}
  Let $\Topcat^\natural$ be any full subcategory of $\Topcat$
  consisting of spaces $X$ such that $\Lform X$ is locally convex and
  has an almost open addition map (and such that $X$ is compact if
  $\bullet$ is ``$1$''), and which is closed under $\QLV$,
  $\Val_\bullet$ and under retracts---for example the category of
  stably compact spaces.
\end{definition}
We note that every object of $\Topcat^\natural$ is an
$\AN_\bullet$-friendly space.  In particular, if $\bullet$ is ``$1$'',
every compact space $X$ such that $\Lform X$ is locally convex is
$\AN_1$-friendly, by definition.

\section{The monad of forks}
\label{sec:monad-forks}

A \emph{fork} on a space $X$ is any pair
$(F^-, F^+)$ of a superlinear prevision $F^-$ on $X$ and of a
sublinear prevision $F^+$ on $X$ satisfying \emph{Walley's condition}:
\[
  F^- (h+h') \leq F^- (h) + F^+ (h') \leq F^+ (h+h')
\]
for all $h, h' \in \Lform X$ \cite{Gou-csl07,KP:predtrans:pow}.  A
fork is \emph{subnormalized}, resp.\ \emph{normalized} if and only if
both $F^-$ and $F^+$ are.

We write $\Pred_{\ADN} X$ for the set of all forks on $X$, and
$\Pred^{\leq 1}_{\ADN} X$, $\Pred^1_{\ADN} X$ for their subsets of
subnormalized, resp.\ normalized, forks.  The \emph{weak topology} on
each is the subspace topology induced by the inclusion into the larger
space $\Pred_{\DN} X \times \Pred_{\AN} X$.  A subbase of the weak
topology is composed of two kinds of open subsets: $[h > r]^-$,
defined as $\{(F^-, F^+) \mid F^- (h) > r\}$, and $[h > r]^+$, defined
as $\{(F^-, F^+) \mid F^+ (h) > r\}$, where $h \in \Lform X$,
$r \in \real^+$.  The specialization ordering of spaces of forks is
the product ordering $\leq \times \leq$, where $\leq$ denotes the
pointwise ordering on previsions.  In particular, all those spaces of
forks are $T_0$.

Whether $\bullet$ is nothing, ``$\leq 1$'', or
``$1$'', $\Pred_\ADN^\bullet$ defines an endofunctor on $\Topcat$,
whose action on morphisms is given by
$\Pred_\ADN^\bullet f \eqdef (\Pred f, \Pred f)$.

\begin{proposition}
  \label{prop:rADP}
  Let $\bullet$ be nothing, ``$\leq 1$'', or ``$1$''.  For every
  topological space $X$ such that $\Lform X$ is locally convex and has
  an almost open addition map (and such that $X$ is compact, in case
  $\bullet$ is ``$1$''), there are continuous maps
  $r_\ADN \colon \QLV {\Val_\bullet X} \to \Pred_\ADN^\bullet X$ and
  $s_\ADN^\bullet \colon \Pred_\ADN^\bullet X \to \QLV {\Val_\bullet
    X}$ defined by:
  \begin{align*}
    r_\ADN (Q, C)
    & \eqdef (r_\DN (Q), r_\AN (C)) \\
    s_\ADN^\bullet (F^-, F^+)
    & \eqdef (s_\DN^\bullet (F^-), s_\AN^\bullet (F^+)).
  \end{align*}
  Additionally, $r_\ADN$ and $s_\ADN^\bullet$ are natural in $X$.
\end{proposition}
\begin{proof}
  For every quasi-lens $(Q, C)$ on $\Val_\bullet X$, we check that
  $(F^-, F^+) \eqdef (r_\DN (Q), r_\AN (C))$ is a fork.  For all
  $h, h' \in \Lform X$,
  \begin{align*}
    F^- (h+h')
    & = \min_{\nu \in Q} \int_{x \in X} (h (x) + h' (x)) \,d\nu \\
    & \leq \min_{\nu \in Q \cap C} \int_{x \in X} (h (x) + h' (x))
      \,d\nu \\
    & \leq \min_{\nu \in Q \cap C} \left(\int_{x \in X} h (x)\,d\nu
      + F^+ (h')\right) \\
    & \qquad\text{since $\int_{x \in X} h' (x)
      \,d\nu \leq F^+ (h')$ for every $\nu \in Q \cap C$} \\
    & = \min_{\nu \in Q \cap C} \int_{x \in X} h (x)\,d\nu + F^+ (h') \\
    & = \min_{\nu \in \upc (Q \cap C)} \int_{x \in X} h (x)\,d\nu +
      F^+ (h') \\
    & \leq \min_{\nu \in Q} \int_{x \in X} h (x)\,d\nu + F^+
      (h')
    \qquad \text{since $Q \subseteq \upc (Q \cap C)$} \\
    & = F^- (h) + F^+ (h').
  \end{align*}
  For Walley's other inequality, we observe that:
  \begin{align*}
    F^- (h) + F^+ (h')
    & = \dsup_{r \in \Rp, r < F^- (h)} (r + F^+ (h')) \\
    & = \sup_{r \in \Rp, r < F^- (h), \nu \in C} \left(r + \int_{x \in X} h' (x) \,d\nu\right).
  \end{align*}
  Then, for every $r \in \Rp$ such that $r < F^- (h)$, we have
  $r < \min_{\nu \in Q} \int_{x \in X} h (x) \,d\nu$, so
  $Q \subseteq [h > r]$, by definition of $F^- = r_\DN (Q)$.  Since
  $(Q, C)$ is a quasi-lens, $C \subseteq cl ([h > r] \cap C)$, so:
  \begin{align*}
    F^- (h) + F^+ (h')
    & \leq \sup_{r \in \Rp, r < F^- (h), \nu \in cl ([h > r] \cap C)}
      \left(r + \int_{x \in X} h' (x) \,d\nu\right) \\
    & = \sup_{r \in \Rp, r < F^- (h), \nu \in [h > r] \cap C}
      \left(r + \int_{x \in X} h' (x) \,d\nu\right)
    & \text{by Fact~\ref{fact:sup:cl}.}
  \end{align*}
  Fact~\ref{fact:sup:cl} applies, since
  $\nu \in \Val_\bullet X \mapsto \int_{x \in X} h' (x) \,d\nu$ is
  lower semicontinuous, as the inverse image of any basic open set
  $]t, \infty]$ is $[h' > t]$; and therefore
  $\nu \in \Val_\bullet X \mapsto r+\int_{x \in X} h' (x) \,d\nu$ is
  lower semicontinuous, too.  Now for every $r \in \Rp$ such that
  $r < F^- (h)$, for every $\nu \in [h > r] \cap C$, we have
  $\int_{x \in X} h (x) \,d\nu > r$ by definition of $[h > r]$, so:
  \begin{align*}
    F^- (h) + F^+ (h')
    & \leq \sup_{r \in \Rp, r < F^- (h), \nu \in [h > r] \cap C}
      \left(\int_{x \in X} h (x) \,d\nu + \int_{x \in X} h' (x)
      \,d\nu\right) \\
    & \leq \sup_{\nu \in C}
      \left(\int_{x \in X} h (x) \,d\nu + \int_{x \in X} h' (x)
      \,d\nu\right) \\
    & = F^+ (h+h').
  \end{align*}
  Additionally, if $\bullet$ is ``$\leq 1$'' (resp., ``$1$''), then
  $F^- = r_\DN (Q)$ and $F^+ = r_\AN (C)$ are subnormalized (resp.,
  normalized), so $r_\ADN$ takes its values in $\Pred_\ADN^\bullet X$.
  The fact that $r_\ADN$ is continuous follows from the fact that
  $r_\DN$ and $r_\AN$ are, and similarly for naturality.

  Turning to $s_\ADN^\bullet$, we must show that for every
  $(F^-, F^+) \in \Pred_\ADN^\bullet X$,
  $(Q, C) \eqdef (s_\DN^\bullet (F^-), s_\AN^\bullet (F^+))$ is a
  quasi-lens.  That $s_\ADN^\bullet$ is continuous and natural will
  then follow from the same properties for $s_\DN^\bullet$ and
  $s_\AN^\bullet$.

  Lemma~3.27 of \cite{JGL-mscs16} 
  states that there is a map (called $s_\ADN^\bullet$ there; although
  this is not the same map as our $s_\ADN^\bullet$, the two are
  definitely related) from $\Pred_\ADN^\bullet X$ to the space of
  lenses on $\Pred_\Nature^\bullet X$, and which sends $(F^-, F^+)$ to
  $s_\DN (F^-) \cap s_\AN^\bullet (F^+)$.  
  In particular, $Q \cap C = s_\DN (F^-) \cap s_\AN^\bullet (F^+)$ is
  a lens, and is therefore non-empty.

  Up to the isomorphism
  $\Pred_\Nature^\bullet X \cong \Val_\bullet X$, Lemma~3.28 of
  \cite{JGL-mscs16} 
  states that for every
  $\nu' \in \Val_\bullet X$, if $\nu' \in s_\DN^\bullet (F^-)$, then
  there is a $\nu \in \Val_\bullet X$ such that $\nu \leq \nu'$ and
  for every $h \in \Lform X$,
  $F^- (h) \leq \int_{x \in X} h (x) \,d\nu \leq F^+ (h)$.  In other
  words, $Q \subseteq \upc (Q \cap C)$.

  Similarly, Lemma~3.29 of \cite{JGL-mscs16} 
  states that for every
  $\nu' \in \Val_\bullet X$ such that $\nu' \in s_\AN^\bullet (F^+)$,
  there is a $\nu \in \Val_\bullet X$ such that $\nu' \leq \nu$ and
  for every $h \in \Lform X$,
  $F^- (h) \leq \int_{x \in X} h (x) \,d\nu \leq F^+ (h)$.  This
  requires $\Lform X$ to be locally convex and to have an almost open
  addition map, and also $X$ to be compact if $\bullet$ is ``$1$''.
  In other words, $C \subseteq \dc (Q \cap C)$.  Since the
  downwards-closure of a set is always included in its closure,
  $C \subseteq cl (Q \cap C)$.  In particular, for every open
  neighborhood $U$ of $Q$, $C \subseteq cl (U \cap C)$.  \qed
\end{proof}

\begin{proposition}
  \label{prop:rADP:convex}
  Let $\bullet$ be nothing, ``$\leq 1$'', or ``$1$''.  For every
  topological space $X$ such that $\Lform X$ is locally convex and has
  an almost open addition map (and such that $X$ is compact, in case
  $\bullet$ is ``$1$''), $r_\ADN$ and $s_\ADN^\bullet$ restrict to a
  homeomorphism between the subspace $\QLVc {\Val_\bullet X}$ of
  $\QLV {\Val_\bullet X}$ consisting of quasi-lenses $(Q, C)$ such
  that both $Q$ and $C$ are convex, and $\Pred_\ADN^\bullet X$.
\end{proposition}
\begin{proof}
  Immediate consequence of Proposition~\ref{prop:rADP}, of the fact
  that $r_\DN$ and $s_\DN^\bullet$ form a homeomorphism between
  $\SV^{cvx} {\Val_\bullet X}$ and $\Pred_\DN^\bullet X$, and of the
  fact that $r_\AN$ and $s_\AN^\bullet$ form a homeomorphism between
  $\HV^{cvx} {\Val_\bullet X}$ and $\Pred_\AN^\bullet X$.  Note that
  this holds because $X$ is $\AN_\bullet$-friendly; if $\bullet$ is
  ``$1$'', notably, then not only is $\Lform X$ locally convex, but we
  have also assumed that $X$ is compact.
  \qed
\end{proof}

Much like there are projection maps $\pi_1 \colon \QLV X \to \SV X$
and $\pi_2 \colon \QLV X \to \HV X$, there are projection maps, which
we write with the same notation,
$\pi_1 \colon \Pred_\ADN^\bullet X \to \Pred_\DN^\bullet X$ and
$\pi_2 \colon \Pred_\ADN^\bullet X \to \Pred_\AN^\bullet X$.  The
pairing $\langle \pi_1, \pi_2 \rangle$ is then the inclusion map
$\Pred_\ADN^\bullet X \to \Pred_\DN^\bullet X \times \Pred_\AN^\bullet
X$.  It is easy to see that this is a topological embedding.

It was shown in \cite[Proposition~3]{Gou-csl07} that there is a monad
of forks on the category of dcpos (resp., pointed dcpos) and
Scott-continuous maps, except that spaces of previsions were given the
Scott topology of the pointwise ordering.  As for previsions, we check
that there is a corresponding monad of (subnormalized, normalized)
forks on $\Topcat$.
\begin{proposition}
  \label{prop:fork:monad}
  Let $\bullet$ be nothing, ``$\leq 1$'', or ``$1$''.  There is a
  monad
  $(\Pred_\ADN^\bullet, \allowbreak \eta^\ADN, \allowbreak \mu^\ADN)$
  on $\Topcat$ (and, by restriction, on $\Topcat^\natural$) whose unit
  and multiplication are defined by:
  \begin{align*}
    \eta^\ADN_X (x)
    & \eqdef (\eta^\DN_X (x), \eta^\AN_X (x)) \\
    \mu^\ADN_X (\mathcal F^-, \mathcal F^+)
    & \eqdef (h \in \Lform X \mapsto \mathcal F^- ((F^-, F^+) \in
      \Pred_\ADN^\bullet X \mapsto F^- (h)), \\
      & \qquad
        h \in \Lform X \mapsto \mathcal F^+ ((F^-, F^+) \in
        \Pred_\ADN^\bullet X \mapsto F^+ (h)))
  \end{align*}
  for every $x \in X$ in the first case, and for every
  $(\mathcal F^-, \mathcal F^+) \in \Pred_\ADN^\bullet
  {\Pred_\ADN^\bullet X}$ in the second case.  In other words, unit
  and multiplication are defined uniquely by the commutativity of the
  following diagrams:
  \begin{equation}
    \label{eq:etaADN}
    \xymatrix{
      X
      \ar[d]_{\eta^\DN_X}
      &
      X
      \ar@{=}[r]
      \ar[d]|{\eta^\ADN_X}
      \ar@{=}[l]
      &
      X
      \ar[d]^{\eta^\AN_X}
      \\
      \Pred_\DN^\bullet X
      &
      \Pred_\ADN^\bullet X
      \ar[l]^{\pi_1}
      \ar[r]_{\pi_2}
      &
      \Pred_\AN^\bullet X
    }
  \end{equation}
  and:
  \begin{equation}
    \label{eq:muADN}
    \xymatrix{
      \Pred_\DN^\bullet {\Pred_\DN^\bullet X}
      \ar[d]_{\mu^\DN_X}
      &
      \Pred_\DN^\bullet {\Pred_\ADN^\bullet X}
      \ar[l]_{\Pred_\DN^\bullet {\pi_1}}
      &
      \Pred_\ADN^\bullet {\Pred_\ADN^\bullet X}
      \ar[l]_{\pi_1}
      \ar[d]|{\mu^\ADN_X}
      \ar[r]^{\pi_2}
      &
      \Pred_\AN^\bullet {\Pred_\ADN^\bullet X}
      \ar[r]^{\Pred_\AN^\bullet {\pi_2}}
      &
      \Pred_\AN^\bullet {\Pred_\AN^\bullet X}
      \ar[d]^{\mu^\AN_X}
      \\
      \Pred_\DN^\bullet X
      &
      &
      \Pred_\ADN^\bullet X
      \ar[ll]^{\pi_1}
      \ar[rr]_{\pi_2}
      &
      &
      \Pred_\AN^\bullet X
    }
  \end{equation}
\end{proposition}
We recall that the functor part acts on morphisms by
$\Pred_\ADN^\bullet f = (\Pred f, \Pred f)$, hence is characterized by
the commutativity of the following diagram:
\begin{equation}
  \label{eq:PADNf}
  \xymatrix{
    \Pred_\DN^\bullet X
    \ar[d]_{\Pred f}
    &
    \Pred_\ADN^\bullet X
    \ar[l]_{\pi_1}
    \ar[d]|{\Pred_\ADN^\bullet f}
    \ar[r]^{\pi_2}
    &
    \Pred_\AN^\bullet X
    \ar[d]^{\Pred f}
    \\
    \Pred_\DN^\bullet Y
    &
    \Pred_\ADN^\bullet Y
    \ar[l]^{\pi_1}
    \ar[r]_{\pi_2}
    &
    \Pred_\AN^\bullet Y
    }
\end{equation}

\section{The weak distributive law}
\label{sec:weak-distr-law:nat}

\begin{proposition}
  \label{prop:PsiPhi}
  Let $X$ be a topological space.  There is a unique continuous map
  $\Theta \colon \Val_\bullet {\QLV X} \to \Pred_\ADN^\bullet X$ such
  that:
  \[
    \xymatrix{
      \Val_\bullet {\SV X}
      \ar[d]_{\Phi}
      & \Val_\bullet {\QLV X}
      \ar[l]_{\Val_\bullet {\pi_1}}
      \ar[d]|{\Theta}
      \ar[r]^{\Val_\bullet {\pi_2}}
      & \Val_\bullet {\HV X}
      \ar[d]^{\Psi}
      \\
      \Pred_\DN^\bullet X
      & \Pred_\ADN^\bullet X
      \ar[l]^{\pi_1}
      \ar[r]_{\pi_2}
      & \Pred_\AN^\bullet X.
      }
  \]
\end{proposition}
\begin{proof}
  In other words, for every $\xi \in \Val_\bullet {\QLV X}$,
  $\Theta (\xi)$ must be equal to
  $(\Phi (\pi_1 [\xi]), \Psi (\pi_2 [\xi]))$.  It remains to show that
  the latter is a fork; continuity will follow from the continuity of
  $\Phi$ and of $\Psi$ (for the latter, see Lemma~\ref{lemma:Psi}).
  Let $(F^-, F^+) \eqdef (\Phi (\pi_1 [\xi]), \Psi (\pi_2 [\xi]))$.
  We know that $F^- \in \Pred_\DN^\bullet X$, that
  $F^+ \in \Pred_\AN^\bullet$, and it remains to show Walley's
  conditions.

  For every $h \in \Lform X$, we have:
  \begin{align}
    \nonumber
    F^- (h)
    & = \int_{Q \in \SV X} \min_{x \in Q} h (x)
      \,d\pi_1 [\xi] \\
    \label{eq:F-}
    & = \int_{(Q, C) \in \QLV X} \min_{x \in Q} h (x)
      \,d\xi \\
    \nonumber
    F^+ (h)
    & = \int_{C \in \HV X} \sup_{x \in C} h (x)
      \,d\pi_2 [\xi] \\
    \label{eq:F+}
    & = \int_{(Q, C) \in \QLV X} \sup_{x \in C} h (x)
      \,d\xi,
  \end{align}
  using the change of variable formula.
  
  For every $h \in \Lform X$, we have:
  \begin{align*}
    \min_{x \in Q} h (x)
    & \geq \min_{x \in \upc (Q \cap C)} h (x)
    & \text{since $Q \subseteq \upc (Q \cap C)$} \\
    & = \min_{x \in Q \cap C} h (x)
      \geq \min_{x \in Q} h (x),
  \end{align*}
  since $Q \cap C \subseteq C$.  This is a circular list of
  inequalities, hence all its terms are equal. 
  
  For every $h, h' \in \Lform X$, for all $(Q, C) \in \QLV X$, we then
  have:
  \begin{align}
    \nonumber
    \min_{x \in Q} (h (x) + h' (x))
    & = \min_{x \in Q \cap C} (h (x) + h' (x)) \\
    \nonumber
    & \leq \min_{x \in Q \cap C} (h (x) + \sup_{x \in C} h' (x)) \\
    \nonumber
    & = \min_{x \in Q \cap C} h (x) + \sup_{x \in C} h' (x) \\
    \label{eq:min:QLV}
    & = \min_{x \in Q} h (x) + \sup_{x \in C} h' (x).
  \end{align}
  Using (\ref{eq:F-}) and (\ref{eq:F+}), it follows that $F^- (h+h')
  \leq F^- (h) + F^+ (h')$.

  We claim that:
  \begin{align}
    \label{eq:sup:QLV}
    \min_{x \in Q} h (x) + \sup_{x \in C} h' (x)
    & \leq \sup_{x \in C} (h (x) + h' (x)).
  \end{align}
  It suffices to show that for every $t \in \Rp$ such that
  $t < \min_{x \in Q} h (x)$,
  $t + \sup_{x \in C} h' (x) \leq \sup_{x \in C} (h (x) + h' (x))$.
  Since $t < \min_{x \in Q} h (x)$,
  $Q \subseteq h^{-1} (]t, \infty])$.  Since $(Q, C)$ is a quasi-lens,
  $C \subseteq cl (h^{-1} (]t, \infty]) \cap C)$.  Hence
  $\sup_{x \in C} h' (x) \leq \sup_{x \in cl (h^{-1} (]t, \infty])
    \cap C)} h' (x) = \sup_{x \in h^{-1} (]t, \infty]) \cap C} h' (x)$
  (using Fact~\ref{fact:sup:cl}), and since $h (x) \geq t$ for every
  $x \in h^{-1} (]t, \infty]) \cap C$, we obtain that
  $t + \sup_{x \in C} h' (x) \leq \sup_{x \in h^{-1} (]t, \infty])
    \cap C} (h (x) + h' (x)) \leq \sup_{x \in C} (h (x) + h' (x))$.

  Using (\ref{eq:F-}) and (\ref{eq:F+}), it follows from
  (\ref{eq:sup:QLV}) that $F^- (h) + F^+ (h') \leq F^+ (h+h')$.  \qed
\end{proof}

\begin{theorem}
  \label{thm:lambda:nat}
  Let $\bullet$ be nothing, ``$\leq 1$'', or ``$1$''.  For every space
  $X$ such that $\Lform X$ is locally convex and has an almost open
  addition map (and such that $X$ is compact if $\bullet$ is ``$1$''),
  there is a unique continuous map
  $\lambda^\natural_X \colon \Val_\bullet {\QLV X} \to \QLV
  {\Val_\bullet X}$ such that:
  \begin{equation}
    \label{eq:lambda:nat}
    \xymatrix{
      \Val_\bullet {\SV X}
      \ar[d]_{\lambda^\sharp_X}
      & \Val_\bullet {\QLV X}
      \ar[l]_{\Val_\bullet {\pi_1}}
      \ar[d]|{\lambda^\natural_X}
      \ar[r]^{\Val_\bullet {\pi_2}}
      & \Val_\bullet {\HV X}
      \ar[d]^{\lambda^\flat_X}
      \\
      \SV {\Val_\bullet X}
      & \QLV {\Val_\bullet X}
      \ar[l]^{\pi_1}
      \ar[r]_{\pi_2}
      & \HV {\Val_\bullet X}
      }
    \end{equation}
    and that is $s_\ADN^\bullet \circ \Theta$.

    The collection $\lambda^\natural$ of maps $\lambda^\natural_X$ is
    a weak distributive law of $\QLV$ over $\Val_\bullet$ on
    $\Topcat^\natural$.
\end{theorem}
\begin{proof}
  Uniqueness is because $\langle \pi_1, \pi_2 \rangle$ is the
  inclusion map of $\QLV {\Val_\bullet X}$, which is injective.  Its
  existence as $s_\ADN^\bullet \circ \Theta$ follows from
  Proposition~\ref{prop:rADP} and Proposition~\ref{prop:PsiPhi}.  The
  fact that $\lambda^\natural$ is natural and satisfies the three
  equations of weak distributive laws follows from the similar
  properties of $\lambda^\sharp$ (Theorem~\ref{thm:lambda}) and
  $\lambda^\flat$ (Theorem~\ref{thm:lambda:flat}).  Given any morphism
  $f \colon X \to Y$ in $\Topcat^\natural$ (and noticing that the
  requirements on $\Topcat^\natural$ imply those on $\Topcat^\flat$,
  which enables us to reuse all our results on $\lambda^\flat$), we
  have the following diagram:
  \[
    \xymatrix{
      \Val_\bullet {\SV X}
      \ar[rdd]|(0.7){\Val_\bullet {\SV f}}
      \ar[d]_{\lambda^\sharp_X}
      & \Val_\bullet {\QLV X}
      \ar[rdd]|(0.7){\Val_\bullet {\QLV f}}
      \ar[l]_{\Val_\bullet {\pi_1}}
      \ar[d]|{\lambda^\natural_X}
      \ar[r]^{\Val_\bullet {\pi_2}}
      & \Val_\bullet {\HV X}
      \ar[rdd]|(0.7){\Val_\bullet {\HV f}}
      \ar[d]^{\lambda^\flat_X}
      \\
      \SV {\Val_\bullet X}
      \ar[rdd]|(0.3){\SV {\Val_\bullet f}}
      & \QLV {\Val_\bullet X}
      \ar[rdd]|(0.3){\QLV {\Val_\bullet f}}|!{[d];[rd]}{\hole}
      \ar[l]|{\hole}^{\pi_1}
      \ar[r]|{\hole}_{\pi_2}
      & \HV {\Val_\bullet X}
      \ar[rdd]|(0.3){\HV {\Val_\bullet f}}|!{[d];[rd]}{\hole}
      &
      \\
      &
      \Val_\bullet {\SV Y}
      \ar[d]_{\lambda^\sharp_Y}
      & \Val_\bullet {\QLV Y}
      \ar[l]_{\Val_\bullet {\pi_1}}
      \ar[d]|{\lambda^\natural_Y}
      \ar[r]^{\Val_\bullet {\pi_2}}
      & \Val_\bullet {\HV Y}
      \ar[d]^{\lambda^\flat_Y}
      &
      \\
      &
      \SV {\Val_\bullet Y}
      & \QLV {\Val_\bullet Y}
      \ar[l]^{\pi_1}
      \ar[r]_{\pi_2}
      & \HV {\Val_\bullet Y}
    }
  \]
  The two squares on the $X$ (back) face commute by definition of
  $\lambda^\natural_X$, and similarly for the two squares on the $Y$
  (front) face.  The vertical, slanted rectangles on the left and on
  the right commute by naturality of $\lambda^\sharp$, resp.\ of
  $\lambda^\flat$.  The four horizontal (slanted) rectangles commute
  by naturality of $\pi_1$ and of $\pi_2$.  The naturality of
  $\lambda^\natural$ is the commutativity of the middle vertical,
  slanted rectangle, which follows from the fact that all the other
  rectangles commute, and from the fact that
  $\langle \pi_1, \pi_2 \rangle$ is injective.

  The first law of weak distributivity laws,
  ${\lambda^\natural_X} \circ \Val_\bullet {\eta^\quasi_X} =
  \eta^\quasi_{\Val_\bullet X}$, is proved similarly.
  The following diagram commutes:
  \[
    \xymatrix{
      \Val_\bullet X
      \ar[d]_{\Val_\bullet {\eta^\Smyth_X}}
      &
      \Val_\bullet X
      \ar@{=}[l]
      \ar[d]|{\Val_\bullet {\eta^\quasi_X}}
      \ar@{=}[r]
      &
      \Val_\bullet X
      \ar[d]^{\Val_\bullet {\eta^\Hoare_X}}
      \\
      \Val_\bullet {\SV X}
      \ar[d]_{\lambda^\sharp_X}
      & \Val_\bullet {\QLV X}
      \ar[l]_{\Val_\bullet {\pi_1}}
      \ar[d]|{\lambda^\natural_X}
      \ar[r]^{\Val_\bullet {\pi_2}}
      & \Val_\bullet {\HV X}
      \ar[d]^{\lambda^\flat_X}
      \\
      \SV {\Val_\bullet X}
      & \QLV {\Val_\bullet X}
      \ar[l]^{\pi_1}
      \ar[r]_{\pi_2}
      & \HV {\Val_\bullet X}
      }
    \]
    Indeed, the top two squares commute by definition of $\eta^\quasi$
    (see (\ref{eq:etaq}), Proposition~\ref{prop:QLV:monad}), and the
    bottom two squares commute by definition of $\lambda^\natural$.
    The composition of the leftmost two vertical arrows is
    $\eta^\Smyth_{\Val_\bullet X}$, and the composition of the
    rightmost two vertical arrows is $\eta^\Hoare_{\Val_\bullet X}$,
    by the first equation of weak distributive laws for
    $\lambda^\sharp$ and $\lambda^\flat$ respectively.  Hence the
    composition of the middle two vertical arrows is
    $\eta^\quasi_{\Val_\bullet X}$, by the uniqueness of $\eta^\quasi$
    in (\ref{eq:etaq}).
  
    We turn to the second weak distributivity law
    $\lambda^\natural_X \circ \Val_\bullet \mu^\quasi_X =
    \mu^\quasi_{\Val_\bullet X} \circ \QLV {\lambda^\natural_X} \circ
    \lambda^\natural_{\QLV X}$.  The left-hand side is composition of
    the middle vertical arrows in the following diagram:
  \[
    \xymatrix@C-5pt{
      \Val_\bullet {\SV {\SV X}}
      \ar[d]_{\Val_\bullet {\mu^\Smyth_X}}
      &
      \Val_\bullet {\SV {\QLV X}}
      \ar[l]_{\Val_\bullet {\SV {\pi_1}}}
      &
      \Val_\bullet {\QLV {\QLV X}}
      \ar[l]_{\Val_\bullet {\pi_1}}
      \ar[d]|{\Val_\bullet {\mu^\quasi_X}}
      \ar[r]^{\Val_\bullet {\pi_2}}
      &
      \Val_\bullet {\HV {\QLV X}}
      \ar[r]^{\Val_\bullet {\HV {\pi_2}}}
      &
      \Val_\bullet {\HV {\HV X}}
      \ar[d]^{\Val_\bullet {\mu^\Hoare_X}}
      \\
      \Val_\bullet {\SV X}
      \ar[d]_{\lambda^\sharp_X}
      &&
      \Val_\bullet {\QLV X}
      \ar[d]|{\lambda^\natural_X}
      \ar[ll]|{\Val_\bullet {\pi_1}}
      \ar[rr]|{\Val_\bullet {\pi_2}}
      &&
      \Val_\bullet {\HV X}
      \ar[d]^{\lambda^\flat_X}
      \\
      \SV {\Val_\bullet X}
      && \QLV {\Val_\bullet X}
      \ar[ll]^{\pi_1}
      \ar[rr]_{\pi_2}
      && \HV {\Val_\bullet X}
      }
    \]
    where the top two squares are by (\ref{eq:muq}), and the bottom
    two squares are by definition of $\lambda^\natural$.  The
    right-hand side appears as the composition of middle vertical
    arrows in the following diagram:
    \[
      \xymatrix@C-5pt{
      \Val_\bullet {\SV {\SV X}}
      \ar[dd]_{\lambda^\sharp_{\SV X}}
      &
      \Val_\bullet {\SV {\QLV X}}
      \ar[d]|{\lambda^\sharp_{\QLV X}}
      \ar[l]_{\Val_\bullet {\SV {\pi_1}}}
      \ar@{}[ld]|{\text{nat.\ $\lambda^\sharp$}}
      &
      \Val_\bullet {\QLV {\QLV X}}
      \ar[l]_{\Val_\bullet {\pi_1}}
      \ar@{}[ld]|{\text{def.\ $\lambda^\natural$}}
      \ar@{}[rd]|{\text{def.\ $\lambda^\natural$}}
      \ar[d]|{\lambda^\natural_{\QLV X}}
      \ar[r]^{\Val_\bullet {\pi_2}}
      &
      \Val_\bullet {\HV {\QLV X}}
      \ar[d]|{\lambda^\flat_{\QLV X}}
      \ar[r]^{\Val_\bullet {\HV {\pi_2}}}
      \ar@{}[rd]|{\text{nat.\ $\lambda^\flat$}}
      &
      \Val_\bullet {\HV {\HV X}}
      \ar[dd]^{\lambda^\flat_{\HV X}}
      \\
      & \SV {\Val_\bullet {\QLV X}}
      \ar[ld]|{\SV {\Val_\bullet {\pi_1}}}
      \ar@{}[d]|{\text{nat. $\pi_1$}}
      & \QLV {\Val_\bullet {\QLV X}}
      \ar[l]|{\pi_1}
      \ar[ld]|{\QLV {\Val_\bullet {\pi_1}}}
      \ar[dd]|{\QLV {\lambda^\natural_X}}
      \ar[rd]|{\QLV {\Val_\bullet {\pi_2}}}
      \ar[r]|{\pi_2}
      & \HV {\Val_\bullet {\QLV X}}
      \ar[rd]|{\HV {\Val_\bullet {\pi_1}}}
      \ar@{}[d]|{\text{nat. $\pi_2$}}
      &  
      \\
      \SV {\Val_\bullet {\SV X}}
      \ar[d]_{\SV {\lambda^\sharp_X}}
      &
      \QLV {\Val_\bullet {\SV X}}
      \ar[l]|{\pi_1}
      \ar@{}[ld]|{\text{nat.\ $\pi_1$}}
      \ar[d]|{\QLV {\lambda^\sharp_X}}
      &
      \ar@{}[ld]|{\text{def.\ $\lambda^\natural$}}
      \ar@{}[rd]|{\text{def.\ $\lambda^\natural$}}
      &
      \QLV {\Val_\bullet {\HV  X}}
      \ar[r]|{\pi_2}
      \ar@{}[rd]|{\text{nat.\ $\pi_2$}}
      \ar[d]|{\QLV {\lambda^\flat_X}}
      &
      \HV {\Val_\bullet {\HV X}}
      \ar[d]^{\HV {\lambda^\flat_X}}
      \\
      \SV {\SV {\Val_\bullet X}}
      \ar[d]_{\mu^\Smyth_X}
      & \QLV {\SV {\Val_\bullet X}}
      \ar[l]|{\pi_1}
      &
      \QLV {\QLV {\Val_\bullet X}}
      \ar[l]^{\QLV {\pi_1}}
      \ar@{}[lld]|{\text{(\ref{eq:muq})}}
      \ar@{}[rrd]|{\text{(\ref{eq:muq})}}
      \ar[d]|{\mu^\quasi_{\Val_\bullet X}}
      \ar[r]_{\QLV {\pi_2}}
      & \QLV {\HV {\Val_\bullet X}}
      \ar[r]|{\pi_2}
      &
      \HV {\HV {\Val_\bullet X}}
      \ar[d]^{\mu^\Hoare_X}
      \\
      \SV {\Val_\bullet X}
      &
      & \QLV {\Val_\bullet X}
      \ar[ll]^{\pi_1}
      \ar[rr]_{\pi_2}
      &
      & \HV {\Val_\bullet X}
    }
  \]
  Those two diagrams have the same (composition of) leftmost vertical
  arrows, and similarly for the rightmost vertical arrows, because
  $\lambda^\sharp$ and $\lambda^\flat$ are weak distributivity laws.
  Since $\langle \pi_1, \pi_2 \rangle$ is injective, the (composition
  of) middle vertical arrows are the same, too.

  Finally, we address the third weak distributivity law
  ${\lambda^\natural_X} \circ \mu_{\QLV X} = \QLV {\mu_X} \circ
  \lambda^\natural_{\Val_\bullet X} \circ \Val_\bullet
  {\lambda^\natural_X}$.  As above, this consists in checking that the
  leftmost compositions of vertical arrows in the following two
  diagrams match, so that the middle compositions of vertical arrows
  coincide, too.
  \[
    \xymatrix{
      \Val_\bullet {\Val_\bullet {\SV X}}
      \ar[d]_{\mu_{\SV X}}
      &
      \Val_\bullet {\Val_\bullet {\QLV X}}
      \ar[l]_{\Val_\bullet {\Val_\bullet {\pi_1}}}
      \ar@{}[ld]|{\text{nat. $\mu$}}
      \ar@{}[rd]|{\text{nat. $\mu$}}
      \ar[d]|{\mu_{\QLV X}}
      \ar[r]^{\Val_\bullet {\Val_\bullet {\pi_2}}}
      &
      \Val_\bullet {\Val_\bullet {HV X}}
      \ar[d]^{\mu_{\HV X}}
      \\
      \Val_\bullet {\SV X}
      \ar[d]_{\lambda^\sharp_X}
      &
      \Val_\bullet {\QLV X}
      \ar[d]|{\lambda^\natural_X}
      \ar[l]|{\Val_\bullet {\pi_1}}
      \ar@{}[ld]|{\text{def.\ $\lambda^\natural$}}
      \ar@{}[rd]|{\text{def.\ $\lambda^\natural$}}
      \ar[r]|{\Val_\bullet {\pi_2}}
      &
      \Val_\bullet {\HV X}
      \ar[d]^{\lambda^\flat_X}
      \\
      \SV {\Val_\bullet X}
      & \QLV {\Val_\bullet X}
      \ar[l]^{\pi_1}
      \ar[r]_{\pi_2}
      & \HV {\Val_\bullet X}
    }
  \]
  Here is the second diagram:
  \[
    \xymatrix{
      \Val_\bullet {\Val_\bullet {\SV X}}
      \ar[d]_{\Val_\bullet {\lambda^\sharp_X}}
      &
      \Val_\bullet {\Val_\bullet {\QLV X}}
      \ar[l]_{\Val_\bullet {\Val_\bullet {\pi_1}}}
      \ar@{}[ld]|{\text{def. $\lambda^\natural$}}
      \ar@{}[rd]|{\text{def. $\lambda^\natural$}}
      \ar[d]|{\Val_\bullet {\lambda^\natural_X}}
      \ar[r]^{\Val_\bullet {\Val_\bullet {\pi_2}}}
      &
      \Val_\bullet {\Val_\bullet {HV X}}
      \ar[d]^{\Val_\bullet {\lambda^\flat_X}}
      \\
      \Val_\bullet {\SV {\Val_\bullet X}}
      \ar[d]^{\lambda^\sharp_X}
      &
      \Val_\bullet {\QLV {\Val_\bullet X}}
      \ar[l]_{\Val_\bullet {\pi_1}}
      \ar@{}[ld]|{\text{def. $\lambda^\natural$}}
      \ar@{}[rd]|{\text{def. $\lambda^\natural$}}
      \ar[d]|{\lambda^\natural_X}
      \ar[r]^{\Val_\bullet {\pi_2}}
      &
      \Val_\bullet {\HV {\Val_\bullet X}}
      \ar[d]^{\lambda^\flat_X}
      \\
      \SV {\Val_\bullet {\Val_\bullet X}}
      \ar[d]_{\SV {\mu_X}}
      &
      \QLV {\Val_\bullet {\Val_\bullet X}}
      \ar[l]|{\pi_1}
      \ar@{}[ld]|{\text{nat. $\pi_1$}}
      \ar@{}[rd]|{\text{nat. $\pi_2$}}
      \ar[d]|{\QLV {\mu_X}}
      \ar[r]|{\pi_2}
      &
      \HV {\Val_\bullet {\Val_\bullet X}}
      \ar[d]^{\HV {\mu_X}}
      \\
      \SV {\Val_\bullet X}
      &
      \QLV {\Val_\bullet X}
      \ar[l]^{\pi_1}
      \ar[r]_{\pi_2}
      &
      \HV {\Val_\bullet X}
    }
  \]
  \qed
\end{proof}

\begin{remark}
  \label{rem:PhiPsi}
  Using Remark~\ref{rem:Phi} and Remark~\ref{rem:Psi}, in the special
  case where $\mu \eqdef \sum_{i=1}^n a_i \delta_{(Q_i, C_i)}$, where
  $n \geq 1$, and where $Q_i \eqdef \upc E_i$, $C_i \eqdef \dc E_i$,
  where each set $E_i$ is finite and non-empty, then
  $\lambda^\natural_X (\mu) = (\upc conv (\mathcal E), cl (conv
  (\mathcal E)))$, where
  $\mathcal E \eqdef \{\sum_{i=1}^n a_i \delta_{x_i} \mid x_1 \in E_1,
  \cdots, x_n \in E_n\}$.
\end{remark}

\begin{remark}
  \label{rem:lambda:nat}
  We recall that if $X$ is weakly Hausdorff or quasi-Polish, then
  $\QLV X \cong \Plotkinn X$, and that those isomorphisms allow us to
  transport the $\QLV$ monad structure over to any full subcategory of
  $\Topcat$ consisting of weakly Hausdorff or quasi-Polish spaces.
  Theorem~\ref{thm:lambda:nat} then allows us to obtain a weak
  distributive law of $\Plotkinn$ over $\Val_\bullet$ on any full
  subcategory of $\Topcat^\natural$ consisting of weakly Hausdorff or
  quasi-Polish spaces---in particular on the category of stably
  compact spaces, which are all weakly Hausdorff.
\end{remark}

\begin{remark}
  \label{rem:PhiPsi:lens}
  Combining Remark~\ref{rem:PhiPsi} and Remark~\ref{rem:lambda:nat},
  in the special case where
  $\mu \eqdef \sum_{i=1}^n a_i \delta_{L_i}$, where $n \geq 1$, and
  where $L_i \eqdef \upc E_i \cap \dc E_i$, where each set $E_i$ is
  finite and non-empty, then
  $\lambda^\natural_X (\mu) = \upc conv (\mathcal E) \cap cl (conv
  (\mathcal E))$, where
  $\mathcal E \eqdef \{\sum_{i=1}^n a_i \delta_{x_i} \mid x_1 \in E_1,
  \cdots, x_n \in E_n\}$.
\end{remark}

\section{The associated weak composite monad}
\label{sec:assoc-weak-comp}

We proceed as with $\Pred_\DN^\bullet$ and $\Pred_\AN^\bullet$, and in
fact we reduce to the cases of those monads.  We start with the
idempotent:
\begin{equation}
  \label{eq:idemp:nat}
  \xymatrix{
    \QLV {\Val_\bullet X} \ar[r]^{\eta_{\QLV {\Val_\bullet X}}}
    & \Val_\bullet {\QLV {\Val_\bullet X}}
    \ar[r]^{\lambda^\natural_{\Val_\bullet X}}
    & \QLV {\Val_\bullet {\Val_\bullet X}}
    \ar[r]^{\QLV {\mu_X}}
    & \QLV {\Val_\bullet X}
  }
\end{equation}
whose splitting will define the action of the functor part of the
associated weak composite monad on $X$.
\begin{lemma}
  \label{lemma:idemp:split:nat}
  For every object $X$ of $\Topcat^\natural$, a splitting of
  (\ref{eq:idemp:nat}) is
  $\xymatrix{\QLV {\Val_\bullet X} \ar[r]^{r_\ADN} &
    \Pred_\ADN^\bullet X \ar[r]^{s_\ADN^\bullet} & \QLV {\Val_\bullet
      X}}$.
\end{lemma}
\begin{proof}
  Composing (\ref{eq:idemp:nat}) with $\pi_1$, resp.\ with $\pi_2$ on
  the right (of the diagram---namely computing $\pi_1 \circ f$ resp.\
  $\pi_2 \circ f$ where $f$ is (\ref{eq:idemp:nat})), and simplifying
  using the naturality of $\pi_1$ (resp., $\pi_2$),
  (\ref{eq:lambda:nat}) (see Theorem~\ref{thm:lambda:nat}) and
  (\ref{eq:etaADN}) (see Proposition~\ref{prop:fork:monad}), we obtain
  (\ref{eq:idemp}), resp.\ (\ref{eq:idemp:flat})---the corresponding
  idempotents in the $\DN$ and $\AN$ cases---composed with $\pi_1$,
  resp.\ $\pi_2$.

  By Lemma~\ref{lemma:idemp:split} (resp.,
  Lemma~\ref{lemma:idemp:split:flat}), this is equal to
  $s_\DN^\bullet \circ r_\DN \circ \pi_1$ (resp.,
  $s_\AN^\bullet \circ r_\AN \circ \pi_2$), and that, in turn, is
  equal to $\pi_1 \circ s_\ADN^\bullet \circ r_\ADN$ (resp.,
  $\pi_2 \circ s_\ADN^\bullet \circ r_\ADN$), by
  Proposition~\ref{prop:rADP}.  We conclude since
  $\langle \pi_1, \pi_2 \rangle$ is an inclusion map, hence is
  monotonic.  \qed
\end{proof}
Generalizing (\ref{eq:idemp:nat}), there is an idempotent:
\begin{equation}
  \label{eq:idemp:alpha:nat}
  \xymatrix{
    \QLV X \ar[r]^{\eta_{\QLV X}}
    & \Val_\bullet {\QLV X}
    \ar[r]^{\lambda^\natural_X}
    & \QLV {\Val_\bullet X}
    \ar[r]^{\QLV \alpha}
    & \QLV X
  }
\end{equation}
for every $\Val_\bullet$-algebra $\alpha$.  As before, we will only
need to compute this splitting for one additional
$\Val_\bullet$-algebra.  We write
${\QLV}^{cvx} {\Pred_\ADN^\bullet X}$ for the subspace of
$\QLV {\Pred_\ADN^\bullet X}$ consisting of \emph{convex}
quasi-lenses, namely quasi-lenses $(\mathcal Q, \mathcal C)$ such that
both $\mathcal Q$ and $\mathcal C$ are convex.
\begin{lemma}
  \label{lemma:idemp:P:nat}
  For every topological space $X$, there is a $\Val_\bullet$-algebra
  $\alpha_\ADN^\bullet \colon \Val_\bullet {\Pred_\ADN^\bullet X} \to
  \Pred_\ADN^\bullet X$ defined by
  $\alpha_\ADN^\bullet (\xi) \eqdef (h \in \Lform X \mapsto
  \alpha_\DN^\bullet (\pi_1 [\xi]), h \in \Lform X \mapsto
  \alpha_\AN^\bullet (\pi_2 [\xi]))$ for every
  $\xi \in \Val_\bullet {\Pred_\ADN^\bullet X}$.

  If $X$ is an object of $\Topcat^\natural$, then a splitting of the
  idempotent (\ref{eq:idemp:alpha:nat}) where
  $\alpha \eqdef \alpha_\ADN^\bullet$ is:
  \begin{align*}
    \xymatrix{
    \QLV {\Pred_\ADN^\bullet X} \ar[r]
    & {\QLV}^{cvx} {\Pred_\ADN^\bullet X} \ar[r]
    & \QLV {\Pred_\ADN^\bullet X,}}
  \end{align*}
  where the arrow on the right is the inclusion map, and the arrow on
  the left is the corestriction of (\ref{eq:idemp:alpha:nat}), and
  maps every $(Q, C) \in {\QLV}^{cvx} {\Pred_\ADN^\bullet X}$ to
  itself.
\end{lemma}
\begin{proof}
  For every $\xi \in \Val_\bullet {\Pred_\ADN^\bullet X}$, let
  $(F_0^-, F_0^+) \eqdef \alpha_\ADN^\bullet (\xi)$.  For every
  $h \in \Lform X$,
  $F_0^- (h) = \alpha_\DN^\bullet (\pi_1 [\xi]) (h) = \int_{F \in
    \Pred_\DN^\bullet X} F (h) \,d\pi_1 [\xi] = \int_{(F^-, F^+) \in
    \Pred_\ADN^\bullet X} F^- (h) \,d\xi$, by the change of variable
  formula, and similarly
  $F_0^+ (h) = \int_{(F^-, F^+) \in \Pred_\ADN^\bullet X} F^+ (h)
  \,d\xi$.  Hence, for all $h, h' \in \Lform X$,
  \begin{align*}
    F_0^- (h+h')
    & = \int_{(F^-, F^+) \in \Pred_\ADN^\bullet X} F^- (h+h') \,d\xi
    \\
    & \leq \int_{(F^-, F^+) \in \Pred_\ADN^\bullet X} (F^-
      (h)+F^+(h')) \,d\xi \\
    & = \int_{(F^-, F^+) \in \Pred_\ADN^\bullet X} F^- (h) \,d\xi
      + \int_{(F^-, F^+) \in \Pred_\ADN^\bullet X} F^+ (h') \,d\xi
    \\
    & = F_0^- (h) + F_0^+ (h'),
  \end{align*}
  and:
  \begin{align*}
    F_0^- (h) + F_0^+ (h')
    & = \int_{(F^-, F^+) \in \Pred_\ADN^\bullet X} (F^-
      (h)+F^+(h')) \,d\xi
    \\
    & \leq \int_{(F^-, F^+) \in \Pred_\ADN^\bullet X} F^+ (h+h')
      \,d\xi \\
    & = F_0^+ (h+h'),
  \end{align*}
  so $(F_0^-, F_0^+) \in \Pred_\ADN^\bullet X$.  The map
  $\alpha_\ADN^\bullet$ is continuous, since the inverse image of the
  subbasic open set $[h > r]^-$ is
  $[((F^+, F^-) \mapsto F^- (h)) > r]$, and the inverse image of the
  subbasic open set $[h > r]^+$ is
  $[((F^+, F^-) \mapsto F^+ (h)) > r]$.

  The map $\alpha_\ADN^\bullet$ is defined so that
  $\pi_1 \circ \alpha_\ADN^\bullet = \alpha_\DN^\bullet \circ
  \Val_\bullet {\pi_1}$ and
  $\pi_2 \circ \alpha_\ADN^\bullet = \alpha_\AN^\bullet \circ
  \Val_\bullet {\pi_2}$.  The fact that $\alpha_\ADN^\bullet$ is a
  $\Val_\bullet$-algebra then follows from the fact that
  $\alpha_\DN^\bullet$ and $\alpha_\AN^\bullet$ are.

  The idempotent (\ref{eq:idemp:alpha:nat}) with
  $\alpha \eqdef \alpha_\ADN^\bullet$ appears as the top row of the
  following diagram.
  \[
    \xymatrix@C+15pt{
      \QLV {\Pred_\ADN^\bullet X} \ar[r]^{\eta_{\QLV X}}
      \ar[d]_{\pi_1}
      \ar@{}[rd]|{\text{nat.\ $\pi_1$}}
      & \Val_\bullet {\QLV {\Pred_\ADN^\bullet X}}
      \ar[r]^{\lambda^\natural_{\Pred_\ADN^\bullet X}}
      \ar[d]|{\Val_\bullet {\pi_1}}
      \ar@{}[rd]|{(\ref{eq:lambda:nat})}
      & \QLV {\Val_\bullet {\Pred_\ADN^\bullet X}}
      \ar[r]^{\QLV {\alpha_\ADN^\bullet}}
      \ar[d]|{\pi_1}
      \ar@{}[rd]|{\text{nat.\ $\pi_1$}}
      & \QLV {\Pred_\ADN^\bullet X}
      \ar[d]^{\pi_1}
      \\
      \SV {\Pred_\ADN^\bullet X}
      \ar[r]|{\eta_{\SV {\Pred_\ADN^\bullet X}}}
      \ar[d]_{\SV {\pi_1}}
      \ar@{}[rd]|{\text{nat.\ $\eta$}}
      &
      \Val_\bullet {\SV {\Pred_\ADN^\bullet X}}
      \ar[r]|{\lambda^\sharp_{\Pred_\ADN^\bullet X}}
      \ar[d]|{\Val_\bullet {\SV {\pi_1}}}
      \ar@{}[rd]|{\text{nat.\ $\lambda^\sharp$}}
      &
      \SV {\Val_\bullet {\Pred_\ADN^\bullet X}}
      \ar[r]|{\SV {\alpha_\ADN^\bullet}}
      \ar[d]|{\SV {\Val_\bullet {\pi_1}}}
      \ar@{}[rd]|{\text{def.\ $\alpha_\ADN^\bullet$}}
      &
      \SV {\Pred_\ADN^\bullet X}
      \ar[d]^{\SV {\pi_1}}
      \\
      \SV {\Pred_\DN^\bullet X}
      \ar[r]_{\eta_{\SV {\Pred_\DN^\bullet X}}}
      &
      \Val_\bullet {\SV {\Pred_\DN^\bullet X}}
      \ar[r]_{\lambda^\sharp_{\Pred_\DN^\bullet X}}
      &
      \SV {\Val_\bullet {\Pred_\DN^\bullet X}}
      \ar[r]_{\SV {\alpha_\DN^\bullet}}
      &
      \SV {\Pred_\DN^\bullet X}
    }
  \]
  The bottom row is the idempotent (\ref{eq:idemp:alpha}) of
  Lemma~\ref{lemma:idemp:P} (with $\alpha \eqdef \alpha_\DN^\bullet$),
  which splits as
  $\xymatrix{\SV {\Pred_\DN^\bullet X} \ar[r]^{r^\sharp} & \SV^{cvx}
    {\Pred_\DN^\bullet X} \ar[r]^{s^\sharp} & \SV {\Pred_\DN^\bullet
      X}}$, where $s^\sharp$ is the inclusion map and $r^\sharp$ maps
  every $Q \in \SV^{cvx} {\Pred_\DN^\bullet X}$ to itself.

  Similarly, (\ref{eq:idemp:alpha:nat}) also appears as the top row of
  the following diagram.
  \[
    \xymatrix@C+15pt{
      \QLV {\Pred_\ADN^\bullet X} \ar[r]^{\eta_{\QLV X}}
      \ar[d]_{\pi_2}
      \ar@{}[rd]|{\text{nat.\ $\pi_2$}}
      & \Val_\bullet {\QLV {\Pred_\ADN^\bullet X}}
      \ar[r]^{\lambda^\natural_{\Pred_\ADN^\bullet X}}
      \ar[d]|{\Val_\bullet {\pi_2}}
      \ar@{}[rd]|{(\ref{eq:lambda:nat})}
      & \QLV {\Val_\bullet {\Pred_\ADN^\bullet X}}
      \ar[r]^{\QLV {\alpha_\ADN^\bullet}}
      \ar[d]|{\pi_2}
      \ar@{}[rd]|{\text{nat.\ $\pi_2$}}
      & \QLV {\Pred_\ADN^\bullet X}
      \ar[d]^{\pi_2}
      \\
      \HV {\Pred_\ADN^\bullet X}
      \ar[r]|{\eta_{\HV {\Pred_\ADN^\bullet X}}}
      \ar[d]_{\HV {\pi_2}}
      \ar@{}[rd]|{\text{nat.\ $\eta$}}
      &
      \Val_\bullet {\HV {\Pred_\ADN^\bullet X}}
      \ar[r]|{\lambda^\flat_{\Pred_\ADN^\bullet X}}
      \ar[d]|{\Val_\bullet {\HV {\pi_2}}}
      \ar@{}[rd]|{\text{nat.\ $\lambda^\flat$}}
      &
      \HV {\Val_\bullet {\Pred_\ADN^\bullet X}}
      \ar[r]|{\HV {\alpha_\ADN^\bullet}}
      \ar[d]|{\HV {\Val_\bullet {\pi_2}}}
      \ar@{}[rd]|{\text{def.\ $\alpha_\ADN^\bullet$}}
      &
      \HV {\Pred_\ADN^\bullet X}
      \ar[d]^{\HV {\pi_2}}
      \\
      \HV {\Pred_\AN^\bullet X}
      \ar[r]_{\eta_{\HV {\Pred_\AN^\bullet X}}}
      &
      \Val_\bullet {\HV {\Pred_\AN^\bullet X}}
      \ar[r]_{\lambda^\flat_{\Pred_\AN^\bullet X}}
      &
      \HV {\Val_\bullet {\Pred_\AN^\bullet X}}
      \ar[r]_{\HV {\alpha_\AN^\bullet}}
      &
      \HV {\Pred_\AN^\bullet X}
    }
  \]
  The bottom row is the idempotent (\ref{eq:idemp:alpha:flat}) of
  Lemma~\ref{lemma:idemp:P:flat} (with
  $\alpha \eqdef \alpha_\AN^\bullet$), which splits as
  $\xymatrix{\HV {\Pred_\AN^\bullet X} \ar[r]^{r^\flat} & \HV^{cvx}
    {\Pred_\AN^\bullet X} \ar[r]^{s^\flat} & \HV {\Pred_\AN^\bullet
      X}}$, where $s^\flat$ is the inclusion map and $r^\flat$ maps
  every $C \in \HV^{cvx} {\Pred_\AN^\bullet X}$ to itself.

  Let $r^\natural$ map every
  $(\mathcal Q, \mathcal C) \in \QLV {\Pred_\ADN^\bullet X}$ to
  $(r^\sharp (\SV {\pi_1} (\mathcal Q)), r^\flat (\HV {\pi_2}
  (\mathcal C)))$.  This is a pair in
  $\SV^{cvx} {\Pred_\ADN^\bullet X} \times \HV^{cvx}
  {\Pred_\ADN^\bullet X}$.  We need to show that this is also a
  quasi-lens, and that $r^\natural$ is continuous.  We observe that
  $\pi_1 \circ r^\natural = r^\sharp \circ \SV {\pi_1} \circ \pi_1$ is
  the corestriction of the idempotent (\ref{eq:idemp:alpha}) to
  $\SV^{cvx} {\Pred_\ADN^\bullet X}$, and
  $\pi_2 \circ r^\natural = r^\flat \circ \HV {\pi_1} \circ \pi_1$ is
  the corestriction of the idempotent (\ref{eq:idemp:alpha:flat}) to
  $\HV^{cvx} {\Pred_\ADN^\bullet X}$, so $r^\sharp$ is the
  corestriction of the idempotent (\ref{eq:idemp:alpha:nat}) to
  ${\QLV}^{cvx} {\Pred_\ADN^\bullet X}$.  In particular, $r^\natural$
  takes its values in (convex) quasi-lenses and is continuous.

  Since $r^\sharp$ maps convex elements to themselves, and similarly
  for $r^\flat$, $r^\natural$ maps convex quasi-lenses to themselves.

  We also define
  $s^\natural \colon {\QLV}^{cvx} {\Pred_\ADN^\bullet X} \to \QLV
  {\Pred_\ADN^\bullet X}$ as the inclusion map.  Since $r^\natural$
  maps convex quasi-lenses to themselves,
  $r^\natural \circ s^\natural = \identity {{\QLV}^{cvx}
    {\Pred_\ADN^\bullet X}}$.  Also, by definition of $r^\natural$ as
  the corestriction of the idempotent (\ref{eq:idemp:alpha:nat}) to
  ${\QLV}^{cvx} {\Pred_\ADN^\bullet X}$, (\ref{eq:idemp:alpha:nat}) is
  equal to $s^\natural \circ r^\natural$.  \qed
\end{proof}

\begin{lemma}
  \label{lemma:Sbar:nat}
  For $S \eqdef \QLV$, $T \eqdef \Val_\bullet$,
  $\lambda \eqdef \lambda^\natural$, the monad $\overline S$ on the
  category of $\Val_\bullet$-algebras over $\Topcat^\natural$ has the
  following properties:
  \begin{enumerate}
  \item For every free $\Val_\bullet$-algebra
    $\mu_X \colon \Val_\bullet {\Val_\bullet X} \to \Val_\bullet X$,
    $\overline S \mu_X \colon \Val_\bullet {\Pred_\ADN^\bullet X} \to
    \Pred_\ADN^\bullet X$ is $\alpha_\ADN^\bullet$ (see
    Lemma~\ref{lemma:idemp:P:nat}).
  \item The unit $\eta^{\overline S}_{\mu_X}$ evaluated at the free
    $\Val_\bullet$-algebra
    $\mu_X \colon \Val_\bullet {\Val_\bullet X} \to \Val_\bullet X$ is
    $r_\ADN \circ \eta^\quasi_{\Val_\bullet X}$.
  \item The multiplication $\mu^{\overline S}_{\mu_X}$ evaluated at
    the free $\Val_\bullet$-algebra
    $\mu_X \colon \Val_\bullet {\Val_\bullet X} \to \Val_\bullet X$ is
    the function that maps every
    $(Q, C) \in {\QLV}^{cvx} {\Pred_\ADN^\bullet X}$ to
    $(h \in \Lform X \mapsto \min_{(F^-, F^+) \in Q} F^- (h), h \in
    \Lform X \mapsto \sup_{(F^+, F^-) \in C} F^+ (h)) \in
    \Pred_\ADN^\bullet X$.
  \item The counit $\epsilon^T$ of the adjunction $F^T \dashv U^T$
    between $\Topcat$ and the category of $T$-algebras
    ($\Val_\bullet$-algebras) is given at each algebra $\beta \colon
    \Val_\bullet X \to X$ as $\beta$ itself.
  \end{enumerate}
\end{lemma}
\begin{proof}
  1.  We know that $\overline S \mu_X$ must be a
  $\Val_\bullet$-algebra on $\overline S {\Val_\bullet X}$, and the
  latter can be taken as $\Pred_\ADN^\bullet X$, by
  Lemma~\ref{lemma:idemp:split:nat}.  We know that
  $\overline S \mu_X = \pi_{\mu_X} \circ S \mu_X \circ \lambda_X \circ
  T\iota_{\mu_X} = r_\ADN \circ \QLV \mu_X \circ {\lambda^\natural_X}
  \circ \Val_\bullet {s_\ADN^\bullet}$.  Hence
  $\pi_1 \circ \overline S \mu_X = r_\DN \circ \SV \mu_X \circ
  \lambda^\sharp_X \circ \Val_\bullet {s_\DN^\bullet} \circ
  \Val_\bullet {\pi_1}$ using the definition of $r_\ADN$
  (Proposition~\ref{prop:rADP}), the naturality of $\pi_1$, the
  definition of $\lambda^\natural$ (Theorem~\ref{thm:lambda:nat}), and
  the definition of $s_\ADN^\bullet$ (Proposition~\ref{prop:rADP}).
  Hence, by Lemma~\ref{lemma:Sbar}, item~1,
  $\pi_1 \circ \overline S \mu_X = \alpha_\DN^\bullet \circ
  \Val_\bullet {\pi_1}$.  Similarly,
  $\pi_2 \circ \overline S \mu_X = r_\AN \circ \HV \mu_X \circ
  \lambda^\flat_X \circ \Val_\bullet {s_\AN^\bullet} \circ
  \Val_\bullet {\pi_2}$, so
  $\pi_2 \circ \overline S \mu_X = \alpha_\AN^\bullet \circ
  \Val_\bullet {\pi_2}$, using Lemma~\ref{lemma:Sbar:flat}.  We also
  have
  $\pi_1 \circ \alpha_\ADN^\bullet = \alpha_\DN^\bullet \circ
  \Val_\bullet {\pi_1}$ and
  $\pi_2 \circ \alpha_\ADN^\bullet = \alpha_\AN^\bullet \circ
  \Val_\bullet {\pi_2}$, by definition of $\alpha_\ADN^\bullet$
  (Lemma~\ref{lemma:idemp:P:nat}).  Since
  $\langle \pi_1, \pi_2 \rangle$ is injective,
  $\overline S \mu_X = \alpha_\ADN^\bullet$.
  
  2.
  $\eta^{\overline S}_{\mu_X} = \pi_{\mu_X} \circ \eta^S_{\Val_\bullet
    X} = r_\ADN \circ \eta^\quasi_{\Val_\bullet X}$.

  3. In general, $\mu^{\overline S}_\alpha$ is a morphism from
  $\overline S {\overline S \alpha}$ to $\overline S \alpha$ in the
  category of $\Val_\bullet$-algebras.  When
  $\alpha = \mu_X \colon \Val_\bullet {\Val_\bullet X} \to
  \Val_\bullet X$, it is therefore one from
  $\Pred_\ADN^\bullet {\Pred_\ADN^\bullet X}$ to
  $\Pred_\ADN^\bullet X$.  We have
  $\mu^{\overline S}_{\mu_X} \eqdef \pi_{\mu_X} \circ
  \mu^S_{\Val_\bullet X} \circ S \iota_{\mu_X} \circ \iota_{\overline
    S \mu_X}$.  We may take
  $\iota_{\overline S \mu_X} = \iota_{\alpha_\ADN^\bullet}$ to be the
  right arrow, from ${\QLV}^{cvx} {\Pred_\ADN^\bullet X}$ to
  $\QLV {\Pred_\ADN^\bullet X}$ in
  Lemma~\ref{lemma:idemp:P:nat}---just subspace inclusion.  Let
  $m^\natural$ be that subspace inclusion.  Writing $m^\sharp$ for the
  subspace inclusion of $\SV^{cvx} {\Pred_\DN^\bullet X}$ into
  $\SV {\Pred_\DN^\bullet X}$ and $m^\flat$ for the subspace inclusion
  of $\HV^{cvx} {\Pred_\AN^\bullet X}$ into
  $\HV {\Pred_\AN^\bullet X}$, we have
  $\SV {\pi_1} \circ \pi_1 \circ m^\natural = m^\sharp \circ \SV
  {\pi_1} \circ \pi_1$ and
  $\HV {\pi_2} \circ \pi_2 \circ m^\natural = m^\flat \circ \HV
  {\pi_2} \circ \pi_2$.
  
  Then
  $S \iota_{\mu_X} = \QLV {s_\ADN^\bullet} \colon \QLV
  {\Pred_\ADN^\bullet X} \to \QLV {\QLV {\Val_\bullet X}}$, by
  Lemma~\ref{lemma:idemp:split:nat}.  We compose that with
  $\mu^\quasi_{\Val_\bullet X} \colon \QLV {\QLV {\Val_\bullet X}} \to
  \QLV {\Val_\bullet X}$, then with
  $r_\ADN \colon \QLV {\Val_\bullet X} \to \Pred_\ADN^\bullet X$.  The
  resulting composition is
  $f \eqdef r_\ADN \circ \mu^\quasi_{\Val_\bullet X} \circ \QLV
  {s_\ADN^\bullet} \circ m^\natural$.

  Then
  $\pi_1 \circ f = r_\DN \circ \mu^\Smyth_{\Val_\bullet X} \circ \SV
  {\pi_1} \circ \pi_1 \circ \QLV {s_\ADN^\bullet} \circ m^\natural$ by
  definition of $r_\ADN$ (Proposition~\ref{prop:rADP}) and by
  (\ref{eq:muq}) (see Proposition~\ref{prop:QLV:monad}).  Now
  $\SV {\pi_1} \circ \pi_1 \circ \QLV {s_\ADN^\bullet} \circ
  m^\natural = \SV {\pi_1} \circ \SV {s_\ADN^\bullet} \circ \pi_1
  \circ m^\natural$ (by naturality of $\pi_1$)
  $= \SV {s_\DN^\bullet} \circ m^\sharp \circ \SV {\pi_1} \circ \pi_1$
  (by definition of $s_\ADN^\bullet$, see Proposition~\ref{prop:rADP})
  $= \SV {s_\DN^\bullet} \circ m^\sharp \circ \SV {\pi_1} \circ
  \pi_1$.  Therefore
  $\pi_1 \circ f = r_\DN \circ \mu^\Smyth_{\Val_\bullet X} \circ \SV
  {s_\DN^\bullet} \circ m^\sharp \circ \SV {\pi_1} \circ \pi_1$.  The
  composition
  $r_\DN \circ \mu^\Smyth_{\Val_\bullet X} \circ \SV {s_\DN^\bullet}
  \circ m^\sharp$ was elucidated in Lemma~\ref{lemma:Sbar}, item~3, as
  mapping every $Q \in \SV^{cvx} {\Pred_\DN^\bullet X}$ to
  $(h \in \Lform X \mapsto \min_{F \in Q} F (h))$.  Hence
  $\pi_1 \circ f$ maps every convex quasi-lens
  $(Q, C) \in {\QLV}^{cvx} {\Pred_\ADN^\bullet X}$ to
  $(h \in \Lform X \mapsto \min_{F \in \SV {\pi_1} (Q)} F (h)) = (h
  \in \Lform X \mapsto \min_{F \in \upc \pi_1 [Q]} F (h)) = (h \in
  \Lform X \mapsto \min_{F \in \pi_1 [Q]} F (h)) = (h \in \Lform X
  \mapsto \min_{(F^-, F^+) \in Q} F^- (h))$.

  Similarly, and using Lemma~\ref{lemma:Sbar:flat}, item~3,
  $\pi_2 \circ f$ maps every convex quasi-lens
  $(Q, C) \in {\QLV}^{cvx} {\Pred_\ADN^\bullet X}$ to
  $(h \in \Lform X \mapsto \sup_{(F^-, F^+) \in Q} F^+ (h))$.

  4. Standard category theory \cite[Chapter~VI, Section~2,
  Theorem~1]{McLane:cat:math}.  \qed
\end{proof}

\begin{theorem}
  \label{thm:weaklift:nat}
  The weak composite monad associated with
  $\lambda^\natural \colon \Val_\bullet {\QLV} \to \QLV
  {\Val_\bullet}$ on $\Topcat^\natural$ is
  $(\Pred_\ADN^\bullet, \eta^\ADN, \mu^\ADN)$.
\end{theorem}
\begin{proof}
  We take $T \eqdef \Val_\bullet$, $S \eqdef \QLV$,
  $\lambda \eqdef \lambda^\natural$.  The functor $U^T$ takes every
  $\Val_\bullet$-algebra $\alpha \colon \Val_\bullet X \to X$ to $X$
  and every $\Val_\bullet$-algebra morphism
  $f \colon (\alpha \colon \Val_\bullet X \to X) \to (\beta \colon
  \Val_\bullet Y \to Y)$ to the underlying continuous map
  $f \colon X \to Y$.  The functor $F^T$ takes every space $X$ to
  $\mu_X \colon \Val_\bullet {\Val_\bullet X} \to \Val_\bullet X$, and
  every continuous map $f \colon X \to Y$ to $\Val_\bullet f$, seen as
  a morphism of $\Val_\bullet$-algebras.

  Then $U^T \overline S F^T$ maps every space $X$ to the splitting
  $\Pred_\ADN^\bullet X$ of
  $\QLV {\mu_X} \circ {\lambda^\natural_X} \circ \eta_{\HV X}$ we have
  obtained in Lemma~\ref{lemma:idemp:split:nat}.  We also obtain
  $\pi_{\mu_X} \eqdef r_\ADN$, $\iota_{\mu_X} \eqdef s_\ADN^\bullet$.
  
  Given any continuous map $f \colon X \to Y$,
  $U^T \overline S F^T (f)$ is equal to (the map underlying the
  $\Val_\bullet$-algebra) $\overline S (\Val_\bullet f)$, namely
  $\pi_{\mu_Y} \circ \QLV {\Val_\bullet f} \circ \iota_{\mu_X} =
  r_\ADN \circ \QLV {\Val_\bullet f} \circ s_\ADN^\bullet$.  We claim
  that this is equal to $\Pred_\ADN^\bullet (f)$.  It suffices to
  verify that
  $\pi_1 \circ r_\ADN \circ \QLV {\Val_\bullet f} \circ s_\ADN^\bullet
  = \pi_1 \circ \Pred_\ADN^\bullet (f)$, and similarly with $\pi_2$.

  By definition of $r_\ADN$ (Proposition~\ref{prop:rADP}), naturality
  of $\pi_1$ and definition of $s_\ADN^\bullet$
  (Proposition~\ref{prop:rADP}),
  $\pi_1 \circ r_\ADN \circ \QLV {\Val_\bullet f} \circ s_\ADN^\bullet
  = r_\DN \circ \SV {\Val_\bullet f} \circ s_\DN^\bullet \circ \pi_1$,
  namely to
  $\Pred (f) \circ \pi_1 = \pi_1 \circ \Pred_\ADN^\bullet (f)$.
  Similarly with $\pi_2$.

  The unit of the weak composite monad is
  $U^T \eta^{\overline S} F^T \circ \eta^T$.  At object $X$, this maps
  every $x \in X$ to $r_\ADN \circ \eta^\quasi_{\Val_\bullet X}$, by
  Lemma~\ref{lemma:Sbar:nat}, item~2.  In order to show that this is
  equal to $\eta^\ADN_X$, it suffices to show that
  $\pi_1 \circ r_\ADN \circ \eta^\quasi_{\Val_\bullet X} = \pi_1 \circ
  \eta^\ADN_X$ and similarly with $\pi_2$.  This follows from the
  definition of $r_\ADN$ (Proposition~\ref{prop:rADP}), the definition
  of $\eta^\quasi$ (see (\ref{eq:etaq}) in
  Proposition~\ref{prop:QLV:monad}) and the definition of $\eta^\ADN$
  (see (\ref{eq:etaADN}) in Proposition~\ref{prop:fork:monad}).

  The multiplication is
  $U^T \mu^{\overline S} F^T \circ U^T \overline S \epsilon^T
  \overline S F^T$.  We recall what that means.  At any object $X$,
  $\overline S F^T X$ is the $T$-algebra
  $\overline S\mu_X \colon T \overline S T X \to \overline S T X$.
  Then $\epsilon^T$ evaluated at this $T$-algebra is the bottom arrow
  in the following diagram:
  \[
    \xymatrix@C+10pt{
      TT\overline S T X
      \ar[d]_{\mu_{\overline S T X}}
      \ar[r]^{T \epsilon^T_{\overline S F^T X}}
      & T \overline S T X
      \ar[d]^{\overline S \mu_X}
      \\
      T \overline S T X
      \ar[r]_{\epsilon^T_{\overline S F^T X}}
      & \overline S T X
    }
  \]
  from the $T$-algebra $\alpha \eqdef \mu_{\overline S T X}$ that is
  the leftmost vertical arrow to the $T$-algebra
  $\beta \eqdef \overline S \mu_X$ on the right.  We apply $\overline
  S$ to the morphism $\epsilon^T_{\overline S F^T X}$ (i.e., we now
  apply $\overline S$ to a morphism of $T$-algebras, not to a
  $T$-algebra, as we did before), and we obtain the composition:
  \begin{equation}
    \label{eq:mu:part:nat}
    \xymatrix{
      \overline S T \overline S T X
      \ar[r]^{\iota_\alpha}
      &
      S T \overline S T X
      \ar[r]^{S \epsilon^T_{\overline S F^T X}}
      &
      S \overline S T X
      \ar[r]^{\pi_\beta}
      &
      \overline S \overline S T X.
    }
  \end{equation}
  Applying $U^T$ to the latter, we obtain the same composition, this
  time seen as a morphism in the base category, instead of as a
  morphism in the category of $T$-algebras.  We finally compose it
  with $U^T \mu^{\overline S}_{F^T X}$, which is
  $\mu^{\overline S}_{\mu_X} \colon \overline S \overline S T X \to
  \overline S T X$.

  By Lemma~\ref{lemma:Sbar:nat}, item~4,
  $\epsilon^T_{\overline S F^T X} \colon \alpha \to \beta$ is $\beta$
  itself, where $\beta = \overline S \mu_X = \alpha_\ADN^\bullet$
  (Lemma~\ref{lemma:Sbar:nat}, item~1) and
  $\alpha = \mu_{\overline S T X} = \mu_{\Pred_\ADN^\bullet X}$: we
  have $\overline S T X = \Pred_\ADN^\bullet X$, using
  Lemma~\ref{lemma:idemp:split:nat}), and
  $\iota_\alpha = s_\ADN^\bullet$.  Using
  Lemma~\ref{lemma:idemp:P:nat},
  $\pi_\beta \colon \QLV {\Pred_\ADN^\bullet X} \to {\QLV}^{cvx}
  {\Pred_\ADN^\bullet X}$ is the corestriction of
  (\ref{eq:idemp:alpha:nat}) (call it $f$), which maps every convex
  element of $\QLV {\Pred_\ADN^\bullet X}$ to itself.

  Hence the multiplication of the weak composite monad, evaluated at
  $X$, is the composition of (\ref{eq:mu:part:nat}) with
  $\mu^{\overline S}_{\mu_X}$ (given in Lemma~\ref{lemma:Sbar:nat},
  item~3), namely:
  \[
    \xymatrix{
      \Pred_\ADN^\bullet {\Pred_\ADN^\bullet X}
      \ar[r]^{s_\ADN^\bullet}
      & \QLV {\Val_\bullet {\Pred_\ADN^\bullet X}}
      \ar[r]^{\QLV {\alpha_\ADN^\bullet}}
      & \QLV {\Pred_\ADN^\bullet X}
      \ar[r]^f
      & {\QLV}^{cvx} {\Pred_\ADN^\bullet X}
      \ar[r]^{\mu^{\overline S}_{\mu_X}}
      & \Pred_\ADN^\bullet X
    }
  \]
  Let us call $g$ for that composition.  Using the explicit expression
  for $\mu^{\overline S}_{\mu_X}$, and the appropriate naturality
  equations, we recognize $\pi_1 \circ g$ as being equal to the
  composition $(\ref{eq:Sbarmu}) \circ \SV {\pi_1} \circ \pi_1$;
  (\ref{eq:Sbarmu}) is the expression of the weak composite monad we
  obtained from $\lambda^\sharp$ in the proof of
  Theorem~\ref{thm:weaklift}, which we elucidated as mapping every
  $\mathcal F \in \Pred_\DN^\bullet {\Pred_\DN^\bullet X}$ to
  $(h \in \Lform X \mapsto \mathcal F (F \in \Pred_\DN^\bullet X
  \mapsto F (h))$.

  Similarly, $\pi_2 \circ g$ is equal to
  $(\ref{eq:Sbarmu:flat}) \circ \HV {\pi_2} \circ \pi_2$, where
  (\ref{eq:Sbarmu:flat}) was elucidated in the proof of
  Theorem~\ref{thm:weaklift:flat} as mapping every
  $\mathcal F \in \Pred_\AN^\bullet {\Pred_\AN^\bullet X}$ to
  $(h \in \Lform X \mapsto \mathcal F (F \in \Pred_\AN^\bullet X
  \mapsto F (h))$.  Hence $g$ maps every
  $(\mathcal F^-, \mathcal F^+) \in \Pred_\ADN^\bullet
  {\Pred_\ADN^\bullet X}$ to the pair of maps
  $(h \in \Lform X \mapsto \HV {\pi_1} (\mathcal F^-) (F \in
  \Pred_\DN^\bullet X \mapsto F (h))$ and
  $(h \in \Lform X \mapsto \HV {\pi_2} (\mathcal F^+) (F \in
  \Pred_\DN^\bullet X \mapsto F (h))$, namely to
  $\mu^\ADN_X (\mathcal F^+, \mathcal F^-)$ (see
  Proposition~\ref{prop:fork:monad}).  \qed
\end{proof}

\begin{remark}
  \label{rem:weaklift:nat}
  The isomorphic weak distributive law of Remark~\ref{rem:lambda:nat}
  on the category of stably compact spaces must induce the same weak
  composite monad, namely
  $(\Pred_\ADN^\bullet, \eta^\ADN, \mu^\ADN)$---or rather, its
  restriction to the category of stably compact spaces.
\end{remark}

\begin{remark}
  \label{rem:Vietoris}
  A space is $T_1$ if and only if its specialization preordering is
  the equality relation.  Every Hausdorff space is $T_1$.  On a $T_1$
  space $X$, a lens is the same thing as a non-empty compact subset.
  Hence $\Plotkinn X$ is the usual hyperspace of non-empty compact
  subsets, with the (usual) Vietoris topology, and the $\Plotkinn$
  monad specializes to what has sometimes been called the Vietoris
  monad on the category $\mathbf{KHaus}$ of compact Hausdorff spaces
  (although Vietoris had no knowledge of monads, and while he defined
  the Vietoris topology on the space of closed subsets of a space
  \cite{Vietoris:hyper}, spaces of subsets were further explored by
  Hausdorff \cite[\S 28]{Hausdorff:Mengen} then by Michael
  \cite{Michael:hyper}, among others).  In other words, relying on
  Remark~\ref{rem:Plotkinn}, removing the now useless signs $\upc$ and
  simplifying, $\Plotkinn f \colon \Plotkinn X \to \Plotkinn Y$ maps
  every non-empty compact subset $L$ to its image $f [L]$ under $f$,
  the unit $\eta^{\Plotkin}_X$ maps every $x \in X$ to $\{x\}$, and
  the multiplication $\mu^{\Plotkin}_X$ maps every
  $\mathcal L \in \Plotkinn {\Plotkinn X}$ to $\bigcup \mathcal L$.
  
  Specializing Remark~\ref{rem:weaklift:nat} to the appropriate
  subcategory, we obtain a weak distributive law of the Vietoris monad
  over $\Val_\bullet$ on $\mathbf{KHaus}$, and its associated weak
  composite monad is the appropriate restriction of
  $(\Pred_\ADN^\bullet, \eta^\ADN, \mu^\ADN)$, by
  Theorem~\ref{thm:weaklift:nat}.
\end{remark}

\begin{remark}
  \label{rem:PhiPsi:lens:T1}
  We refine Remark~\ref{rem:PhiPsi:lens} in the case of regular
  spaces.  A space is \emph{regular} if and only if for every point
  $x$, every open neighborhood $U$ of $x$ contains a closed
  neighborhood $C$ of $x$; equivalently, if every open set is the
  union of the directed family of interiors $\interior C$ of closed
  subsets of $U$.  For every regular space $X$, $\Val_1 X$ is
  Hausdorff: given any two distinct elements
  $\mu, \nu \in \Val_\bullet$, there must be an open subset $U$ of $X$
  such that $\mu (U) \neq \nu (U)$, say $\mu (U) < \nu (U)$.  Let $r$
  be such that $\mu (U) < r < \nu (U)$.  Since
  $r < \nu (U) = \dsup_C \nu (\interior C)$, where $C$ ranges over the
  closed subsets of $U$, we have $r < \nu (\interior C)$ for some
  closed subset $C$ of $U$.  Then $\nu$ is in $[\interior C > r]$,
  $\mu$ is in $[X \diff C > 1-r]$ (since
  $1 = \mu (X) = \mu ((X \diff C) \cup U) \leq \mu (X \diff C) + \mu
  (U) < \mu (X \diff C) + r$, hence $\mu (X \diff C) > 1-r$), and
  $[\interior C > r]$ and $[X \diff C > 1-r]$ are disjoint; for the
  latter, any probability valuation in the intersection must map the
  disjoint union $\interior C \cup (X \diff C)$ to a number strictly
  larger than $r + (1-r)=1$, which is impossible.  (One could show
  that $\Val_1 X$ is in fact regular Hausdorff, namely $T_3$, but we
  will not need that.)

  Every compact Hausdorff space $X$ is regular, so the specialization
  ordering on $\Val_1 X$ is equality.  Specializing
  Remark~\ref{rem:PhiPsi:lens} to this case, the weak distributive law
  $\lambda^\natural$ on $\mathbf{KHaus}$ maps every
  $\mu \eqdef \sum_{i=1}^n a_i \delta_{E_i}$, where $n \geq 1$, and
  each set $E_i$ is finite and non-empty, to
  $\lambda^\natural_X (\mu) = \upc conv (\mathcal E) \cap cl (conv
  (\mathcal E))$ (where
  $\mathcal E \eqdef \{\sum_{i=1}^n a_i \delta_{x_i} \mid x_1 \in E_1,
  \cdots, x_n \in E_n\}$)
  $= conv (\mathcal E) \cap cl (conv (\mathcal E)) = conv (\mathcal
  E)$.  This is equivalent to the formula given by Goy and Petri\c san
  for their weak distributive law of the powerset monad over the
  finite distribution monad \cite[Lemma~3.1]{goypetr-dp}.  (Note that
  their powerset includes the empty set, while our monad is one of
  non-empty compact sets, but that seems to be a minor difference
  here.)
\end{remark}

\section{Radon measures}
\label{sec:radon-measures}

We finish with a discussion of Radon measures.  A \emph{Radon measure}
on a Hausdorff topological space $X$ is a Borel measure $\mu$ such
that for every Borel subset $B$ of $X$, for every $\epsilon > 0$,
there is a compact set $K_\epsilon$ included in $B$ such that
$\mu (B \diff K_\epsilon) \leq \epsilon$
\cite[Definition~7.1.1]{Bogachev:mes:II}.  One needs to take care to
modify this definition slightly on non-Hausdorff spaces
\cite[Definition~7.2]{KL:measureext}, if only because $K_\epsilon$ may
fail to be Borel: a \emph{Radon measure} $\mu$ on a coherent sober
space $X$ is a Borel measure on a $\sigma$-algebra that contains the
patch topology on $X$, whose restriction to $\Open {(X^\patch)}$ is
locally finite (every point has a patch-open neighborhood of finite
$\mu$-measure) and \emph{inner regular}, meaning that
$\mu (B) = \dsup_{K \in \mathcal C X, K \subseteq B} \mu (K)$ for
every Borel subset $B$; the collection $\mathcal C X$ consists of the
intersections of non-empty families of finite unions of lenses.  (See
\ref{sec:KL} for Keimel and Lawson's definition, and why we changed it
to the current one.)
The \emph{patch topology} on $X$ is the coarsest topology that
contains the original topology on $X$, as well as the complements of
all compact saturated subsets of $X$.  The patch topology coincides
with the original topology on $X$ if $X$ is Hausdorff, since in that
case every compact set is closed.

In the following, a \emph{well-filtered} space $X$ is a topological
space such that for every filtered family ${(Q_i)}_{i \in I}$ of
compact saturated subsets, for every open subset $U$ of $Z$, if
$\fcap_{i \in I} Q_i \subseteq U$ then $Q_i \subseteq U$ for some
$i \in I$.  A family ${(Q_i)}_{i \in I}$ is \emph{filtered} if and
only if it is non-empty, and for any two elements $Q_i$ and $Q_{i'}$,
there is an element $Q_{i''}$ included in both.  In a well-filtered
space $X$, for every filtered family ${(Q_i)}_{i \in I}$,
$\fcap_{i \in I} Q_i$ is compact saturated
\cite[Proposition~8.3.6]{JGL-topology}.  Every sober space is
well-filtered \cite[Proposition 8.3.5]{JGL-topology}.  
\begin{lemma}
  \label{lemma:lens:compact}
  For every coherent well-filtered space $X$, every element of
  $\mathcal C X$ is compact in the patch topology, hence in the
  original topology on $X$.
\end{lemma}
In the sequel, we would only need to show that every element of
$\mathcal C X$ is compact in $X$, not in $X^\patch$.  But
Lemma~\ref{lemma:lens:compact} fixes a flaw.  Let me explain.

Lemma~7.1~(5) of \cite{KL:measureext} states that, in a coherent sober
space $X$, every element of $\mathcal C X$ is compact in the patch
topology, and the proof really only uses well-filteredness, not
sobriety.  However, the result is obtained from the statement that
every patch-closed subset of $X$ is the union of an element of
$\mathcal C X$ with a closed subset of $X$
\cite[Lemma~7.1~(4)]{KL:measureext}, and that is wrong.  For a
counterexample (to item (4), not (5), of that lemma), consider $\real$
with its Scott topology for $X$.  Then $X^\patch$ is $\real$ with its
usual metric topology, while it is an easy exercise to show that the
elements of $\mathcal C X$ are the compact subsets of $\real$ (in
accordance with Lemma~\ref{lemma:lens:compact}; the easiest elementary
proof consists in verifying that every element of $\mathcal C X$ is
closed and bounded in $\real$).  Then any union of an element of
$\mathcal C X$ with a closed subset of $X$ must be bounded from above.
Hence, for example, $\mathbb Z$, which is closed in
$X^\patch = \real$, is not of that form.  The fixed proof below, which
is in the spirit of the original proof of
\cite[Lemma~7.1~(5)]{KL:measureext}, was communicated to me by Jimmie
Lawson on August 22nd, 2024.

\begin{proof}
  Just as in \cite[Lemma~7.1]{KL:measureext}, let the \emph{lens
    topology} be the topology on $X$ whose closed subsets are the
  intersections of finite unions of lenses.  We will write $X^\ell$
  for $X$ with the lens topology.  The closed subsets of $X^\ell$ are
  the elements of $\mathcal C X$, plus possibly $X$ itself (the empty
  intersection) if $X$ is not compact.  It is clear that the lens
  topology is coarser than the patch topology.  It is a consequence of
  Alexander's subbase lemma that $X^\ell$ is compact (see
  \cite[Lemma~7.1~(2)]{KL:measureext}).

  We now claim that: $(*)$ for every $K \in \mathcal C X$, the lens
  and the patch topologies on $X$ induce the same subspace topology on
  $K$.  Since the lens topology is coarser than the patch topology on
  $X$, the same inclusion holds for their subspace topologies on $K$.
  Conversely, we must show that given any patch-closed subset $D$ of
  $X$, $D \cap K$ is closed in $K$ seen as a subspace of $X^\ell$.
  Since patch-closed subsets are intersections of finite unions of
  closed subsets and of compact saturated subsets of $X$, it suffices
  to show this when $D$ is either closed or compact saturated in $X$.

  Let us write $K$ as
  $\bigcap_{i \in I} \bigcup_{j=1}^{n_i} (Q_{ij} \cap C_{ij})$ where
  $I \neq \emptyset$, each $Q_{ij}$ is compact saturated and each
  $C_{ij}$ is closed in $X$.
  
  If $D$ is a closed subset $C$ of $X$, then
  $C \cap K = \bigcap_{i \in I} \bigcup_{j=1}^{n_i} (Q_{ij} \cap C
  \cap C_{ij})$.  Every set $Q_{ij} \cap C \cap C_{ij}$ is a lens (or
  empty), so $C \cap K$ is in $\mathcal C X$, hence closed in
  $X^\ell$.  Then $C \cap K$ is also equal to the intersection of
  itself with $K$, hence is closed in $K$ as a subspace of $X^\ell$.

  If $D$ is a compact saturated subset $Q$ of $X$, then similarly
  $Q \cap K = \bigcap_{i \in I} \bigcup_{j=1}^{n_i} (Q \cap Q_{ij}
  \cap C_{ij})$.  Since $X$ is coherent, each set $Q \cap Q_{ij}$ is
  compact saturated in $X$, so $Q \cap Q_{ij} \cap C_{ij}$ is a lens
  (or is empty).  Therefore $Q \cap K$ is in $\mathcal C X$, hence is
  closed in $X^\ell$.  Then $Q \cap K$ is also equal to the
  intersection of that set $Q \cap K$ with $K$, hence is closed in $K$
  as a subspace of $X^\ell$.  This ends the proof of claim $(*)$.

  Using $(*)$, we can now show that every $K \in \mathcal C X$ is
  compact in $X^\patch$.  $K$ is closed in $X^\ell$, hence it is also
  closed in $K$, seen as a subspace of $X^\ell$.  Since $X^\ell$ is
  compact, $K$ is compact in $X^\ell$, hence also in $K$, seen as a
  subspace of $X^\ell$.  But, by $(*)$, the latter is also equal to
  $K$, see as a subspace of $X^\patch$.  Any subset of a subspace is
  compact in the surrounding space, so $K$ is compact in $X^\patch$.

  Finally, every compact subset of $X^\patch$ is compact in $X$: every
  cover by open subsets from $X$ is also an open cover by open subsets
  from $X^\patch$.  \qed
\end{proof}

Given any topological space $X$, the collection $\blacksquare Q$ of
open neighborhoods of an arbitrary compact saturated subset $Q$ is a
Scott-open subset of $\Open X$.  It follows that any union
$\bigcup_{i \in I} \blacksquare {Q_i}$, where each $Q_i$ is compact
saturated in $X$, is Scott-open in $\Open X$.  A space is
\emph{consonant} if and only if all the Scott-open subsets of
$\Open X$ are of this form.  The notion arises from
\cite{DGL:consonant}, where it was proved that every regular
\v{C}ech-complete space is consonant; every locally compact space is
consonant, too, as well as every LCS-complete space \cite[Proposition
12.1]{dBGLJL:LCScomplete}.
\begin{proposition}
  \label{prop:V=Radon}
  For $\bullet$ equal to ``$\leq 1$'', resp.\ ``$1$'', for every
  coherent sober space $X$, let $\mathbf R_\bullet X$ be the set of
  Radon subprobability (resp.\ probability) measures on $X$, with the
  \emph{weak topology}, defined by subbasic open sets
  $\{\mu \in \mathbf R_\bullet X \mid \mu (U) > r\}$, $U \in \Open X$,
  $r \in \Rp$.

  If $X$ is consonant, weakly Hausdorff, coherent and sober (e.g., if
  $X$ is stably locally compact), then restriction to $\Open X$
  defines a homeomorphism between $\mathbf R_\bullet X$ and
  $\Val_\bullet X$.
\end{proposition}
\begin{proof}
%
  Every Radon measure $\mu$ on the $\sigma$-algebra generated by the
  patch topology on a coherent sober space $X$ has a restriction $\nu$
  to $\Open X$ that is clearly strict and modular.  We claim that it
  is \emph{tight}, namely that for every open subset $U$, for every
  $r < \nu (U)$, there is a compact saturated subset $Q$ of $U$ such
  that $r < \nu (V)$ for every open neighborhood of $V$.  Let
  $U \in \Open X$ and let $r < \nu (U)$; equivalently, $r < \mu (U)$.
  By definition of Radon measures, there is a proper closed subset $K$
  of $X$ in the lens topology such that $K \subseteq U$ and
  $r < \mu (K)$.  By Lemma~\ref{lemma:lens:compact}, $K$ is compact in
  $X$, so $Q \eqdef \upc K$ is compact saturated in $X$.  Since $U$ is
  upwards-closed, $Q \subseteq U$, and we have
  $r < \mu (K) \leq \mu (Q) \leq \mu (V) = \nu (V)$ for every open
  neighborhood $V$ of $Q$.  We now use the fact that every tight map
  from $\Open X$ to $\creal$ is Scott-continuous
  \cite[Lemma~6.2]{JGL:kolmogorov}, so $\nu$ is a continuous
  valuation.
  
  In the converse direction, by \cite[Theorem~7.3]{KL:measureext},
  every locally finite, tight valuation $\nu$ on a weakly Hausdorff
  coherent sober space $X$ extends to a unique Radon measure on the
  $\sigma$-algebra generated by the patch topology.  In this sentence,
  that $\nu$ is locally finite means that every point $x$ has an open
  neighborhood $U$ such that $\nu (U) < \infty$; this is certainly the
  case for all (sub)probability valuations.  Every continuous
  valuation $\nu$ on a consonant space is tight
  \cite[Lemma~6.2]{JGL:kolmogorov}.  Hence every (sub)probability
  valuation on a consonant, weakly Hausdorff, coherent sober space
  extends to a unique Radon measure on the $\sigma$-algebra generated
  by the patch topology.

  In conclusion, restriction to $\Open X$ defines a bijection of
  $\mathbf R_\bullet X$ onto $\Val_\bullet X$.  It is clear that it is
  a homeomorphism.  \qed
\end{proof}

It follows that there is a monad of Radon (sub)probability measures
$\mathbf R_\bullet$ on the category of stably compact spaces, which is
isomorphic to the monad $\Val_\bullet$ of (sub)probability valuations.
(We recall that $\Val_\bullet$ preserves stable compactness;
otherwise, the previous sentence makes no sense.)

Remark~\ref{rem:weaklift:nat} then entails the following.
\begin{proposition}
  \label{prop:weaklift:nat:radon}
  Let $\bullet$ be ``$\leq 1$'' or ``$1$''.  There is a weak
  distributive law of $\QLV$ (equivalently, $\Plotkinn$) over
  $\mathbf R_\bullet$ (equivalently, $\Val_\bullet$) on the category
  of stably compact spaces, and it induces the composite monad
  $(\Pred_\ADN^\bullet, \eta^\ADN, \mu^\ADN)$.
\end{proposition}

Similarly, Remark~\ref{rem:Vietoris} entails the following.
\begin{proposition}
  \label{prop:Vietoris}
  Let $\bullet$ be ``$\leq 1$'' or ``$1$''.  There is a weak
  distributive law of the Vietoris monad over $\mathbf R_\bullet$
  (equivalently, $\Val_\bullet$) on $\mathbf{KHaus}$, and it induces
  the composite monad $(\Pred_\ADN^\bullet, \eta^\ADN, \mu^\ADN)$.
\end{proposition}

\section*{Competing interests}

The author declares none.

\section*{Acknowledgements}

The elegant proof of Lemma~\ref{lemma:lens:compact} is due to Jimmie
Lawson, and replaces a more complex argument of mine.  Thanks, Jimmie!
Thanks to Richard Garner, too, who kindly informed me that Gabi B\"ohm
was the inventor of weak distributive laws, not him.  Alexandre Goy
detected a mistake in Remark~\ref{rem:PhiPsi:lens:T1}, which came from
similar mistakes in Remarks~\ref{rem:Phi}, \ref{rem:Psi},
\ref{rem:PhiPsi} and~\ref{rem:PhiPsi:lens}.  He also mentioned Quentin
Aristote's work to me \cite{Aristote:weak:distr}.  Quentin has
independently (and before me) obtained weak distributive laws of the
Vietoris monad over $\mathbf R$ on stably compact spaces.  He also
showed that there is no \emph{monotone} weak distributive law of the
same on compact Hausdorff spaces.  Hence---unless one of us made a
mistake---the weak distributive law of
Section~\ref{sec:radon-measures} on $\mathbf{KHaus}$ is not monotone.
We have started discussing the matter.

\bibliographystyle{elsarticle-harv}
\ifarxiv

\else
\bibliography{projlim}
\fi

\appendix

\section{Uniqueness of Radon measures}
\label{sec:KL}

There is a gap in the proof of \cite[Theorem~7.3]{KL:measureext}.  It
is claimed that, for a weakly Hausdorff coherent sober space $X$,
every tight, locally finite valuation $\nu$ on $X$ extends to a unique
Radon measure on a $\sigma$-algebra $\Sigma$ containing the patch
topology of $X$.  While existence---the difficult part---holds,
uniqueness is dealt rather briefly in the final sentence of the proof
(``Clearly a Radon measure is uniquely determined by its values on
$\mathcal C$''; the set $\mathcal C$ is what we call $\mathcal C X$).
But that only says that the extension is determined uniquely from its
values on sets in $\mathcal C X$, and no reason is given as to why the
latter values would be determined uniquely from $\nu$.

In trying to find a fix, we needed to change the definition of Radon
measures slightly.  We define Radon measures on coherent sober spaces
as measures on a $\sigma$-algebra $\Sigma$ containing the patch
topology, whose restrictions to $\Open {(X^\patch)}$ are locally
finite, and which are inner regular.  Keimel and Lawson's definition
\cite[Definition~7.2]{KL:measureext} is a measure $\mu$ on $\Sigma$
that is inner regular and satisfying the following property, instead
of local finiteness:
\begin{quote}
  (Fin) $\mu (K) < \infty$ for every $K \in \mathcal C X$.
\end{quote}
Every measure $\mu$ whose restriction to $\Open {(X^\patch)}$ is
locally finite satisfies (Fin): for each $x \in K$, find a patch-open
neighborhood $O_x$ of $x$ such that $\mu (O_x) < \infty$; finitely
many cover $K$, since $K$ is patch-compact
(Lemma~\ref{lemma:lens:compact}), and their union, which contains $K$,
has finite $\mu$-measure.  The converse fails, even for inner regular
measures.  Consider \emph{Appert space} \cite[Chapter~VIII, \S~II,
p.84]{Appert:abs}, see also
\cite[Counterexample~98]{SS:countexamples}.  This is
$\nat \diff \{0\}$ with the following topology: its open subsets are
all the subsets that do not contain $1$, plus all the subsets $U$
containing $1$ whose asymptotic density
$\liminf_{n \to \infty} \mathrm{card}\; (U \cap \{1, \cdots, n\})/n$
equals $1$.  This is a Hausdorff, hence certainly weakly Hausdorff,
coherent and sober space, whose compact subsets are exactly the finite
subsets of $\nat$.  Then the counting measure, which maps every subset
of $\nat$ to its cardinality, is inner regular and satisfies property
(Fin), but is not locally finite since every open neighborhood of $1$
is infinite.

Now, the extension of $\nu$ built by Keimel and Lawson---call it
$\mu$---does not just satisfy property (Fin), but satisfies the
stronger property of being locally finite on $\Open {(X^\patch)}$.
Indeed, for every $x \in X$, since $\nu$ is locally finite, there is
an open neighborhood $U$ of $x$ in $X$ such that $\nu (U) < \infty$;
then $U$ is also patch-open, and $\nu (U) = \mu (U)$.


We now claim that there can be only one Radon measure $\mu$ (in our
sense) on the patch space that extends $\nu$, provided that $X$ is
weakly Hausdorff, coherent, and well-filtered.  As Keimel and Lawson
argue, the value of $\mu$ on Borel sets is entirely determined by its
values on elements $K$ of $\mathcal C X$, by inner regularity.  In
proving that local finiteness implies (Fin), we have actually shown a
bit more: there is a patch-open neighborhood $O$ of $K$ such that
$\mu (O) < \infty$.  Then $\mu (K) = \mu (O) - \mu (O \diff K)$; the
difference makes sense because
$\mu (O \diff K) \leq \mu (O) < \infty$.  We use inner regularity a
second time, and we express $\mu (O \diff K)$ as
$\dsup_{K'} \mu (K')$, where $K'$ ranges over the elements of
$\mathcal C X$ included in $O \diff K$.  Then
$\mu (K) = \finf_{K'} (\mu (O) - \mu (K')) = \finf_{K'} \mu (O \diff
K')$.  Each set $K'$ is patch-closed, so $O \diff K'$ is patch-open.
Hence $\mu (K)$ is entirely determined by the values of $\mu$ on the
patch-open sets.

We argue that the restriction $\underline{\smash\mu}$ of $\mu$ to the
patch-open subsets of $X$ is a continuous valuation on $X^\patch$.  It
is clearly strict and modular, and Scott-continuity is proved as
follows.  Let ${(O_i)}_{i \in I}$ be a directed family of patch-open
subsets of $X$, and let $O$ be its union.  Then
$\underline{\smash\mu} (O) = \mu (O) = \dsup_{K \subseteq O} \mu (K)$,
where $K$ (implicitly) ranges over $\mathcal C X$.  Since $K$ is
patch-compact (Lemma~\ref{lemma:lens:compact}),
$K \subseteq O = \dcup_{i \in I} O_i$ is equivalent to the existence
of an $i \in I$ such that $K \subseteq O_i$.  Therefore
$\underline{\smash\mu} (O) = \dsup_{i \in I, K \subseteq O_i} \mu (K)
= \dsup_{i \in I} \mu (O_i) = \dsup_{i \in I} \underline{\smash\mu}
(O_i)$.  In the previous paragraph, we have argued that $\mu$ is
determined uniquely as a function of $\underline{\smash\mu}$, and we
will show that $\underline{\smash\mu}$ is determined uniquely from
$\nu$.

The family $\mathcal F_\nu$ of open subsets $U$ (in the original
topology on $X$) such that $\nu (U) < \infty$ is directed, because it
contains the empty set and is closed under binary unions, using
modularity.  The fact that $\nu$ is locally finite means that
$\bigcup \mathcal F_\nu = X$.

Every patch-open set is a union of sets of the form $U \diff Q$ with
$U$ open and $Q$ compact saturated in $X$ (with the original
topology).  Since $\nu$ is locally finite, we can write $U \diff Q$ as
the union of the sets $U \cap U_0 \diff Q$, $U_0 \in \mathcal F_\nu$.
We note that $U \cap U_0$ is itself in $\mathcal F_\nu$.  Hence every
patch-open set is also a union of sets from the collection
$\mathcal B$ of sets of the form $U \diff Q$ with
$U \in \mathcal F_\nu$ (not just $U \in \Open X$) and $Q$ compact
saturated in $X$.  Namely, $\mathcal B$ forms a subbase of the patch
topology.  Additionally, $\mathcal B$ is closed under finite
intersections, and in particular forms a base.

With these observations, we claim that $\underline{\smash\mu}$ is
uniquely determined by its values on elements of $\mathcal B$.  For
every patch-open set $O$, written as a union $\bigcup_{i \in I} B_i$
of elements of $\mathcal B$, hence as a directed union
$\dcup_{J \in \Pfin (I)} \bigcup_{j \in J} B_j$,
$\underline{\smash\mu} (O)$ is equal to
$\dsup_{J \in \Pfin (I)} \underline{\smash\mu} (\bigcup_{j \in J}
B_j)$, hence is uniquely determined from the values that
$\underline{\smash\mu}$ takes on finite unions $\bigcup_{j \in J} B_j$
of elements of $\mathcal B$.  If
$\underline{\smash\mu} (B_j) = \infty$ for some $j \in J$, then we
must have $\underline{\smash\mu} (\bigcup_{j \in J} B_j) = \infty$,
since $\underline{\smash\mu}$ is monotonic.  Otherwise, we can rewrite
$\underline{\smash\mu} (\bigcup_{j \in J} B_j)$ as
$\sum_{K \subseteq J, K \neq \emptyset} {(-1)}^{|K|+1}
\underline{\smash\mu} (\bigcap_{j \in K} B_j)$ (this is the
inclusion-exclusion formula, a consequence of modularity), and we
recall that $\bigcap_{j \in K} B_j \in \mathcal B$.  Hence in both
cases, the value of $\underline{\smash\mu} (\bigcup_{j \in J} B_j)$ is
uniquely determined by the values that $\underline{\smash\mu}$ takes
on elements of $\mathcal B$.

It remains to see that the value $\underline{\smash\mu} (U \diff Q)$ is entirely
determined by $\nu$ for every set $U \diff Q$ in $\mathcal B$.  We
will show that
$\underline{\smash\mu} (U \diff Q) = \dsup_{V \supseteq Q} (\nu (U) - \nu (U \cap
V))$, where $V$ ranges over $\Open X$.  We note that the difference
$\nu (U) - \nu (U \cap V)$ makes sense because
$\nu (U \cap V) \leq \nu (U) < \infty$.

To this end, it is enough to show that for every $r \in \Rp$ such that
$r < \underline{\smash\mu} (U \diff Q)$, there is an open neighborhood
$V$ of $Q$ such that $r \leq \nu (U) - \nu (U \cap V)$.  This will
show that
$\underline{\smash\mu} (U \diff Q) \leq \dsup_{V \supseteq Q} (\nu (U)
- \nu (U \cap V))$, and the reverse inequality is clear.

Since $\mu$ is inner regular and $\underline{\smash\mu}$ is a restriction of $\mu$
(hence, in particular, $r < \mu (U \diff Q)$) there is an element
$K \in \mathcal C X$ such that $K \subseteq U \diff Q$ and
$r \leq \mu (K)$.  Now $K$ is an intersection
$\bigcap_{i \in I} \bigcup_{j=1}^{n_i} L_{ij}$ where $I$ is non-empty
and each $L_{ij}$ is a lens, and we can therefore rewrite $K$ as a
filtered intersection $\fcap_{J \in \Pfin^* (I)} K_J$, where
$\Pfin^* (I)$ denotes the collection of non-empty finite subsets of
$I$, and $K_J \eqdef \bigcap_{i \in J} \bigcup_{j=1}^{n_i} L_{ij}$.
By distributing intersections of unions, and realizing that any finite
non-empty intersection of lenses is a lens (or the empty set), we see
that $K_J$ is a finite union of lenses.  From $K \subseteq U \diff Q$,
we deduce that $\fcap_{J \in \Pfin^* (I)} K_J \diff (U \diff Q)$ is
empty.  But $K_J \diff (U \diff Q) = (K_J \diff U) \cup (K_J \cap Q)$
is also a finite union of lenses, because $L \diff U$ and $L \cap Q$
are lenses (or empty) for every lens $L$.  (The second claim relies on
the coherence of $X$.)  By Lemma~\ref{lemma:lens:compact}, each set
$K_J \diff (U \diff Q)$ is then patch-compact, and is patch-closed.
Now any filtered family of compact closed sets whose intersection is
empty contains the empty set; this is an easy exercise using the
definition of compactness.  Therefore $K_J \diff (U \diff Q)$ is empty
for some $J \in \Pfin^* (I)$.  In other words,
$K_J \subseteq U \diff Q$.  We have $K \subseteq K_J$, so
$r \leq \mu (K_J)$.  Therefore, replacing $K$ by $K_J$ if necessary,
we may assume that $K$ itself is a finite union of lenses, say
$\bigcup_{i=1}^n (Q_i \cap C_i)$, where each $Q_i$ is compact
saturated and $C_i$ is closed in $X$, in such a way that
$r \leq \mu (K)$ and $K \subseteq U \diff Q$.

From $K \subseteq U \diff Q$, we obtain that
$Q_i \cap C_i \subseteq U \diff Q$ for every $i \in \{1, \cdots, n\}$,
namely that $Q_i \subseteq U \cup U_i$, where $U_i$ is the complement
of $C_i$, and $Q \cap Q_i \subseteq U$.  Since $X$ is weakly
Hausdorff, $Q$ has an open neighborhood $V_i$ and $Q_i$ has an open
neighborhood $W_i$ such that $V_i \cap W_i \subseteq U_i$, for every
$i$.  Then
$Q_i \cap C_i \subseteq W_i \cap C_i \subseteq X \diff V_i$, for every
$i \in \{1, \cdots, n\}$.  We let $V \eqdef \bigcap_{i=1}^n V_i$.
Then $K \subseteq X \diff V$.  We also have $K \subseteq U$, so
$K \subseteq U \diff V$.  From $r \leq \mu (K)$, we deduce that
$r \leq \mu (U \diff V)$.  Now
$\mu (U \diff V) = \mu (U) - \mu (U \cap V)$, an equality that makes
sense because all the sets involved have finite $\nu$-measure,
equivalently finitely $\mu$-measure, since $\mu$ extends $\nu$.  Then
$\mu (U \diff V) = \nu (U) - \nu (U \cap V)$, and this completes the
proof.

\end{document}
